\documentclass[12pt,letterpaper]{article}

\usepackage{renstyles}

\renewcommand{\RED}[1]{{#1}}

\title{Coupling Deep Learning with Full Waveform Inversion}

\author{
Wen Ding\thanks{
		Department of Applied Physics and Applied Mathematics, Columbia University, New York, NY 10027;
		\href{mailto:wd2288@columbia.edu}{wd2288@columbia.edu}
	}
\and
Kui Ren\thanks{
		Department of Applied Physics and Applied Mathematics, Columbia University, New York, NY 10027;
		\href{mailto:kr2002@columbia.edu}{kr2002@columbia.edu}
	}
\and
Lu Zhang\thanks{Department of Computational Applied Mathematics and Operations Research, and Ken Kennedy Institute, Rice University, Houston, TX, 77005, USA.  Email: lz82@rice.edu}  
}

\date{}
\begin{document}

\maketitle



\begin{abstract}
Full waveform inversion (FWI) aims to reconstruct unknown physical coefficients in wave equations using the wavefield data generated from multiple incoming sources. In this work, we propose an offline-online computational strategy for coupling classical least-squares-based computational inversion with modern learning-based approaches for FWI to achieve advantages that can not be achieved with only one of the components. \RED{In brief, we develop an offline learning strategy to construct a robust approximation of the inverse operator through weighted optimization and utilize it to design a new objective function for approximate online inversion with new datasets. The approximate online inversion then serves as a warm start for the true online inversion.} We demonstrate through numerical simulations that our coupling strategy improves the computational efficiency of FWI with reliable offline training on moderate computational resources (in terms of both the size of the training dataset and the computational cost needed).

\end{abstract}






\section{Introduction}
\label{SEC:Intro}

Full waveform inversion (FWI) refers to the process of extracting information on physical parameters of wave equations from data related to the solutions to the wave equations~\cite{AkBiGh-SC02,BoDrMaZa-IP18,BrOpVi-Geophysics10,BuSaZaCh-Geophysics95,ChSa-GJI20,EnFrYa-CMS16,EpAkGhBi-IP08,LiBeLePeTr-GJI22,MeAlBrMeOuVi-Geophysics18,Plessix-GJI06,Pratt-Geophysics99,PrShHi-GJI98,SiPr-Geophysics04,Symes-GP08,TrTaLi-GJI05,ViAsBrMeRiZh-GRS14}. In seismic imaging, this is manifested as the problem of reconstructing the speed distribution of seismic waves in the interior of the Earth from measured wave field data on the Earth surface. \RED{The sources of the measured waves could be either natural, such as earthquakes, or human-induced, including geophysical exploration activities like air guns and seismic vibrators.} We refer interested readers to~\cite{BuGh-Geophysics09,Fichtner-Book11,MoTr-GJI16,ViOp-Geophysics09} and references therein for overviews on the recent development in the field of FWI for geophysical applications. While the term FWI was mainly coined in the seismic imaging community, FWI also has a wide range of applications in other imaging applications, such as in medical ultrasound imaging~\cite{BaTr-JAS20,BeMoKoLa-PMB17,GuCaTaNaWa-DM20,JaLuCo-IP20,LiViPaAn-IEEE22,LuPeTrCo-arXiv21,MaPoLiWaAn-SIAM18,WiMaBoPiKl-SR20}. From the practical point of view, the main difference between geophysical and medical FWI is that the quality of the dataset collected in medical applications, both in terms of the variety of source-detector configurations can be arranged and in terms of the frequency contents of the incident sources, is much richer than that of the geophysical FWI dataset.

For the sake of concreteness, let us consider the simplest model of acoustic wave propagation in a heterogeneous medium $\Omega$ with wave speed field $m(\bx)>0$. The wave field $u$, in time interval $(0, T]$, solves
\begin{equation}\label{EQ:Wave}
	\begin{array}{rcll}
	\dfrac{1}{m^2(\bx)}\dfrac{\partial^2 u}{\partial^2 t} -\Delta u & = & 0, & \mbox{in}\ \ (0, T] \times \Omega\\
	\dfrac{\partial u}{\partial {\bf n}} &=& h(t, \bx), & \mbox{on}\ (0, T]\times \partial\Omega\\
    u(0, \bx)&=&0, & \mbox{in}\ \Omega\\
    \dfrac{\partial u}{\partial t}(0,\bx)&=&0, & \mbox{in}\ \Omega
	\end{array}
\end{equation}
Here, ${\bf n}$ is the unit outward normal vector of the domain boundary at ${\bf x}\in\partial\Omega$. The data that we measure is time traces of the solution to the wave equation~\eqref{EQ:Wave} at a set of detector locations, say $\Gamma\subset\bbR^d$, that is,
\begin{equation}\label{EQ:Data}
	g: = u(t, \bx)|_{(0, T]\times \Gamma}\,.
\end{equation}

The objective of FWI in this setting is to recover the unknown wave speed field $m$ in the wave equation~\eqref{EQ:Wave} from the measured data $u(t, \bx)_{|(0, T]\times \Gamma}$ collected in a multi-source multi-detector configuration. This is a challenging inverse problem that has rich mathematical and computational content.
The main computational strategy, due to the lack of explicit/semi-explicit reconstruction methods, in solving the FWI inverse problem (as well as many other model-based inverse problems), is the classical $L^2$ least-squares formulation where we search for the inverse solution by minimizing the $L^2$ mismatch between model predictions and observed data. To formulate this more precisely, we assume that we collect data from $N_s$ acoustic sources $\{h_s\}_{s=1}^{N_s}$, and we denote by $f(m; h_s)$ the forward model that takes $m$ to the corresponding wave field data $g_s$ ($1\le s\le N_s$). Then the inverse problem of reconstructing $m$ from measured data $g_s^\delta$ aims at solving the following operator equation:
\begin{equation}\label{EQ:Nonl IP}
	\bff(m) = \bg^\delta
\end{equation}
where
\[
	\bff(m):=\begin{pmatrix} f(m; h_1)\\ \vdots \\ f(m; h_s)\\ \vdots \\ f(m; h_{N_s})\end{pmatrix}\ \ \ \mbox{and}\ \ \ 
	\bg^\delta:=\begin{pmatrix} g_1^\delta \\ \vdots \\g_s^\delta \\ \vdots \\ g_{N_s}^\delta \end{pmatrix}\,.
\]
The superscript $\delta$ denotes the fact that the datum $g$ is polluted by measurement noise. The classical $L^2$ least-squares method performs the reconstruction by searching for $m$ that minimizes the mismatch functional (with the possibility of adding a regularization term):
\begin{equation}\label{EQ:Obj Classical}
    \Psi(m):=\dfrac{1}{2}\|\bff(m)-\bg^\delta\|_{[L^2((0, T]\times \Gamma)]^{N_s}}^2\,.
\end{equation}
This is a challenging numerical optimization problem that has been extensively studied in the past three decades. Many novel methods have been developed to address two of the main challenges: (i) the high computational cost needed to reconstruct high-resolution images of $m$, and (ii) the abundance of local minimizers (due to the non-convexity of the least-squares functional) that trap iterative reconstruction algorithms; see for instance~\cite{BuGh-Geophysics09,Fichtner-Book11,SaSy-Book89} for a detailed explanation of those challenges among others.

In recent years, there has been great interest in the FWI community to use deep learning techniques, based on neural networks,  to replace the classical least-squares based inversion methods~\cite{AdArPo-IEEE21,ArAdFaJe-DLAA20,CoDrAmMa-GP20,FaZhLiZhWaSuZh-Geophysics20,FaArJeClBi-EAGE18,HuJiWuCh-Geophysics21,JiGuLiWaLiSh-IEEE21,KaOvPlPeZhAl-Geophysics21,LiYaReXuJiCh-Geophysics21,MaPoRaSe-arXiv21,RoYaLiThWo-IEEE20,SoAl-arXiv21,SuAl-arXiv20,SuDe-Geophysics20,SuInHu-Geophysics21,WuLi-IEEE20,YaMa-Geophysics19,YuMa-arXiv20,ZhGa-IEEE21,ZhLi-IEEE20}. Assume that we are given a set of sampled data
\begin{equation}\label{EQ:Data Training}
	\{\bg_j:=(g_{j1}, \cdots, g_{js}, \cdots, g_{jN_s})^\fT, m_j\}_{j=1}^N\,,
\end{equation}
where $\{m_j\}_{j=1}^N$ are a set of $N$ velocity profiles sampled from a given distribution and $\{\bg_j\}_{j=1}^N$ are the corresponding wave field predictions generated from $N_s$ sources $\{h_s\}_{s=1}^{N_s}$ with the model $\bg=\bff(m)$. Deep learning methods try to train a neural network, denoted by $\bff_{\theta}^{-1}(\bg)$, with $\theta$ denoting the set of parameters (that is, the weight matrices and the bias vectors) of the neural networks, that represents the inverse operator $\bff^{-1}$. A training process based on the $L^2$ loss functional can be formulated as:
\begin{equation*}
	\wh \theta=\argmin_{\theta\in\Theta} \mathfrak L(\theta) \ \ \mbox{with}\ \ \mathfrak L(\theta):=\dfrac{1}{2N} \sum_{j=1}^N \| m_j- \bff_\theta^{-1}(\bg_j)\|_{L^2(\Omega)}^2
\end{equation*}
where $\Theta$ represents the space of parameters of the network, and a regularization term can be added to the loss function $\mathfrak L(\theta)$ to help stabilize the training process. The number of samples $N$ needs to be large enough in order for $\mathfrak L(\theta)$ to be a good approximation to the expectation of the mismatch over the distribution: $\bbE_m[\|m-\bff_\theta^{-1}(\bg(m))\|_{L^2(\Omega)}^2]$. Many other types of loss functions can be used, but we will not dive into this direction. Note that since we know the forward operator $\bff$ and are only interested in learning its inverse operator, the datasets used in the training process are synthetic: for each data point $(\bg_j, m_j)$, $\bg_j$ is constructed by solving the wave equation~\eqref{EQ:Wave} with the given speed field $m_j$ and evaluate~\eqref{EQ:Data}.  

Numerical experiments, such as those documented in~\cite{ArAdFaJe-DLAA20,KaOvPlPeZhAl-Geophysics21,LiYaReXuJiCh-Geophysics21,SoAl-arXiv21,WuLi-IEEE20,YaMa-Geophysics19,YuMa-arXiv20,ZhGa-IEEE21,ZhLi-IEEE20}, showed that, with sufficiently large training datasets, it is possible to train highly accurate inverse operators that can be used to directly map measured wave field data into the velocity field. This, together with the recent success in learning inverse operators for other inverse problems (see, for instance, ~\cite{AdOk-IP17,BaYeZaZh-IP20,BuGaLaPrRaSi-SIAM21,CeJiLiZh-JCP25,FaFaYi-JCP20,RaPeKa-JCP19,SmAzRo-arXiv20} for some examples), has led many to believe, probably overly optimistically, that one can completely replace classical computational inversion with offline deep learning.

Despite the tremendous success in deep learning for FWI, it is still computationally challenging to train a once-for-all inverse machine $\bff_\theta^{-1}$. First, with the aim of reconstructing high-resolution images of the velocity field $m(\bx)$, the size of the neural networks to be constructed as a discrete representation of $\bff_\theta^{-1}$ is prohibitively large. Training such large networks requires the usage of a prohibitively large amount of data, even in relatively simple inverse scattering problems, as numerically demonstrated in~\cite{ZhHaRaBo-JCP23}. Second, it is well known that $\bff: m\mapsto \bg$ is a smoothing operator (between appropriate spaces; seen for instance, ~\cite{Isakov-Book06} and references therein for more precise mathematical characterization of the statement). The inverse operator is therefore de-smoothing. Learning such operators requires the ability to capture precisely high-frequency information in the training data, and this is very hard to do in the training process as deep neural networks tend to capture low-frequency components of the data much more efficiently than the high-frequency components~\cite{RaBaApDrLiHaBeCo-PMLR19,RoJaKaKr-NEURIPS19,XuZhXi-arXiv18}. \RED{In addition to the above, the inverse operator $\bff_\theta^{-1}$, which we learned from model-generated data, often has limited generalization, making it challenging to apply the operator to new measured datasets.}

In this work, we propose an offline-online computational strategy for coupling classical least-squares-based computational inversion with learning-based approaches for FWI to achieve advantages that can not be achieved with only one of the components. Roughly speaking, we utilize an offline-trained approximate inverse of the operator $\bff$ to precondition the online least-squares-based numerical reconstructions. Instead of pursuing high-quality training of highly accurate inverse operators, we train neural networks that only capture the main features in the velocity field. This significantly reduces the requirements for both the dataset size and computational resources needed during the training process, and the trained model is more generalizable to other classes of velocity models. Meanwhile, the offline-trained approximate inverse is sufficient as a nonlinear preconditioner to improve the speed of convergence of the classical least-squares-based FWI numerical reconstruction in the online stage of the inversion. Let us mention that similar ideas have been proposed recently in two works on inverse scattering problems for the Helmholtz equation~\cite{JiKhYa-SIAM24,ZhHaRaBo-JCP23}.  

The rest of the paper is organized as follows. We first describe the proposed coupling strategy in Section~\ref{SEC:Coupling} in the abstract setting. We then present some preliminary understanding of the training and reconstruction stage of the method in Section~\ref{SEC:Understanding}. In Section~\ref{SEC:Implementation}, we discuss the details of implementing the strategy. Extensive numerical simulations are presented in Section~\ref{SEC:Num} to demonstrate the performance of the learning-inversion coupling. Concluding remarks are offered in Section~\ref{SEC:Concl}.

\section{Coupling learning with FWI}
\label{SEC:Coupling}

Our main objective here is to couple the deep learning based image reconstruction approach with the classical least-squares-based image reconstruction method for FWI. More precisely, we utilize the approximate inverse we learned with neural networks to construct a new objective function for least-squares-based FWI reconstruction from measured data.

\subsection{Robust offline learning of main features}

In the offline learning stage, we use deep learning to train an approximate inverse of the operator $\bff$. As we outlined in the previous section, our main argument is that the learning process can only be performed reliably on a small number of dominant features of the velocity field. First, resolving all details of the velocity field requires oversized neural networks that demand an exceedingly large amount of training data, not to mention that such networks are computationally formidable to train reliably. Second, large neural networks or large sizes display serious frequency bias in capturing frequency contents in the training datasets~\cite{XuZhXi-arXiv18}, making it inefficient in fitting high-frequency components of the velocity field. Despite the challenges in resolving high-frequency features, it has been demonstrated in various scenarios that learning the low-frequency components of the velocity profile can be done in a robust manner~\cite{LiGuLiAbZhYaXu-IEEE21,ReNiYaJiCh-IEEE21,SuDe-Geophysics20}. This means that if we take the Fourier representation, the lower Fourier modes of the inverse operator can be learned stably. This good, low-frequency approximate inverse is our main interest in the learning stage (even though an accurate inverse itself would be better if it could be realistically obtained).

Let $\fM$ be the feature map we selected, and $\fm$ the corresponding feature vector, that is,
\begin{equation*}
	\fM: m(\bx)\in\cM \longmapsto \mathfrak m\in\bbM\,,
\end{equation*}
where $\cM\subseteq L^2(\Omega)$ is the class of velocity field that we are interested in and $\bbM$ the space of the feature vectors. Motivated by the analysis of weighted optimization in~\cite{EnReYa-IP20,EnReYa-ICLR22}, we train a network, which we still denote as $\bff_\theta^{-1}:\ \bg\mapsto \fm$, using the synthetic dataset~\eqref{EQ:Data Training}, through the optimization problem
\begin{equation}\label{EQ:Learning Opt}
	\wh \theta=\argmin_{\theta\in\Theta} \cL(\theta)\ \ \mbox{with}\ \ \cL(\theta):=\dfrac{1}{2N} \sum_{j=1}^N \left\| \bmu  \Big(\fm_j-\bff_\theta^{-1}(\bg_j)\Big)\right\|_{\bbM}^2
\end{equation}
where the weight matrix $\bmu$ is selected to weight the loss heavily on the features we are interested in while damping the features that are hard to learn stably.  The selection of the feature vectors as well as the weighting vector $\bmu$ will be discussed in Section~\ref{SEC:Num} in more detail. For the purpose of illustrating the main idea, let us point out that one example is to think of~\eqref{EQ:Learning Opt} as the equivalence of \begin{equation*}
	\wh \theta=\argmin_{\theta\in\Theta} \dfrac{1}{2N} \sum_{j=1}^N \int_{\Omega}\Big(\int_{\Omega} \mu(\bx-\by)\big(m_j(\by)-\bff_\theta^{-1}(\bg_j)(\by)\big)d\by \Big)^2 d\bx
\end{equation*}
in the Fourier domain, i.e. when the features we use are Fourier modes, with $\bmu$ the Fourier transform of the  
kernel $\mu(\bx)$. If we take $\mu$ to be a smoothing kernel, such as a Gaussian kernel, $\bmu$ will decay fast with the increase of the frequency. In such a case, the learning problem~\eqref{EQ:Learning Opt} focuses on the lower Fourier modes of the velocity field $m$.

Weighted optimization schemes of the form~\eqref{EQ:Learning Opt} with weight $\bmu$ to emphasize dominant features in the learning problems have been extensively studied in the learning and inverse problems community; see~\cite{EnReYa-ICLR22} and references therein. When the features we selected are Fourier basis, it has been shown that correct selecting of the weight $\bmu$ in the training scheme can lead to more robust learning results for a class of models $\bff_{\theta}^{-1}$ following certain distributions, sometimes at the expense of learning accuracy, with better generalization capabilities~\cite{EnReYa-ICLR22}. This is the main motivation for us to adopt this strategy for our purpose in this research.

\subsection{Online stage (i): main feature reconstruction}

In the online reconstructions stage, we utilize the approximate inverse we trained to construct a new objective function for FWI image reconstruction from given noisy data $\bg^\delta$. More precisely, instead of solving the model~\eqref{EQ:Nonl IP}, we aim at solving the modified model
\begin{equation}\label{EQ:Nonl IP Precond}
	\wh \bff_{\wh\theta}^{-1} \big(\bff(m)\big) = \wh \bff_{\wh\theta}^{-1}(\bg^\delta)
\end{equation}
where $$\wh \bff_{\wh\theta}^{-1}:={\fM}^{-1}\circ\bff_{\wh\theta}^{-1}:\ \bg\mapsto m$$ is the learned approximate to $\bff^{-1}$ (while $\bff_{\wh\theta}^{-1}:\ \bg\mapsto \fm$ is the learned representation in $\fM$).

The least-squares formulation for the reconstruction problem now takes the form
\begin{equation}\label{EQ:Min}
	\wh m_{(i)}=\argmin_{m\in\cM}\Phi(m)\,,
\end{equation}
with
\begin{equation}\label{EQ:Obj}
	\Phi(m):=\dfrac{1}{2} \|\wh \bff_{\wh \theta}^{-1}\Big(\bff(m)\Big) - \wh \bff_{\wh \theta}^{-1} (\bg^\delta)\|_{L^2(\Omega)}^2+\dfrac{\gamma}{2}\|{\fM}^{-1}(\bmu^{-1}  \fm)\|_{L^2(\Omega)}^2\,.
\end{equation}
The last term in the objective functional is a Tikhonov regularization functional that imposes a smoothness constraint on the target velocity field $m$ ($\fm$ being the feature vector of $m$). This smoothness constraint is selected such that it is consistent with the training process. The natural initial guess for any iterative solution scheme for this minimization problem is $m_0:=\wh\bff_{\wh\theta}^{-1}(\bg^\delta)$.

Let us emphasize that there is a significant difference between the $L^2$ objective function $\Phi(m)$ we introduced in~\eqref{EQ:Obj}, ignoring the regularization term, and the standard $L^2$ objective function $\Psi(m)$ defined in~\eqref{EQ:Obj Classical}. Our objective function $\Phi(m)$ measures the mismatch between the approximations of the predicted velocity field and the true velocity field corresponding to the measured data, while the standard objective function $\Psi(m)$ measures the mismatch between the predicted wave field data with the measured wave field data. In other words, our objective function works on the parameter space (also called the model space in the FWI literature, that is, the space of the velocity field) while the standard objective function is defined on the signal space (that is, the space of wave field signals at the detectors). With reasonably-trained $\wh\bff_{\wh\theta}^{-1}$, the functional $\Phi(m)$ has an advantageous landscape for optimization purposes, as we will demonstrate in the numerical simulations in Section~\ref{SEC:Num}.

\subsection{Online stage (ii): full least-square refinement}

Once we have reconstructed the main features in the online stage (from the new data), we can refine the reconstruction by 
\begin{equation}\label{EQ:Online Stage II}
    \wh{m}_{(ii)} = \argmin_{m\in\cM} \Psi(m)+\dfrac{\gamma}{2}\|m-\wh{m}_{(i)})\|_{L^2(\Omega)}^2,\quad\mbox{with initial guess}\ m_0=\wh m_{(i)}\,.
\end{equation}
where $\Psi(m)$ is the data mismatch defined in~\eqref{EQ:Obj Classical}, and $\wh m_{(i)}$ is the solution from stage (i) of the online reconstruction step defined in~\eqref{EQ:Min}. We use the stage (i) reconstruction as a regularizer as well as the initial guess of this stage of the reconstruction. This can be seen as a version of the warm start strategy in the literature; see, for instance,~\cite{JiKhYa-SIAM24,ZhHaRaBo-JCP23}. While the reconstruction process~\eqref{EQ:Online Stage II} is based on a full least-square problem, its efficiency will be better than a direct least-squares reconstruction since we start with a good approximation $\wh m_{(i)}$.

\subsection{The benefits of the coupling approach}

The offline-online coupling scheme we proposed allowed us to focus on training a robust approximate inverse instead of the exact inverse. This makes the learning process more stable and also requires less computational resources (in terms of the amount of data, the size of the network, and the computational cost for optimization) than training an accurate inverse. Moreover, the sacrifice in accuracy brings better generalizability for the learned approximate inverse. On the computational side, the trained approximate inverse serves as a ``preconditioner" for the inversion process. It can not only provide a good initial guess for the reconstruction but also simplify the landscape of the optimization problem. 

We finish this section with the following remark. In the ideal case when all the operators involved are invertible as they should be, the solution to~\eqref{EQ:Nonl IP Precond} is identical to the solution to~\eqref{EQ:Nonl IP}, assuming that $\bg$ indeed lives in the range of $\bff$. Therefore, our formulation does not change the true solution to the original inverse problem. However, as we will see, the new formulation utilizes the result of learning to facilitate the FWI reconstruction in terms of saving computational costs as well as making the optimization landscape more desirable.

\section{Formal understanding of the coupling}
\label{SEC:Understanding}

We now attempt to gain a more systematic understanding of the coupling strategy. As we have argued in the previous sections, it is computationally challenging to train neural networks that are accurate approximations of the inverse operator and are very generalizable at the same time. However, there is certainly some dominant information in the inverse operator that we could extract with learning, and this is the approximate inverse that we are interested in constructing.

\subsection{Elements of network training}
\label{SUBSEC:Training Error}

Due to the fact that the training data we have are generated from exactly the same operator we are trying to represent with the neural network, the learning process we have is much more under control than those purely data-driven learning problems in applications. Here we highlight a few critical issues in the learning process without getting into the details of the implementation of the learning algorithm.

\paragraph{Sampling training data.} To learn the inverse operator, we need to pay attention to both its input space and its output space. While our focus is to learn the low-frequency component of the inverse operator, we want the training data to include as much high-frequency information as possible to gain generalization capability in the input space. Let $\bK_{\rm out}$ be the frequency range for the network output that we are interested in recovering and $\bK_{\rm in}$ the frequency range of the velocity fields that generated the wavefield data. We construct the training dataset as
\[
    \{m_j(\bx), \bg:=\bff(m_j(\bx)+\wt m_j(\bx))\}_{j=1}^N
\]
where $\wt m_j(\bx)$ are selected such that $\cF(\wt m_j)(\bk)=0$ $\forall \bk\in \bK_{\rm out}$, and $\cF(\wt m_j)(\bk)\neq 0$ $\forall \bk\in \bK_{\rm in}\backslash \bK_{\rm out}$ ($\cF(m)$ denoting the Fourier transform of $m$). In other words, we train the network with input wavefield data whose velocity field has a richer frequency content than the output velocity field. This construction enriches the frequency content of the input data without increasing the computational cost of the training process.

\RED{Standard results on the acoustic wave equation~\eqref{EQ:Wave}, see for instance,~\cite[Theorem 8.1]{Isakov-Book06} and recalled below in~\Cref{PROP:Frechet Diff}, ensure that the map $m \mapsto g$ is a smoothing map under mild assumptions. Moreover, the map is Fr\'echet differentiable when restricted to a more regular velocity space. Therefore, the boundary data produced by the forward operator lies in a quantitatively smoother Sobolev space than the velocity space in which $m$ resides. This reflects the intrinsic smoothing effect of the wave equation on the boundary measurements, even though the degree of smoothing is minimal in some sense. This suggests that a well-trained network approximation should exhibit good interpolation capabilities in applications when the space of the velocity field we are interested in working with is sufficiently smooth.
\begin{proposition}[~\cite{BaSy-CPDE96,DiDoNaPaSi-SIAM02,Isakov-Book06}]\label{PROP:Frechet Diff}
    Let $\Omega$ be a smooth domain, and $m\in [\underline{m}, \overline{m}]$ for some $0<\underline{m}<\overline{m}<+\infty$. Assume that $m \in W^{k,\infty}(\Omega)$ and $h\in W^{k-1/2,\infty}((0, T]\times \partial\Omega)$ ($k\ge 1$). 
Then there is a unique solution $u$ to ~\eqref{EQ:Wave} and it satisfies the following inequality:
\[
\|u(t,\cdot)\|_{\cH^k(\Omega)}+\|\partial_t u(t,\cdot)\|_{\cH^{k-1}(\Omega)}\le C\|h\|_{\cH^{k-1/2}((0, T)\times \partial\Omega}\,.
\]
Assume further that $m\in \cC^4(\bar\Omega)$. Then the map: 
    \[
        \bff(m): \begin{array}{rcl}
        m&\longmapsto& \bg\\[1ex]
        \cC^4(\overline\Omega) &\longmapsto& \cH^{3/2}((0, T)\times \partial\Omega)
        \end{array}
    \]
    is Fr\'echet differentiable at any $m \in\cC^4(\overline\Omega)$ in the direction $\wt m$ such that $m+\wt m \in [\underline{m},\overline{m}]$.
\end{proposition}
The result is standard. We refer interested readers to~\cite{BaSy-CPDE96,DiDoNaPaSi-SIAM02,Isakov-Book06} and references therein for more precise formulations of it in different scenarios. This result also ensures that if we can train a stable network, then the learning quality is guaranteed; see Lemma~\ref{LMMA:Stab} below.}

\RED{\paragraph{Heuristic analysis of training error distribution.} Our main objective of this work is to focus the learning process on the low-frequency content of the output of the inverse operator. We do this with the weighted optimization scheme~\eqref{EQ:Learning Opt} by selecting weight $\bmu$ that penalizes heavily the low-frequency component of the mismatch of true data and the network prediction. The impact of such weighting schemes on the learning results has been analyzed extensively; see~\cite{EnReYa-IP20,EnReYa-ICLR22} and reference therein. We illustrate this in an extremely simplified setting. Let $\bF:=(\bff^{-1})'(m_0)$ be the linearization of $\bff^{-1}$ at $m_0$ for a one-dimensional medium. Assume that the learning loss function $\mathcal L(\theta)$ in~\eqref{EQ:Learning Opt} is minimized to the order of $\eps^2$ in the training process. Then, in the leading order, the trained $\bF$ satisfies
\[
    \bmu  (\bm-\bF\bG)\sim \cO(\eps),
\]
where $\bm=[\bm_1, \cdots, \bm_N]$ is the matrix whose columns are vectors of the Fourier coefficients of the training velocity samples $\{m_j\}_{j=1}^N$, $\bG=[\bg_1, \cdots, \bg_N]$ is a matrix whose columns are vectors of the input data, and $\cO(\eps)$ is a diagonal matrix of size order $\eps$. The trained linearized inverse operator, when applied to a new input data $\bg^\delta$, gives the result (with $\bG^{\fT}$ being the transpose of $\bG$)
\[
    \bF\bg^\delta \sim \big(\bm-\bmu^{-1}  \eps\big)\bG^{\fT}(\bG\bG^\fT)^{-1}\bg^\delta\,.
\]
The nature of $\bmu$ clearly indicates that the relative error in the learned output is larger in the high-frequency Fourier modes of the reconstructed velocity field.}

\subsection{Inversion with accurate training}

As we discussed in the previous section, when the network is trained so that $\wh\bff_{\wh\theta}^{-1}=\bff^{-1}$, the objective function $\Phi(m)$, defined in~\eqref{EQ:Obj}, in the reconstruction stage is a convex functional of $m$. When the learning is not perfect but sufficiently accurate, the functional $\Phi(m)$ still has an advantageous landscape. This is given in the following result.
\RED{\begin{lemma}\label{fact:landscape}
	Let $\wh\bff_{\wh\theta}^{-1}: \bg\in \cG \subseteq [L^2([0, T]\times\Gamma)]^{N_s} \mapsto m\in  \cM \subseteq L^2(\Omega)$ be an approximation to $\bff^{-1}$ with Fr\'echet derivative at $\bg$ given as $d\wh  \bff_{\wh\theta}^{-1}[\bg]$. Assume that
	\begin{equation}\label{EQ:Ass Learning 1}
    \sup\limits_{m \in\cM} \Vert \wh{\bff}_{\wh{\theta}}^{-1}\big(\bff(m)\big) - m\Vert_{L^2(\Omega)} \leq \epsilon
    \end{equation}
    and
    \begin{equation}\label{EQ:Ass Learning 2}
    A := 1+\sup\limits_{\bg \in\cG} \Vert d\wh  \bff_{\wh\theta}^{-1}[\bg]\Vert_{\cL([L^2([0,T]\times\Gamma)]^{N_s}; L^2(\Omega))}< +\infty 
    \end{equation}
	for some $\epsilon>0$ and $\bg^{\bdelta}=\bff(m)+\bdelta$ for some $\bdelta$ with $\|\bdelta\|_{[L^{2}([0, T]\times\Gamma)]^{N_s}}$ sufficiently small. Then we have that
	\begin{align}\label{EQ:Convexity}
	\left\vert\Vert \wh \bff_{\wh\theta}^{-1}\big(\bff(m)\big)- \wh \bff_{\wh\theta}^{-1}(\bg^{\bdelta})\Vert_{L^2(\Omega)}-\Vert m-m_0\Vert_{L^2(\Omega)}\right\vert \leq 2\epsilon  + A\Vert \bdelta\Vert_{[L^2([0, T]\times\Gamma)]^{N_s}}\,.
	\end{align}
\end{lemma}
}
\begin{proof}
	We denote by $r(m) = \wh\bff_{\wh\theta}^{-1}\big(\bff(m)\big)-m$. We then have, by Taylor's theorem, that
	\[
	    \wh \bff_{\wh\theta}^{-1}(\bg^{\bdelta}) =\wh  \bff_{\wh\theta}^{-1}(\bff(m_0)+ \bdelta )=m_0+r(m_0)+ d\wh \bff_{\wh\theta}^{-1}[\bff(m_0)]\big(\bdelta \big) +o(\bdelta),
    \]
	where $\lim\limits_{\bdelta\rightarrow 0} \frac{\|o(\bdelta)\|_{L^2(\Omega)}}{\|\bdelta\|_{[L^2([0, T]\times\Gamma)]^{N_s}}} = 0$. We therefore have
	\[
	    \wh	\bff_{\wh\theta}^{-1}\big(\bff(m) \big)- \wh \bff_{\wh\theta}^{-1}(\bg^{\bdelta})= m-m_0+r(m)-r(m_0)-d\wh \bff_{\wh\theta}^{-1}[\bff(m_0)]\big(\bdelta\big)+o(\bdelta)\,.
    \]
	We can now use the triangle inequality to conclude that
	\begin{multline*}
	\left\vert\Vert \wh \bff_{\wh\theta}^{-1}\big( \bff(m) \big)- \wh \bff_{\wh\theta}^{-1}(\bg^{\bdelta})\Vert_{L^2(\Omega)}-\Vert m-m_0\Vert_{L^2(\Omega)}\right\vert\\
 \leq\Vert r(m)-r(m_0)-d\wh \bff_{\wh\theta}^{-1}[\bff(m_0)]\big(\bdelta \big) +o(\bdelta)\Vert_{L^2(\Omega)}\\
	\leq \Vert r(m)\Vert_{L^2(\Omega)} +\Vert r(m_0)\Vert_{L^2(\Omega)} + \Vert d\wh \bff_{\wh\theta}^{-1}[\bff(m_0)]\big(\bdelta\big)\Vert_{L^2(\Omega)} + \Vert o(\bdelta)\Vert_{L^2(\Omega)} \\ 
	\leq 2\epsilon  + A\Vert \bdelta\Vert_{[L^2([0, T]\times\Gamma)]^{N_s}}\,
	\end{multline*}
	where the last step comes from the assumptions in~\eqref{EQ:Ass Learning 1} and~\eqref{EQ:Ass Learning 2}. The proof is complete.
\end{proof}
\RED{This result says that the new objective function $\Phi(m)$ in~\eqref{EQ:Obj} behaves like the quadratic functional $\|m-\wh\bff_{\wh\theta}^{-1}(\bg)\|_{L^2(\Omega)}^2$ when $\epsilon$ and $\bdelta$ vanishing. We note that although convexity is not guaranteed when $\epsilon$ and $\bdelta$ are nonzero, the objective function $\Phi(m)$ inherits a locally favorable landscape: small values $\epsilon$ and $\bdelta$ ensure that $\Phi(m)$ behaves as a controlled perturbation of a convex functional in a neighborhood of $m_0$. Moreover, it is clear that we can replace the strong assumption on the accuracy of $\wh \bff_{\wh\theta}^{-1}$, $\sup\limits_{m}\Vert \wh{\bff}_{\wh{\theta}}^{-1}\big(\bff(m)\big) - m\Vert_{L^2(\Omega)} \leq \epsilon$, with the weaker assumption $\Vert \wh{\bff}_{\wh \theta}^{-1}\big(\bff(m)\big) - m\Vert_{L^2(\Omega)} \le \epsilon\|m\|_{L^2(\Omega)}$, in which case the $2\epsilon$ term in the bound~\eqref{EQ:Convexity} will be replaced by $\epsilon (\|m\|_{L^2(\Omega)}+\|m_0\|_{L^2(\Omega)})$. The conclusion still holds.}

Due to the smoothing property of the forward operator as given in Proposition~\ref{PROP:Frechet Diff}, the stability of the trained inverse operator, measured by the boundedness of its Fr\'echet derivative, is enough to ensure the accuracy of the neural network reconstruction. Therefore, if we could train network approximations with such stability properties, they have good generalization capabilities in the output space.
\begin{lemma}\label{LMMA:Stab}
    Let $m, m_0\in\cC^2(\Omega)\cap[\underline{m}, \overline{m}]$ for some $0<\underline{m}< \overline{m}<+\infty$. Then, when $\|m-m_0\|_{L^2(\Omega)}$ is sufficiently small, there exists a constant $\fc$ such that
    \begin{equation}\label{EQ:Stab}
    \Vert \wh{\bff}_{\wh{\theta}}^{-1}\big(\bff(m)\big) -\wh{\bff}_{\wh{\theta}}^{-1}\big(\bff(m_0)\big) \Vert_{L^2(\Omega)} \leq \fc \|m-m_0\|_{L^2(\Omega)}
    \end{equation}
\end{lemma}
\begin{proof}
By Proposition~\ref{PROP:Frechet Diff}, the map $m\mapsto \bg:=\bff(m)$ is Fr\'echet differentiable with the derivative at $m$ in direction $\wt m$ denoted as $d\bff[m](\wt m)$. By Taylor's theorem, we have \begin{multline*}
    \wh \bff_{\wh\theta}^{-1}\big(\bff(m)\big) =\wh  \bff_{\wh\theta}^{-1}\big(\bff(m_0)+ d\bff[m_0](m-m_0)+o(m-m_0)\big) \\
    =\wh  \bff_{\wh\theta}^{-1}\big(\bff(m_0)\big)+d\wh \bff_{\wh\theta}^{-1}[\bff(m_0)]\big(d\bff[m_0](m-m_0)\big) +\wt o(m-m_0),
\end{multline*}
	where $\lim\limits_{m\rightarrow m_0} \frac{\|o(m-m_0)\|_{[L^2([0, T]\times\Gamma)]^{N_s}}}{\|m-m_0\|_{[L^2(\Omega)}} = \lim\limits_{m\rightarrow m_0} \frac{\|\wt o(m-m_0)\|_{L^2(\Omega)}}{\|m-m_0\|_{[L^2(\Omega)]}} = 0$. We therefore have
	\[
	    \wh	\bff_{\wh\theta}^{-1}\big(\bff(m) \big)- \wh \bff_{\wh\theta}^{-1}(\bff(m_0))= d\wh \bff_{\wh\theta}^{-1}[\bff(m_0)]\big(d\bff[m_0](m-m_0)\big)+\wt o(m-m_0)\,.
    \]
    The bound in~\eqref{EQ:Stab} then follows from the assumption~\eqref{EQ:Ass Learning 2}.
\end{proof}

\RED{When the class of velocity models is sufficiently nice, for instance, when each $m(\bx)$ can be represented with a small number of Fourier coefficients in a narrow frequency band, one can hope that accurate training is achievable. When this is the case, the assumptions in Lemma~\ref{fact:landscape} and Lemma~\ref{LMMA:Stab} are more likely to hold, and therefore the learned model can be utilized to facilitate the FWI reconstruction with the new dataset.} 

\subsection{Computational simplifications}

The reconstruction stage of the coupling can be greatly simplified when the training of the neural network approximation is sufficiently accurate.

First, the coupling method will degenerate into a deep learning based method when we have confidence in our ability to train an accurate deep neural network representation of the inverse operator in FWI. Indeed, when $\wh \bff_{\wh \theta}^{-1}=\bff^{-1}$, that is, $\wh \bff_{\wh \theta}^{-1}$ is exactly the inverse, the reconstruction step~\eqref{EQ:Min} simplifies to
\begin{equation*}
	\wh m_{(i)}=\argmin_{m\in\cM} \dfrac{1}{2}\|m- \wh \bff_{\wh \theta}^{-1}(\bg^\delta)\|_{L^2(\Omega)}^2	+\dfrac{\gamma}{2}\|\nabla m\|_{L^2(\Omega)}^2\,,
\end{equation*}
assuming, only for the sake of simplifying the notation, that the weighting operator $\mu(\bx-\by)$ is taken as an integral operator whose kernel has a Fourier transform that behaves like $\wh\mu\sim |\bk|^{-1}$. This gives a fast inversion for the new data and immediately leads to the optimal selection of the regularization parameter when the regularization term is not too complicated. In this case, we simply did a post-processing on the deep learning reconstruction given by the operator $\wh \bff_{\wh \theta}^{-1}$. The solutions to this are explicitly given as
\[
	\wh m_{(i)}=(\cI+\gamma\Delta)^{-1}\wh \bff_{\wh \theta}^{-1}(\bg^\delta)\,,
\]
where $\cI$ is the identity and $\Delta$ is the Laplacian operator. Therefore, $m$ is simply a smoothed version of the result produced by the trained neural network, $\wh \bff_{\wh \theta}^{-1}(\bg^\delta)$. The exact form of the smoothing effect depends on the selection of $\mu$.

Second, when we can not train an accurate $\bff^{-1}$, but can train a good approximation to the inverse, that is, when the operator $\cI-\wh \bff_{\wh \theta}^{-1} \circ \bff$ is not zero but small in an appropriate operator norm, the FWI problem~\eqref{EQ:Nonl IP Precond} can be solved by using ``Neumann series". More precisely, we can rewrite  ~\eqref{EQ:Nonl IP Precond} as
\begin{equation*}
	m - K(m)=\wh \bff_{\wh\theta}^{-1}(\bg^\delta), \ \ \ K:=\cI-\wh \bff_{\wh \theta}^{-1} \circ \bff\,.
\end{equation*}
Even though $K$ is nonlinear, assuming that the training is sufficiently accurate, we can formally write the solution in a Neumann series as
\begin{equation}\label{EQ:Neumann}
	\wh m_{(i)} =(\cI-K)^{-1}(\wh \bff_{\wh\theta}^{-1}(\bg^\delta))=\sum_{j=0}^\infty K^j\big(\wh \bff_{\wh\theta}^{-1}(\bg^\delta)\big)\,.
\end{equation}
Numerically, we observe that the better the approximation $\wh \bff_{\wh\theta}^{-1}$ is to $\bff^{-1}$, the faster the series converges. In theory, however, this series doesn't necessarily converge (again, since $K$ is nonlinear). For the training we had, see more discussion in Section~\ref{SEC:Num}, a few terms of the Neumann series often provide sufficient accuracy for the reconstruction.

\RED{Let us emphasize that by the informal analysis in Section~\ref{SUBSEC:Training Error}, the error in the learning implies roughly that $|\cF\big(m-K(m)\big)(\bk)| \sim \zeta(\bk)\|m\|_{L^2(\Omega)}$ with $\zeta(\bk)$ denotes the frequency-dependent error amplification factor arising from the learned inverse operator, and it is large for large $|\bk|$. Due to the fact that the operator norm of $\cI-K$ is bounded below by $\max_{\bk}\zeta(\bk)$, this means that the convergence speed of the Neumann series is controlled by the worst training error in the (high-frequency) Fourier modes.}

\subsection{Refinement to reconstruct outside of training domain}\label{Sec:outside_domain}

It is important to point out that the weight $\bmu$ in the weighted training scheme~\eqref{EQ:Learning Opt} should be selected to emphasize the low-frequency components of the output and penalize the high-frequency components. It should not completely remove the high-frequency components. If it does, then the high-frequency components of the velocity field in the reconstruction stage can not be recovered with the optimization problem~\eqref{EQ:Min}. This is an obvious yet important observation that we summarize as a lemma to emphasize it.
\begin{lemma}
    Let $\wh\bff_{\wh\theta}^{-1}$ be such that for any $m$, $\cF[\wh\bff_{\wh\theta}^{-1}(\bff(m))](\bk)=\bzero$ $\forall |\bk|>k_0$, and $\wh m{(i)}$ be reconstructed from~\eqref{EQ:Min} with a gradient-based iterative scheme or the Neumann series method in~\eqref{EQ:Neumann}. Then $\cF[\wh m_{(i)}](\bk)=\bzero$ $\forall |\bk|>k_0$.
\end{lemma}
\begin{proof}
Under the assumption on $\wh\bff_{\wh\theta}^{-1}$, it is straightforward to check that
$\cF(m_0)(\bk)=0$ ($m_0:=\wh\bff_{\wh\theta}^{-1}(\bg^\delta)$) $\forall |\bk|>k_0$, and $\cF(K^j m_0)(\bk)=0$ $\forall |\bk|>k_0$, for any $j\ge 1$. Therefore $\cF(\wh m{(i)})(\bk)=0$ $\forall |\bk|>k_0$. Let $m_\ell$ be the $\ell$-th iteration of a gradient based iterative scheme, then $\cF(r(m_\ell))(\bk)=0$ ($r(m):=\wh \bff_{\wh \theta}^{-1}\Big(\bff(m)\Big) - \wh \bff_{\wh \theta}^{-1} (\bg^\delta)$) $\forall |\bk|>k_0$. This leads to the fact that $\cF\Big(d\Phi[m_\ell](\delta m)\Big)(\bk)=0$ for any $\delta m$. Therefore, $\cF(m_{\ell+1})(\bk)=0$ $\forall |\bk|>k_0$. The rest of the proof follows from an induction.
\end{proof}

For any velocity field that can be written as $m_b+\delta m$ with $m_b$ the prediction of the trained neural network and $\delta m$ outside of the range of the neural network but either has small amplitude (compared to that of $m$) or has large amplitude by small support compared to the size of the domain (in which case $\delta m$ is very localized), we can recover $\delta m$ with an additional linearized reconstruction step. We linearize the inverse problem around the network prediction $\bff_{\wh\theta}^{-1}(\bg^\delta)$. The reconstruction can be performed with a classical migration scheme, or equivalently by minimizing the following quadratic approximation to the functional~\eqref{EQ:Obj}:
\begin{equation}\label{EQ:Obj Quadra}
    \Psi_Q(m)=\dfrac{1}{2}\left\|\bff(m_b)+d\bff[m_b]\Big(m-m_b\Big)-\bg^\delta\right\|_{[L^2((0, T]\times \Gamma)]^{N_s}}^2+\dfrac{\gamma}{2}\| m-m_b\|_{L^2(\Omega)}^2\,,
\end{equation}
where $m_b:=\bff_{\wh\theta}^{-1}(\bg^\delta)$.

\section{Computational implementation}
\label{SEC:Implementation}

We now provide some details on the implementation of the coupling framework we outlined in the previous section. For computational simplicity, we focus on the implementation in two spatial dimensions even though the methodology itself is independent of the dimension of the problem.

\subsection{Computational setup}
\label{SUBSEC:Setup}

\begin{figure}[htb!]
	\centering
	\includegraphics[width=0.45\linewidth,trim=5cm 2cm 4cm 2cm,clip]{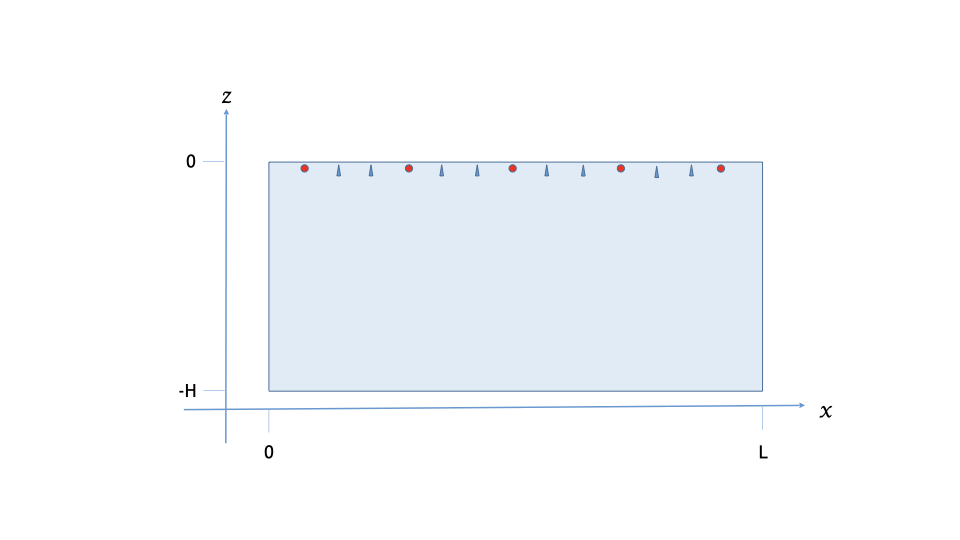}\hskip 1cm
	\includegraphics[width=0.45\linewidth,trim=5cm 2cm 4cm 2cm,clip]{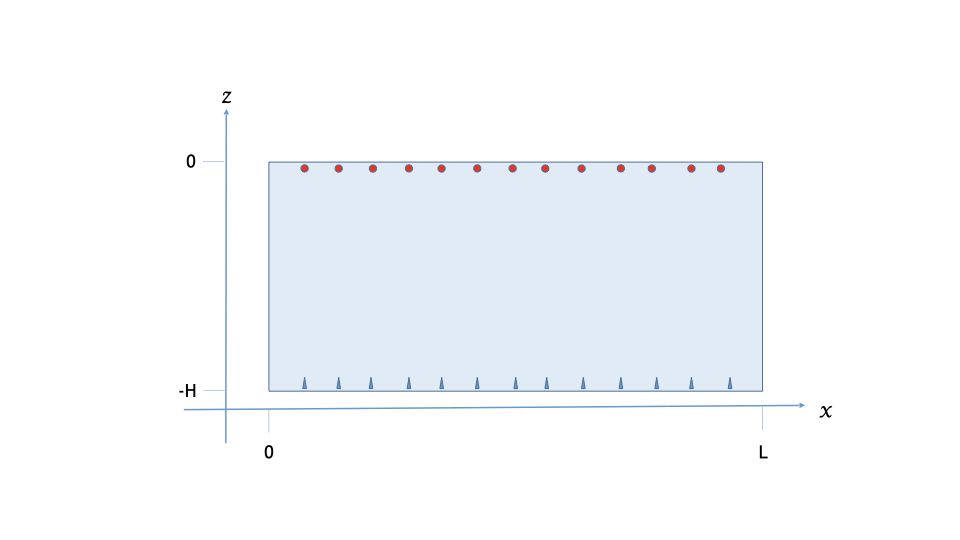}
	\caption{The two-dimensional computational domain $\Omega=(0, L)\times(-H, 0)$ for wave propagation. Periodic boundary conditions are imposed on the left and right boundaries. In geophysical applications, sources and detectors are placed on the top boundary (left), while in medical ultrasound applications, sources (red dots) and detectors (blue triangles) can be placed on both the top and the bottom boundaries (right).}
	\label{FIG:Setup}
\end{figure}
For the purpose of being concrete, we first describe briefly the geometrical setting under which we implement the learning and reconstruction algorithms. Let $\bx=(x, z)$. The computational domain of interests is $\Omega=(0, L)\times(-H, 0)$. We impose periodic boundary conditions on the left and right boundaries of the domain. Probing sources and detectors are placed on the top and bottom boundaries $\Gamma_t=(0, L)\times\{0\}$ and $\Gamma_b=(0, L)\times\{-H\}$, depending on the exact applications we have in mind. In geophysical applications, the sources and detectors are both placed on the top boundary, while in medical ultrasound type of applications, the sources and detectors could be placed on the opposite sides; see Figure~\ref{FIG:Setup} for an illustration. Under this setup, the wave equation ~\eqref{EQ:Wave} with a source $h(t, x)$ on the top boundary and a reflective bottom boundary takes the form
\begin{equation}\label{EQ:Wave Equation}
\begin{array}{rcll}
\dfrac{1}{m^2}\dfrac{\partial^2 u}{\partial^2 t} -\Delta u & = & 0, & \mbox{in}\ \ (0, T]\times \Omega,\\[1ex]
u(0, x, z)=\dfrac{\partial u}{\partial t}(0, x, z)  &=& 0,& (x,z)\in(0, L)\times (-H, 0),\\
u(t, 0, z) & = & u(t, L, z), &  (t, z)\in (0, T] \times (-H, 0),\\
\dfrac{\partial u}{\partial z}(t,x,-H) & = & 0, & (t, x)\in (0, T]\times (0, L),\\
\dfrac{\partial u}{\partial z}(t, x, 0) & = & h(t, x), & (t, x)\in (0, T] \times (0, L).
\end{array}
\end{equation}
Similar equations can be written down for other types of source-detector configurations.

\subsection{The neural network for learning}
\label{SUBSEC:Network}

With the above computational setup, we can generate the training dataset~\eqref{EQ:Data Training} by solving the wave equation~\eqref{EQ:Wave Equation} with given source functions. We will describe in detail how the training dataset is generated, including the spatial-temporal discretization of the wave equation~\eqref{EQ:Wave Equation}.
 
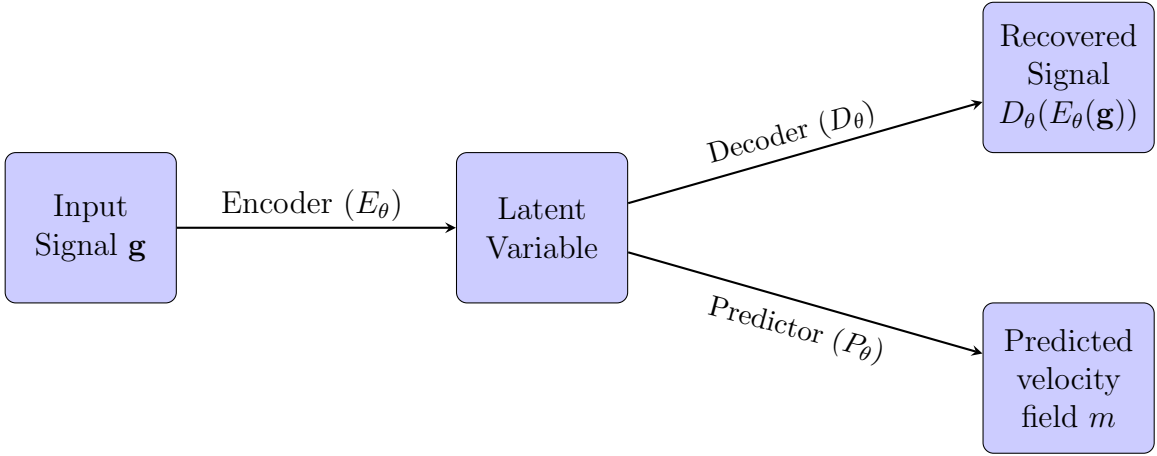
\begin{figure}[!htb]
	\centering
\begin{tikzpicture}[node distance=2cm]
\tikzstyle{startstop} = [rectangle, rounded corners, minimum width=1cm, minimum height=2cm,text centered,text width=2cm, draw=black, fill=blue!20];
\tikzstyle{arrow} = [thick,->,>=stealth];
\node (input)[startstop]{Input Signal $\bg$};
\node (latent)[startstop,text=black,right of=input,xshift=4cm]{Latent Variable};
\node (recov)[startstop,right of=latent,xshift=5cm, yshift=2cm]{Recovered Signal $D_\theta(E_\theta(\bg))$};
\node (pred)[startstop,right of=latent,xshift=5cm, yshift=-2cm]{Predicted velocity field $m$};
\draw [arrow] (input) -- node[anchor=east,xshift=1.3cm, yshift=3mm]{Encoder ($E_\theta$)}(latent);
\draw [arrow,bend left] (latent) -- node[anchor=east,xshift=1.3cm, yshift=4mm]{}(recov);
\draw [arrow] (latent) -- node[anchor=east,xshift=1.3cm, yshift=-4mm]{}(pred);
\end{tikzpicture}
\put(-170,110){\rotatebox{15}{\small Decoder ($D_\theta$)}}
\put(-170,53){\rotatebox{345}{\small Predictor ($P_\theta$)}}
	\caption{Network flow for learning the approximate inverse operator. Training objective is to select $\theta$ such that $\bg=D_\theta(E_\theta(\bg))$ and $m=P_\theta(E_\theta(\bg))$ for every datum pair ($\bg$, $m$).}
	\label{FIG:Training}
\end{figure}
We construct an autoencoder network scheme to represent the inverse operator. The learning architecture contains three major substructures: an encoder network $E_\theta$, a decoder network $D_\theta$, and an additional predictor network $P_\theta$; see Figure~\ref{FIG:Training} for an illustration of the network flow. More information on the construction of the encoder, the decoder and the predictor is documented in Appendix~\ref{SEC:Appendix C}. The encoder-decoder substructure is trained to regenerate the input data, while the predictor reads the latent variable to predict the velocity field $m$. In terms of the input-output data, the network training aims at finding the network parameter $\theta$ such that
\begin{equation}
    \bg_j= D_\theta(E_\theta(\bg_j))\ \ \mbox{and} \ \ m_j= P_\theta(E_\theta(\bg_j)),\ \ 1\le j\le N\,.
\end{equation}
This is done by a minimization algorithm that minimizes a combined $\ell^1$-$\ell^2$ loss function with the $\ell^1$ loss for the encoder-decoder substructure while $\ell^2$ for the encoder-predictor substructure. More precisely, we train the network by solving
\begin{equation}
    \begin{matrix}
     &\displaystyle  \wh \theta=\argmin_{\theta\in\Theta} \wt \cL(\theta),\\[0ex]
     \mbox{with}&\\[1ex]
	&\wt \cL(\theta) := \dfrac{1}{N}\dsum_{j=1}^N \|\bg_j-D_\theta(E_\theta(\bg_j))\|_{\ell^1} +\dfrac{1}{2N}\dsum_{j=1}^N \|\bmu\circledast \big(\fm_j-P_\theta(E_\theta(\bg_j))\big)\|_{\bbM}^2 \,,
    \end{matrix}
\end{equation}
\RED{where $\|\cdot\|_{\bbM}$ denotes the Euclidean norm taken over the feature vector $\bbM$.} While the $\ell^1$ loss for the encoder-decoder substructure is standard in the learning literature, the second part of the loss function is simply what we introduced in~\eqref{EQ:Learning Opt}. Once the training is performed, the approximated inverse is taken as
\[
	\bff_{\wh \theta}^{-1}: = P_{\wh \theta} \circ E_{\wh \theta}\,.
\]

Let us emphasize that the main motivation for us to adopt this autoencoder framework, instead of directly training a network for $\bff_{\wh \theta}^{-1}$, is to take advantage of the commonly observed capability of autoencoders to identify lower dimension features from high-dimensional input data. That is, very often, one can train the autoencoder such that the latent variable $E_\theta(\bg)$ contains most of the useful information in $\bg$ but has a much lower dimension than $\bg$. This lowers the dimension of the predictor network and therefore makes it easier to train the overall network. Moreover, the weighted optimization we used in the encoder-predictor substructure further stabilizes the learning process by focusing on matching the lower-frequency components of the output.

\subsection{Learning-assisted FWI inversion}\label{SEC:algorithms}

To implement the preconditioned FWI reconstruction method, that is, the solution to the least-squares optimization problem~\eqref{EQ:Min}, we tested two different algorithms. 

\paragraph{Quasi-Newton method with adjoint state.} We implemented a quasi-Newton method based on the BFGS gradient update rule~\cite{ReBaHi-SIAM06} for the numerical reconstruction. This BFGS optimization algorithm itself is standard, so we will not describe it in detail here. The algorithm requires the gradient of the objective function $\Phi(m)$ defined in~\eqref{EQ:Obj}. We evaluate the gradient with a standard adjoint state method. The procedure is documented in Algorithm~\ref{ALG:Adjoing-Gradient} of Appendix~\ref{SEC:Appendix A}. The main complication that the learning stage brings into the adjoint state calculation is that we will need the transpose of the gradient of the neural network with respect to its input. This imposes restrictive accuracy requirements on the training of the neural network in the sense that we need the network to learn not only the map from measurement to the velocity field but also the derivative of the operator.  
 
\paragraph{Neumann series method.} The informal Neumann series method based on ~\eqref{EQ:Neumann} is more training-friendly since it does not require the adjoint operator of the learned approximate inverse $\wh \bff_{\wh\theta}^{-1}$. This is why we use the method in computation, \emph{even though we do not have a theory about the convergence of the series}. We implemented a $J$-term truncated Neumann series approximation
\begin{equation}\label{EQ:Neumann J}
	\wh m_{(i)} =\sum_{j=0}^{J-1} K^j\big(\wh \bff_{\wh\theta}^{-1}(\bg^\delta)\big)\,.
\end{equation}
The computational procedure is summarized in Algorithm~\ref{ALG:Neumann} of Appendix~\ref{SEC:Appendix B}.

\section{Numerical experiments}
\label{SEC:Num}

We now present some numerical simulations to illustrate some of the main characteristics of the proposed framework of coupling deep learning with model-based FWI reconstruction. We fix the computational domain to be $\Omega = [0,1]\times[-1,0]$, that is, $L = H = 1$. In this proof-of-concept study, we use acoustic source functions that can generate data at all frequencies. We leave it as future work to consider the situation where low-frequency wavefield data are impossible to measure in applications such as seismic imaging.

\subsection{Velocity feature models}
 
In this work, we consider two different feature models for the output velocity field of the neural network.

\paragraph{Generalized Fourier feature model.} In the first model, we represent $m(\bx)$ as linear combinations of the Laplace-Neumann eigenfunctions on the computational domain $\Omega$. To be precise, let $(\lambda_\bk, \varphi_\bk)$ ($\bk=(k_x, k_z)\in \bbN_0\times \bbN_0$) be the eigenpair of the eigenvalue problem:
	\begin{equation*}\label{EQ:Laplace}
	-\Delta \varphi = \lambda \varphi, \quad \mbox{in}\ \ \Omega, \qquad \bn\cdot\nabla \varphi = 0, \quad \mbox{on}\ \ \partial \Omega\,.
	\end{equation*}
	where $\bn(\bx)$ is the unit outward normal vector of the domain boundary at $\bx\in\partial\Omega$. Then $\lambda_\bk=\left({k_x\pi}\right)^2+\left({k_z\pi}\right)^2$, and
	\[
	\varphi_\bk (x, z) = \cos({k_x\pi} x) \cos({k_z\pi} z)\,.
	\]
	In our numerical simulations, we take 
	\begin{equation}\label{EQ:Velocity Model 1} 
	m(\bx) = \sum_{k_x, k_z=0}^M \fm(\bk) \ \varphi_\bk(x,z)\,, 
	\end{equation}
	for some given $M$. The generation of the random coefficients $\fm(\bk)$ will be described in detail in the next section.

\paragraph{Gaussian mixture model.} The second feature model we take is the Gaussian mixture model. More precisely, we represent $m({\bf x})$ as a superposition of Gaussian functions:
\begin{equation}\label{EQ:Velocity Model 2}
	m(\bx) = m_0+ \sum_{k=1}^M c_k e^{-\dfrac{1}{2}(\bx-\bx_0^k)^\fT\Sigma_k^{-1}(\bx-\bx_0^k)}\,.
\end{equation}
With a small number of highly localized Gaussians, successful reconstruction of such a model could provide insight into source locating problems in seismic applications~\cite{ChChWuYa-JCP18}. This is the main motivation for us to consider this model.

\subsection{Learning dataset generation} \label{SUBSEC:Data Generation}

To generate training data, we generate a set of velocity fields and then solve the wave equation model~\eqref{EQ:Wave Equation} with source functions $\{h_s\}_{s=1}^{N_s}$ to get the corresponding wave field data at the detectors.

\paragraph{Generating velocity fields.} We first construct a set of $N$ random velocity fields $\{m_j\}_{j=1}^N$ using the representation~\eqref{EQ:Velocity Model 1} or~\eqref{EQ:Velocity Model 2}. We do this by randomly choosing the coefficients $\{\fm(\bk)\}_{\bk\in\bbN_0\times\bbN_0}$ from the uniform distribution~$\cU[-0.5, 0.5]$ when considering the model~\eqref{EQ:Velocity Model 1} and the coefficients $c_k$ from $\cU[0, 5]$, $\bx_0^k$ from $\cU(-H, 0)\times\cU(0, L)$, $(\Sigma_k)_{ij}$ from $\cU[0, 0.2] + 0.1$ and $m_0 = 10$ when using the model~\eqref{EQ:Velocity Model 2}. To mimic the frequency content of realistic velocity fields, we force the coefficient $\fm(\bk)$ in the random Fourier model~\eqref{EQ:Velocity Model 1} to decay asymptotically as
\begin{equation}\label{EQ:decay_rule}
    \fm(\bk) \sim \fm(\bk)[(k_x+1)(k_z+1)]^{-\alpha},\ \ \mbox{for large}\ \ |\bk|=\sqrt{k_x^2+k_z^2}
\end{equation}
\RED{with $\alpha\ge 0$ given in the concrete examples later, and it serves the same purpose and plays the role of this weighting matrix $\bmu$ in \eqref{EQ:Learning Opt}. Note that while we require the Fourier modes to decay with the frequency, the algebraic rate is slow enough to keep many modes relevant in the modeling and inversion process.}

To make sure that the velocity fields we generated are physically meaningful, we rescale them so that the velocity lives in a range $[\underline{m}, \overline{m}]$ ($0<\underline{m}<\overline{m}<+\infty$). The linear rescaling is done through the operation
\begin{equation}\label{EQ:Scaling}
    m(\bx) \leftarrow \dfrac{\overline{m}-\underline{m}}{m^*- m_*}m(\bx)+\dfrac{\underline{m}m^*-\overline{m} m_*}{m^*-m_*},
\end{equation}
where $\displaystyle m^*:=\max_\bx m(\bx)$ and $\displaystyle m_*:=\min_\bx m(\bx)$. 

In Figure~\ref{FIG:Velo Samples} we show some typical samples of the velocity field generated from the aforementioned process. The top panel of Figure~\ref{FIG:Velo Samples} shows the surface plots of $4$ different randomly generated velocity fields using the model~\eqref{EQ:Velocity Model 1} with $M = 4$. The bottom panel presents the surface plots of $4$ random realizations of the velocity field given by the model~\eqref{EQ:Velocity Model 2} with $M = 2$. Random noise at different levels will be added to the sampled velocity fields to study the generalization of the learning scheme we have. The exact level of noise will be given later in concrete examples.
\begin{figure}[!htb]
	\centering
    \includegraphics[width=0.24\textwidth,trim=0cm 2cm 0cm 1cm,clip]{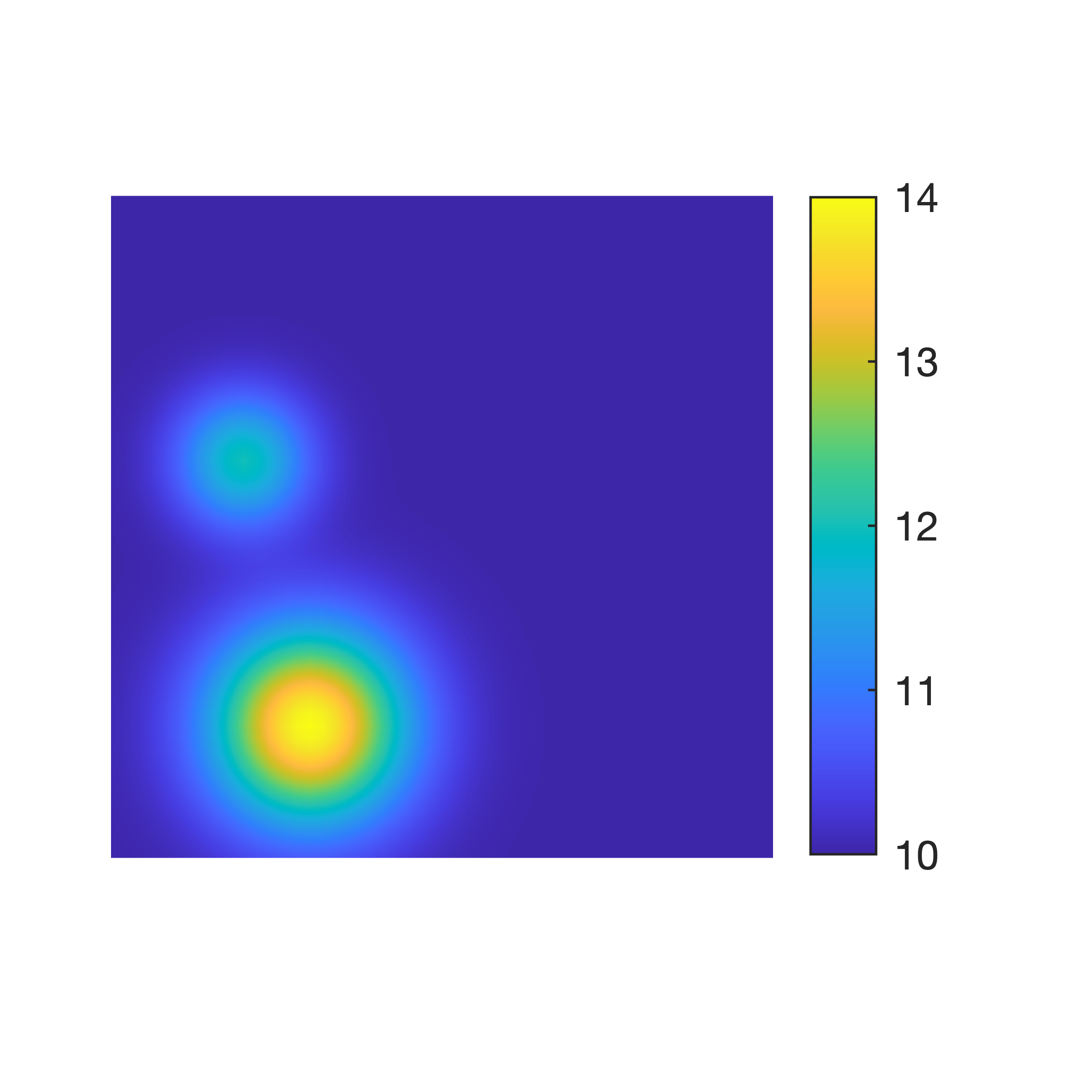}
	\includegraphics[width=0.24\textwidth,trim=0cm 2cm 0cm 1cm,clip]{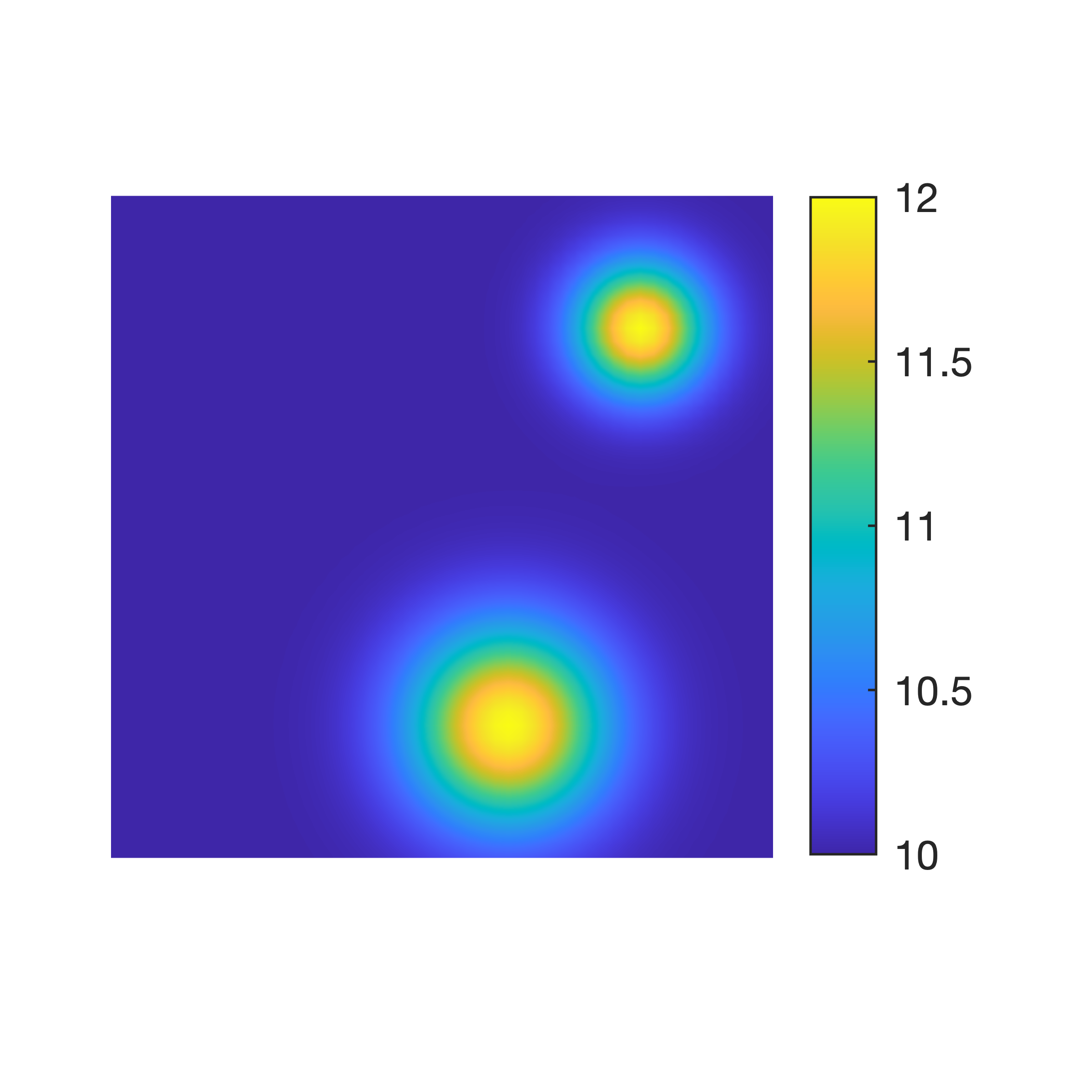}
	\includegraphics[width=0.24\textwidth,trim=0cm 2cm 0cm 1cm,clip]{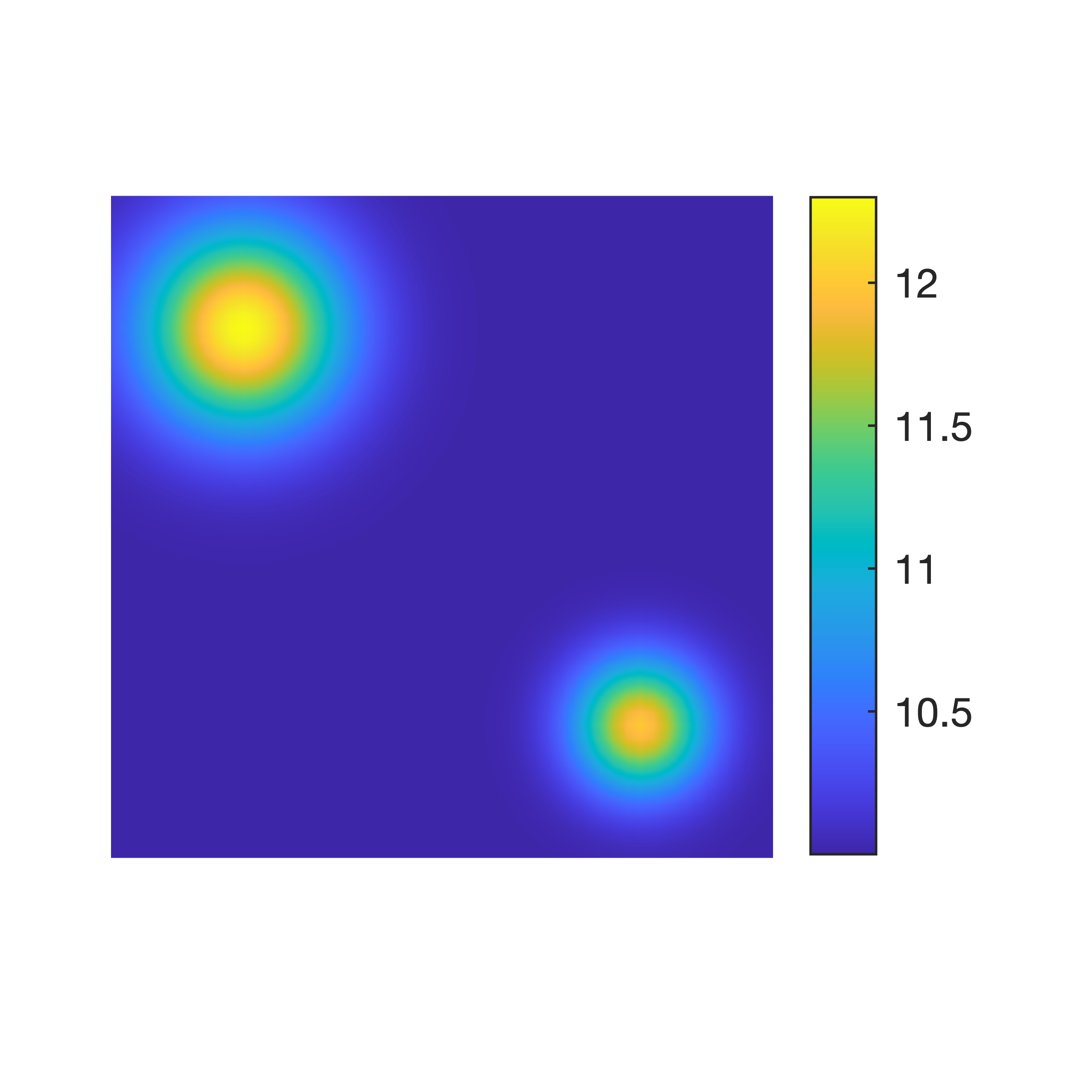}
	\includegraphics[width=0.24\textwidth,trim=0cm 2cm 0cm 1cm,clip]{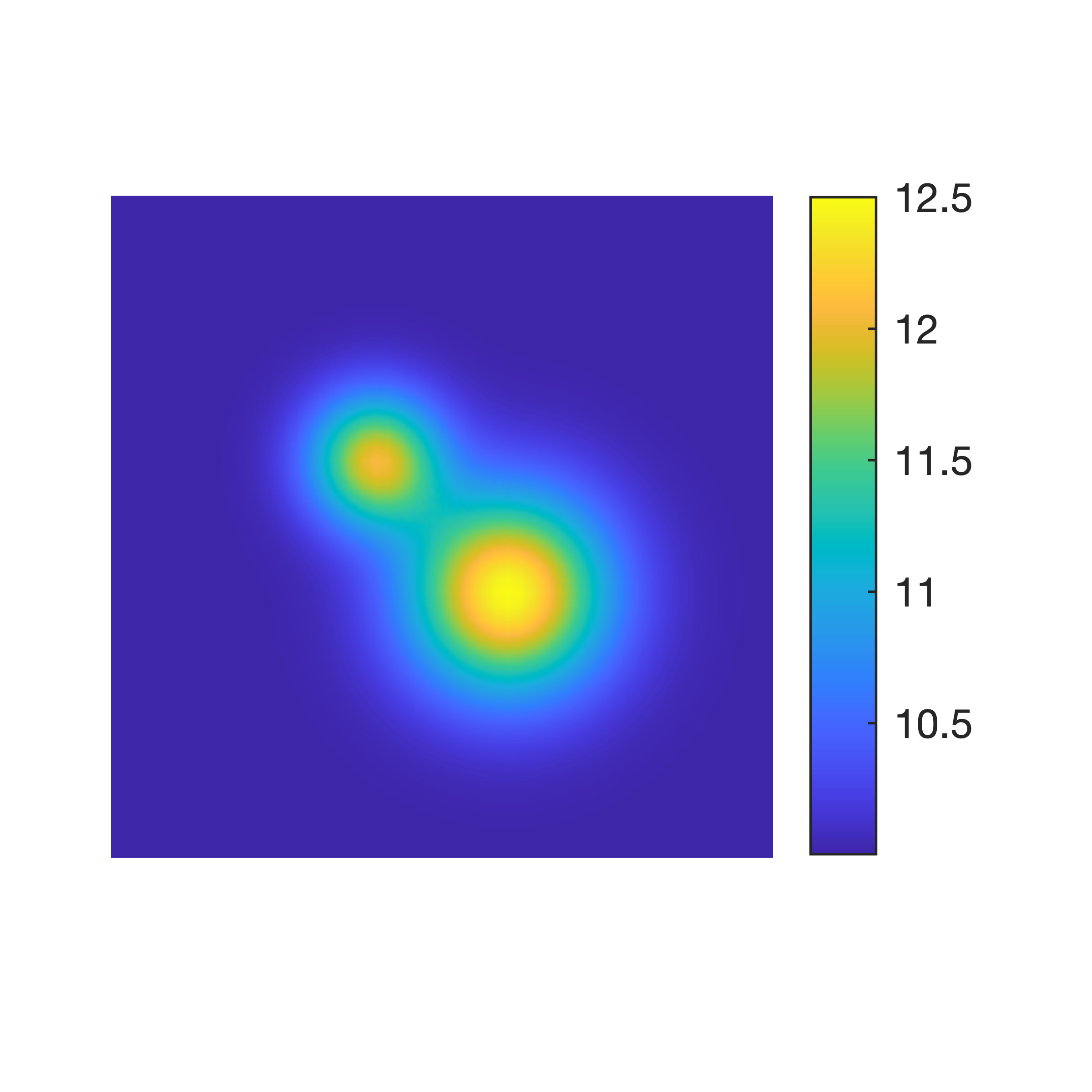}\\
	\includegraphics[width=0.24\textwidth,trim=0cm 2cm 0cm 1cm,clip]{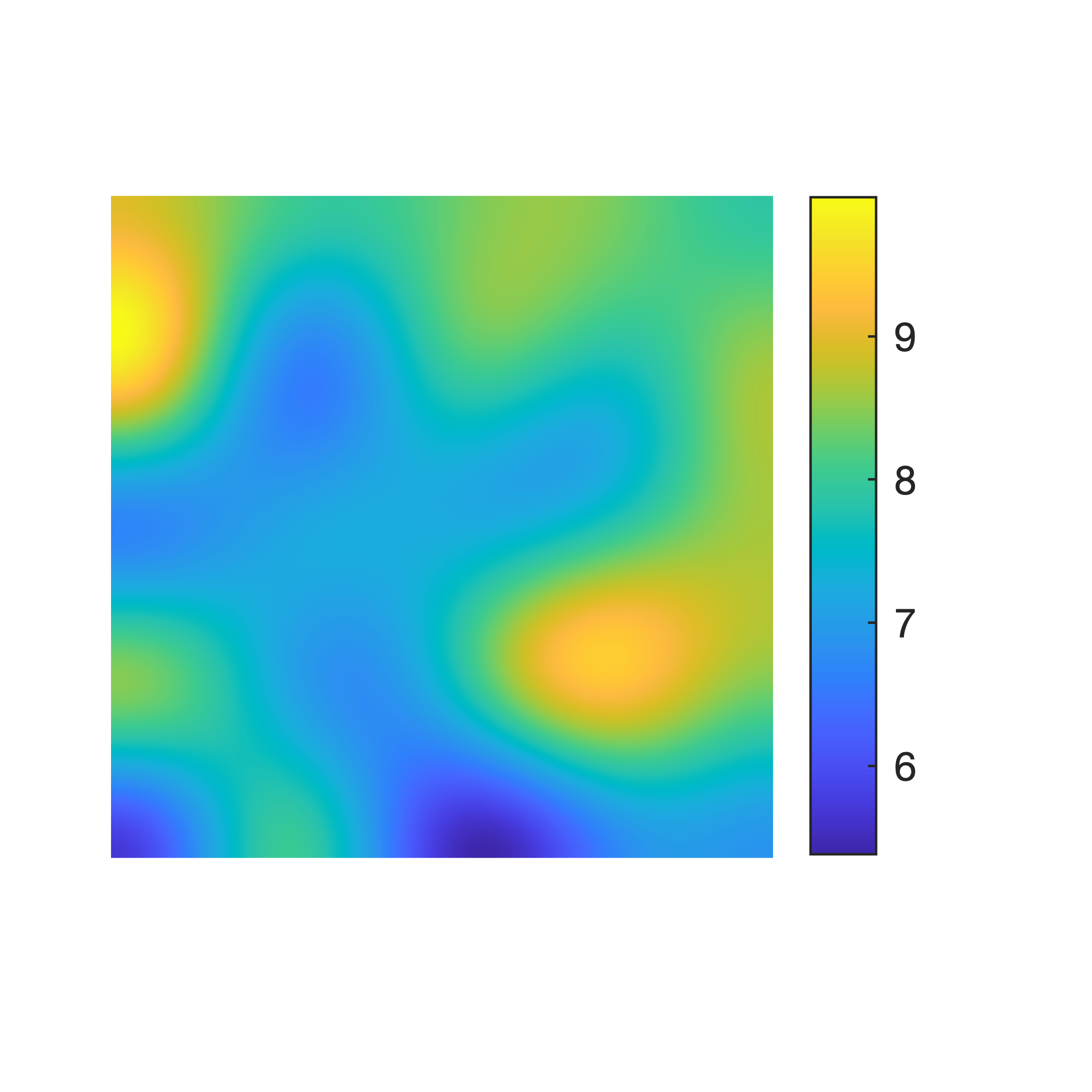}
	\includegraphics[width=0.24\textwidth,trim=0cm 2cm 0cm 1cm,clip]{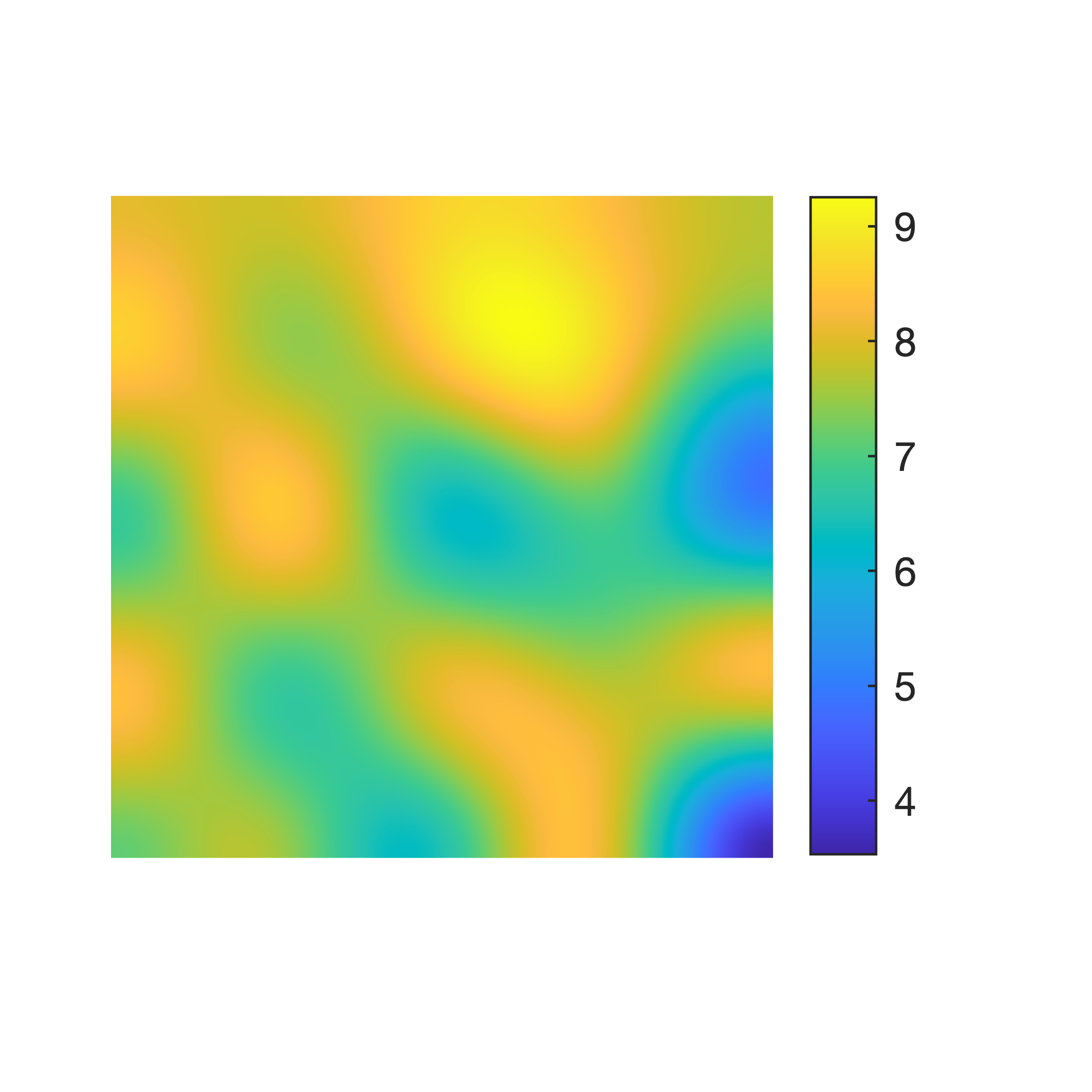} 
	\includegraphics[width=0.24\textwidth,trim=0cm 2cm 0cm 1cm,clip]{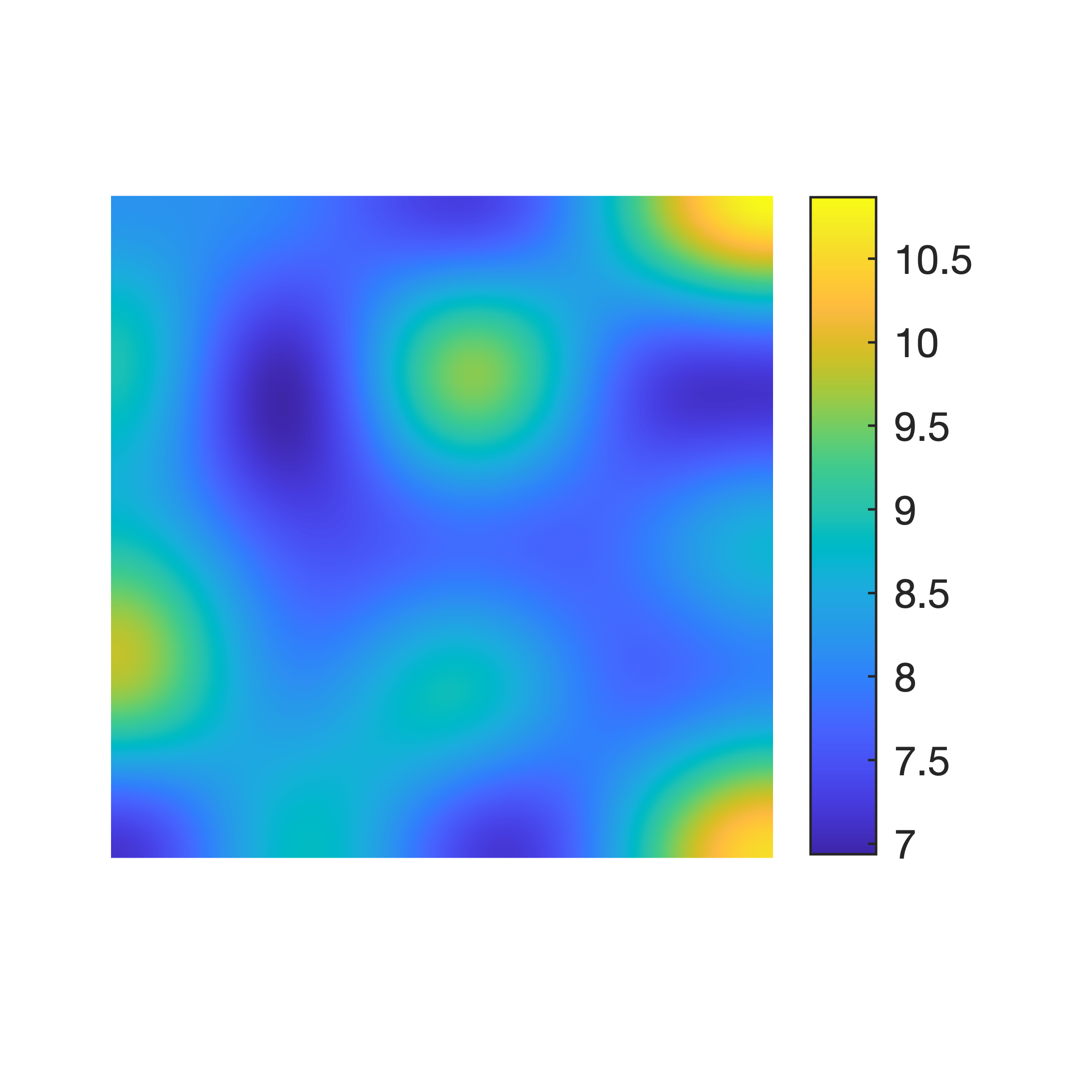}
	\includegraphics[width=0.24\textwidth,trim=0cm 2cm 0cm 1cm,clip]{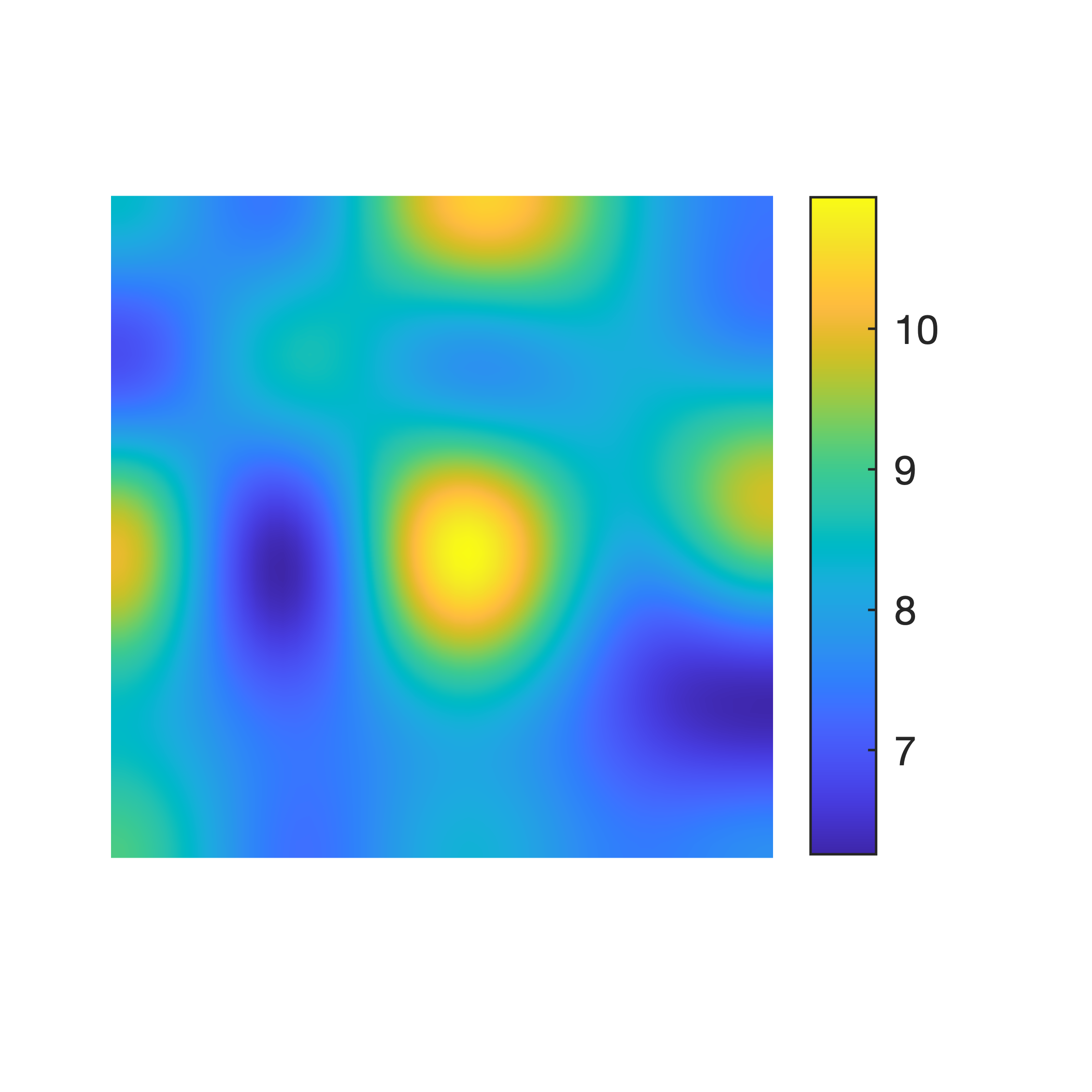}
	\caption{Random samples of the velocity field for training of the neural networks. Top row: velocity fields generated from~\eqref{EQ:Velocity Model 1} with $M = 2$; bottom row: velocity fields generated from~\eqref{EQ:Velocity Model 2} with $M = 4$.}
	\label{FIG:Velo Samples}
\end{figure}
 
\paragraph{Finite difference scheme for the wave equation.} We use the time-domain stagger-grid finite difference scheme that is second-order in both the time and the spatial directions to solve the wave equation~\eqref{EQ:Wave Equation}. Precisely, the discretization is performed with elements over the Cartesian grids formed by $(x_k,z_l) = (k\Delta x,l\Delta z),k,l = 0,1,...,N_L$ with $\Delta x= L/N_L$ and $\Delta z = H/N_L$. The receivers are equally placed at the bottom surface, coinciding with the grid points, as documented in the right panel of Figure~\ref{FIG:Setup}, namely, there are $N_L+1$ receivers for each velocity model. We then record the wave signal starting at time $t_0$ and take another shot every $j\Delta t$ until the termination time $T$; here, $j$ is a positive integer, and $\Delta t$ is the uniform time step size for the forward wave solver. As an example for illustration, we take 
\begin{equation}\label{surface_source}
h(t,x) = e^{\frac{-(x - 0.6)^2}{0.01}}-e^{\frac{-(x - 0.3)^2}{0.01}}
\end{equation}
to be the top source in~\eqref{EQ:Wave Equation} and present the recorded time series wave signals in Figure~\ref{samples_signal}. Table~\ref{signal_generate} summarizes the parameters we used to generate these wavefield signals.
\begin{table}
 	\begin{center}
 		\scalebox{1.0}{
 			\begin{tabular}{c c c c c c c}
 				\hline
 				$L$ & $H$ & $N_L$ & $\Delta t$ & $t_0$ & $j$ & $T$ \\
 				\hline
 				$1$ & $1$ & $50$ &$0.0005$&  $0$ & $20$ & $0.5$\\
 				\hline
 			\end{tabular}
 		}
 	\end{center}
 	\caption{Values of parameters in the spatial and temporal discretization of the wave equation and the time node of the recorded wave signal.}
 	\label{signal_generate}
\end{table}

Figure~\ref{samples_signal}, from the left panel to the right panel, shows the time series wave signals at the bottom surface generated from the velocity model satisfying (\ref{EQ:Velocity Model 1}) with $M = 4$, and the velocity model satisfying (\ref{EQ:Velocity Model 2}) with $M = 2$, respectively; from the top panel to the bottom panel are the wave signals without noise, with $10\%$ multiplication Gaussian noise, and with $10\%$ additive Gaussian noise, respectively.
 \begin{figure}[!hbt]
 	\centering
    \includegraphics[width=0.48\textwidth,trim=1cm 0cm 1cm 0cm,clip]{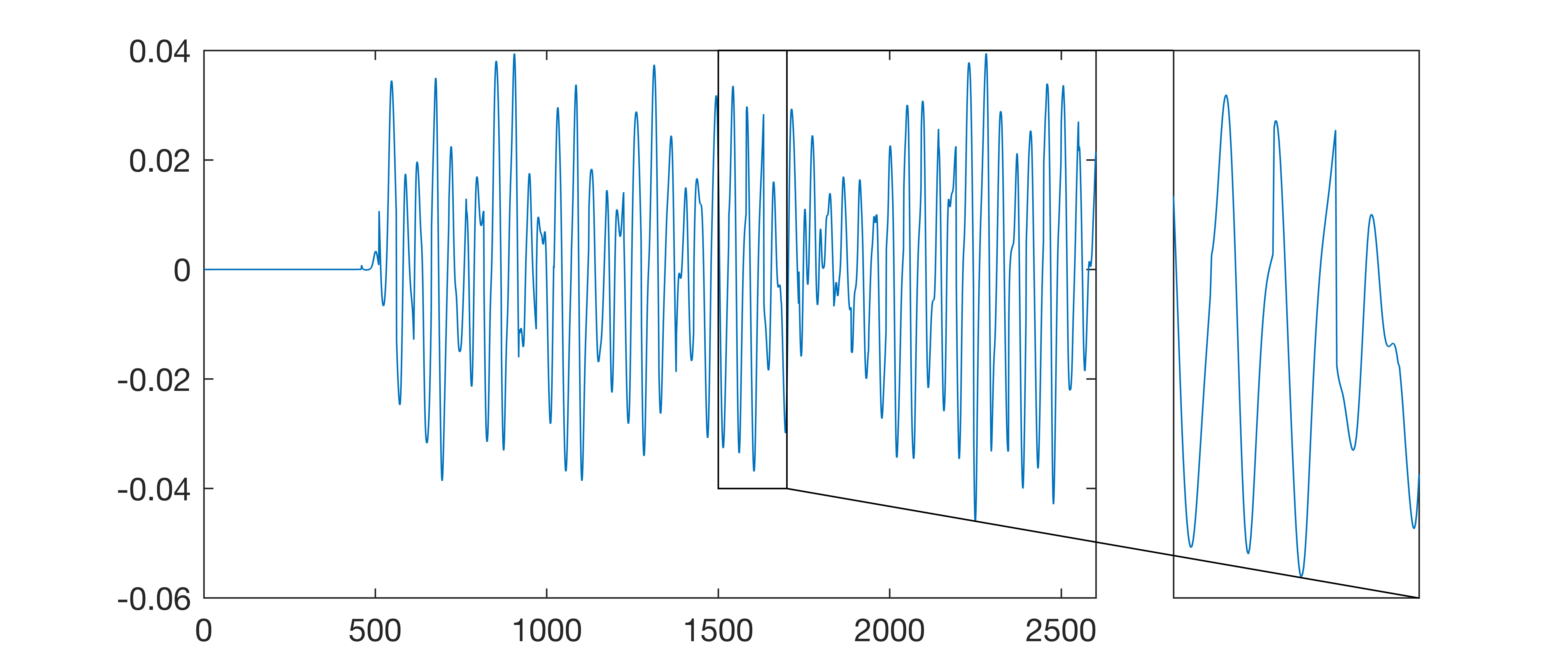}
 	\includegraphics[width=0.48\textwidth,trim=1cm 0cm 1cm 0cm,clip]{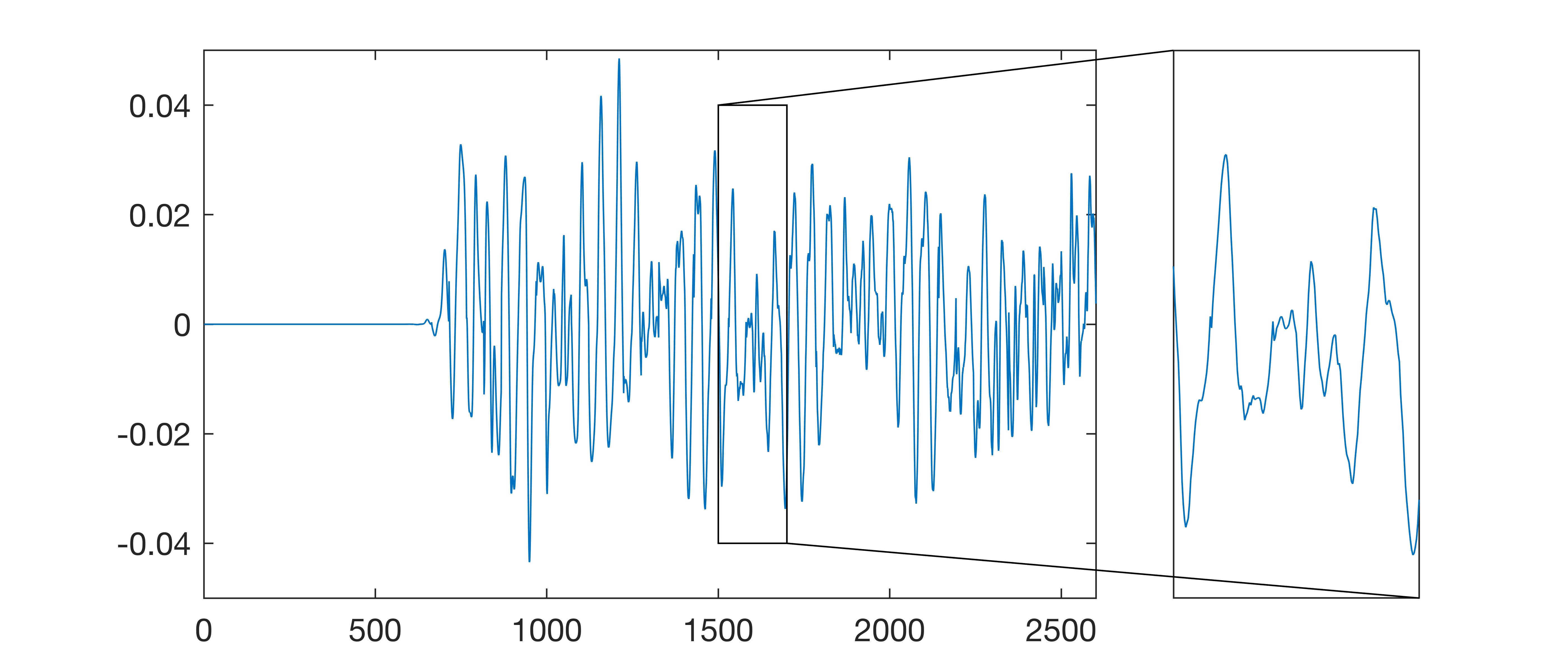}\\
    \includegraphics[width=0.48\textwidth,trim=1cm 0cm 1cm 0cm,clip]{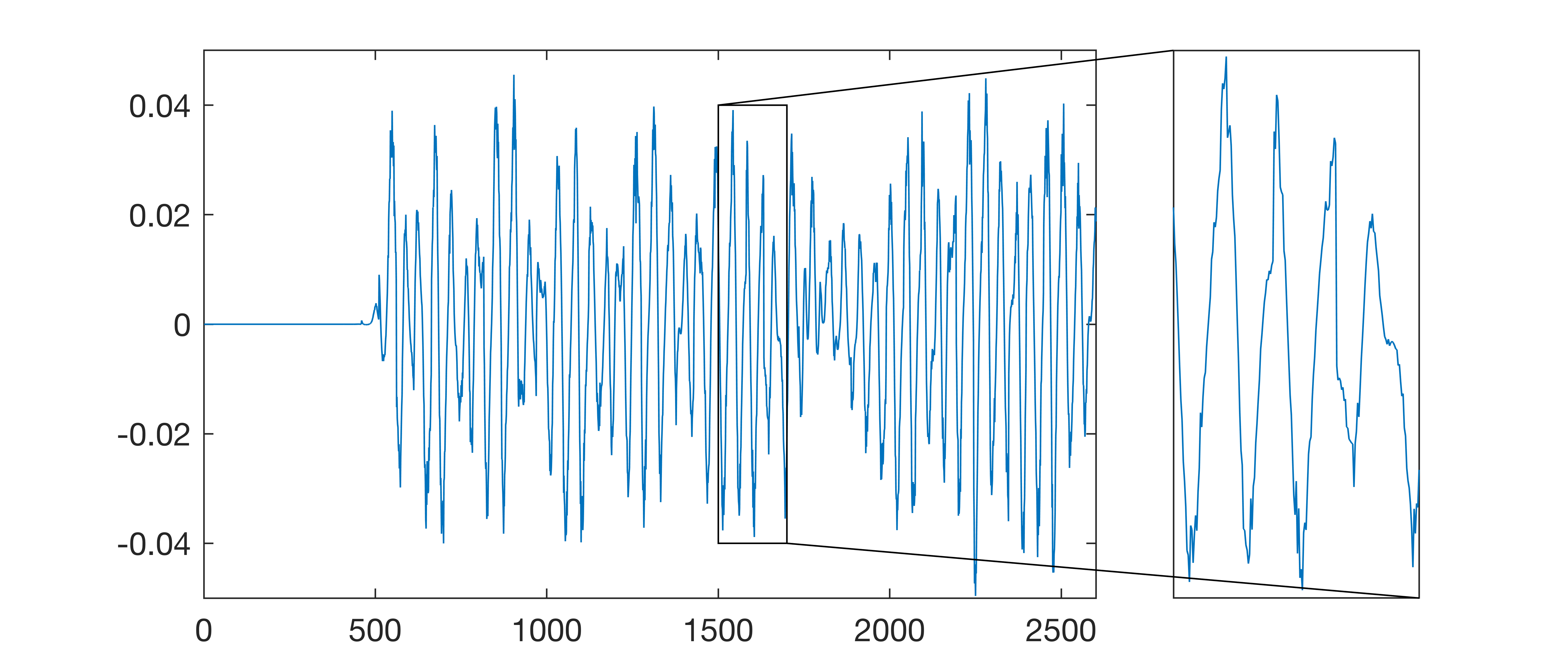} 
 	\includegraphics[width=0.48\textwidth,trim=1cm 0cm 1cm 0cm,clip]{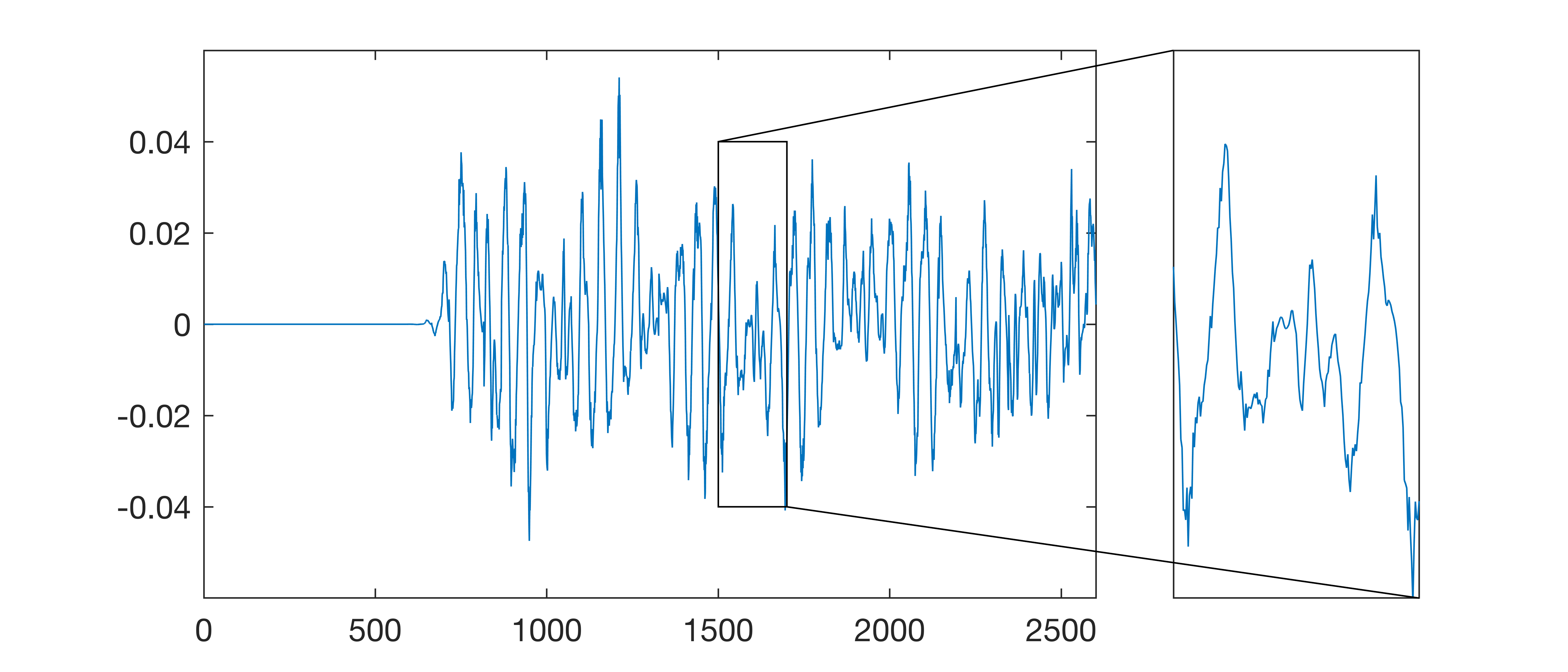} \\
    \includegraphics[width=0.48\textwidth,trim=1cm 0cm 1cm 0cm,clip]{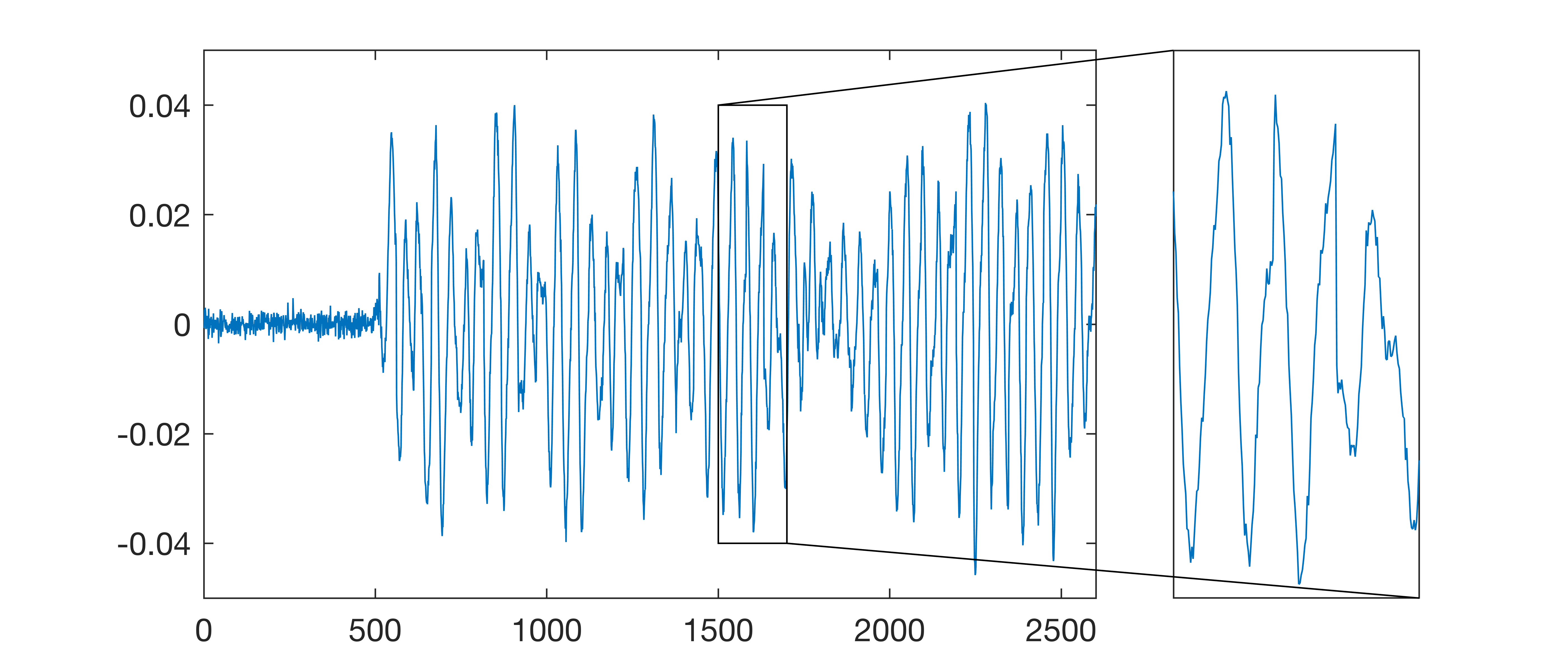}
 	\includegraphics[width=0.48\textwidth,trim=1cm 0cm 1cm 0cm,clip]{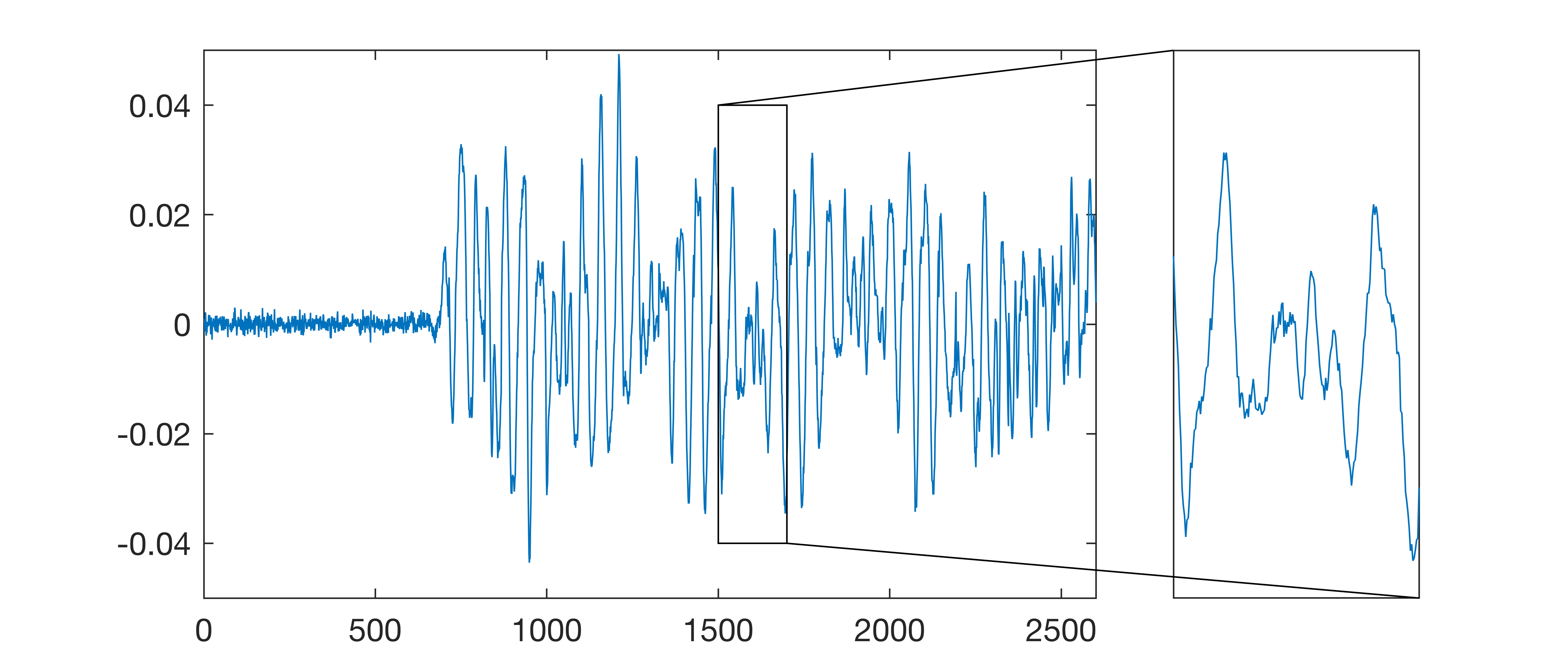} 
 	\caption{\small{The left panel presents time series wave signals at the bottom surface generated from a velocity model satisfying (\ref{EQ:Velocity Model 1}) with $M = 2$, while the right panel shows time series wave signals at the bottom surface generated from a velocity model constructed by (\ref{EQ:Velocity Model 2}) with $M = 4$. From the top to the bottom are time series wave signals without noise, with $10\%$ multiplication Gaussian noise and with $10\%$ additive Gaussian noise, respectively. }}\label{samples_signal}
 \end{figure}

 Last, we note that to obtain a reliable learning dataset, one needs to guarantee the stability of the time integrator when solving \eqref{EQ:Wave Equation}. Recall that the second-order time-domain stagger-grid finite difference forward wave solver is stable under the following CFL condition
 \begin{equation}\label{cfl}
 \Delta t \leq \frac{\min\{\Delta x, \Delta z\}}{\sqrt{2}\max_{\bf x}\{m({\bf x})\}}.
 \end{equation}
 To guarantee the stability of the forward solver for all velocity samples, we force
 \[\Delta t = \Delta t^\ast < \frac{\min\{\Delta x, \Delta z\}}{\sqrt{2}\max_{\bf x}\{\overline{m}({\bf x})\}},\]
where $\overline{m}$ is used in the scaling~\eqref{EQ:Scaling}, for the data generation of the offline training stage. In this work, we set $\Delta t^\ast = 0.0005$ as shown in Table \ref{signal_generate} based on our setting.
 
\subsection{Training and testing performance}

We now present a systematic numerical exploration on the training and testing performance of the offline training stage. Given that the training and application of the Gaussian mixture velocity model~\eqref{EQ:Velocity Model 2} with a small amount of Gaussian functions is extremely successful (due to the smallness of the parameter space) according to our numerical experience, we will focus on the training of the generalized Fourier velocity model~\eqref{EQ:Velocity Model 1}.

\paragraph{Training dataset size.} We first emphasize that the training results we show in this section are obtained on a very small dataset in the following sense. The number of data points in the artificial dataset $\{\bg_j, m_j\}_{j=1}^N$ is small with $N=10^6$. Moreover, for each $m_j$, we collect the wavefield from $N_s=3$ illumination sources and $N_d=51$ detectors. Those source-detector pairs are a subset of the source-detector pairs for the dataset we used in the reconstruction step. Moreover, at each detector, we use only data at $51$ time steps out of the $1000$ time steps in the numerical solutions. This small dataset is used so that we can handle the computational cost of the training process with our limited computing resources. It is also intentionally done to demonstrate that one can train reasonable approximate inverse with a significantly smaller dataset if one is willing to sacrifice a little of the training accuracy. 

\paragraph{Training-testing dataset split.} We perform a standard training-validation cycle on the neural network approximate inverse. Before the training process starts, we randomly split the artificial dataset of $N=10^6$ data points into a training dataset and a testing dataset. The training dataset takes $80\%$ of the original dataset, while the test dataset takes the rest $20\%$ of the data points. The training dataset and the validation dataset have no intersection, namely, no data points in the validation set are present in the training dataset.

\subsubsection{Random Fourier velocity model: case of non-decaying coefficients}

We start with the most challenging scenario where we train the neural network to approximate the inverse operator for the velocity model~\eqref{EQ:Velocity Model 1} with randomly generated Fourier coefficients without any decay requirement on the coefficients, that is, we set the decay rate $\alpha=0$ in~\eqref{EQ:decay_rule}. This is an extremely challenging case because the effective parameter space of this class of velocity models grows exponentially with respect to the number of Fourier models we have in the model. Ideally, one would need an exponentially large training dataset in order to have reasonable training results. However, due to the smooth property of the map $\bff: m\mapsto \bg$, we demonstrate below that with a relatively small dataset and a very limited number of source-detector pairs and time shots, our training result is fairly encouraging.

\begin{figure}[!htb]
	\centering
    \includegraphics[width=0.8\linewidth]{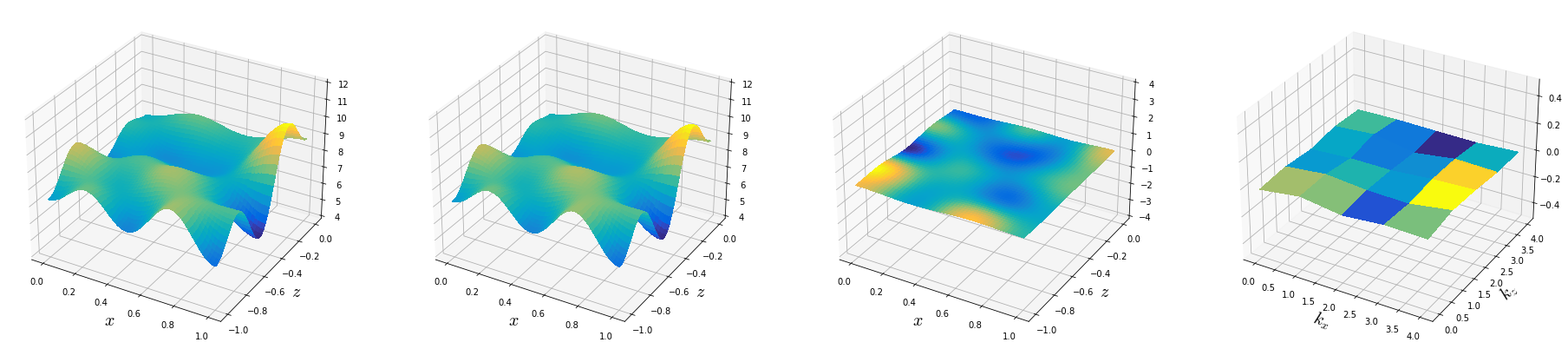}
    \includegraphics[width=0.8\linewidth]{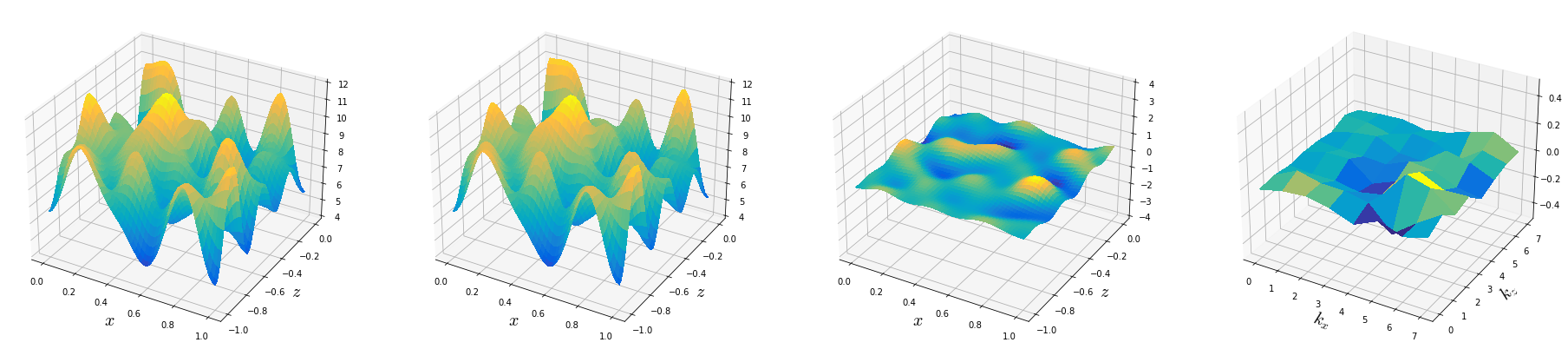}
    \includegraphics[width=0.8\linewidth]{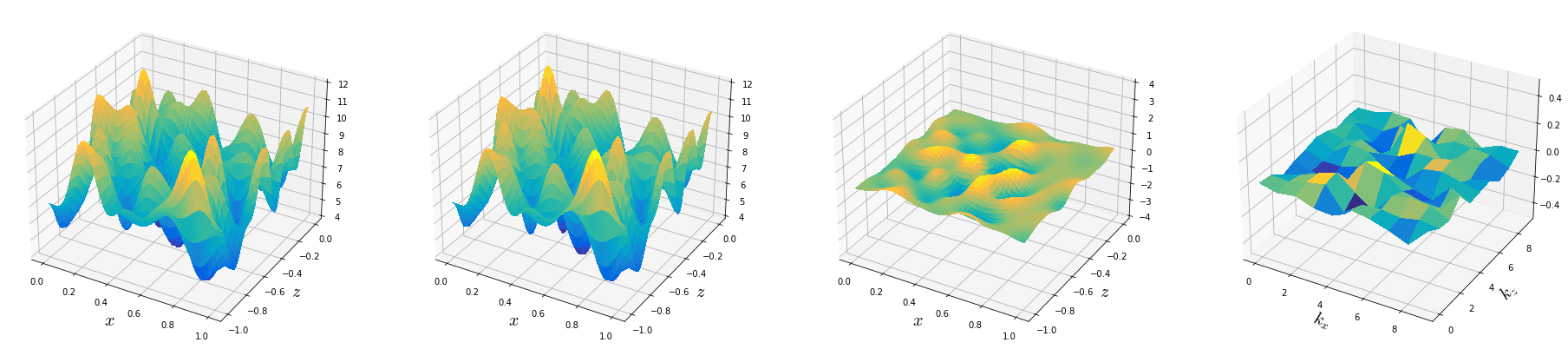}
	\caption{Three randomly selected velocity fields from the testing dataset: $5\times 5$ coefficients Fourier model, $8\times 8$ coefficients Fourier model, $10\times 10$ coefficients Fourier model. All cases have decay rate $\alpha=0$ (column 1), the corresponding predictions by the trained neural network (column 2), the error of the prediction (column 3), and the error in the neural network prediction ($\wt m(\bx)$) in the Fourier domain ($\fm(\bk)-\wt{\fm}(\bk)$) (column 4).}
	\label{FIG:Random Fourier Results}
\end{figure}
In Figure~\ref{FIG:Random Fourier Results}, we show three randomly selected velocity fields ($m$) from the testing dataset, the corresponding neural network predictions ($\wt m=\wh\bff_{\wh\theta}^{-1}(\bff(m))$), and the error in the prediction ($m-\wt m$). The largest number of Fourier modes allowed in these learning processes is $10$, meaning that $0\le k_x, k_z\le 9$ in the velocity model~\eqref{EQ:Velocity Model 1}. The training output is a $10\times 10$ matrix containing the content of $\fm(\bk)$ in~\eqref{EQ:Velocity Model 1}. The output space is therefore $100$-dimensional. A naive visual inspection of the results in Figure~\ref{FIG:Random Fourier Results} shows that the training process is quite successful as the testing errors seem to be pretty reasonable, especially given that our training dataset is fairly small ($0.8\times 10^6$ data points to be precise). While it is expected that when the number of Fourier modes allowed in the velocity model is very large, the validation error will be sufficiently large if we keep the training sample size, we do observe that validation error is quite small in general for cases when less than $10\times 10$ Fourier modes are pursued in the learning process. Increasing computational power would certainly improve training quality.

Let us remark that our training results indeed show that we have better accuracy in learning the low-frequency components of the inverse operator, as we discussed in the previous sections of the work. In the right column of Figure~\ref{FIG:Random Fourier Results}, we provide the Fourier coefficients of the errors in the network prediction. In all velocity fields, we see clearly larger errors in the higher-frequency components of the network velocity recovery. This is a universal phenomenon that we observed over the testing dataset.

\begin{figure}[!htb]
	\centering
	\includegraphics[width=0.6\linewidth]{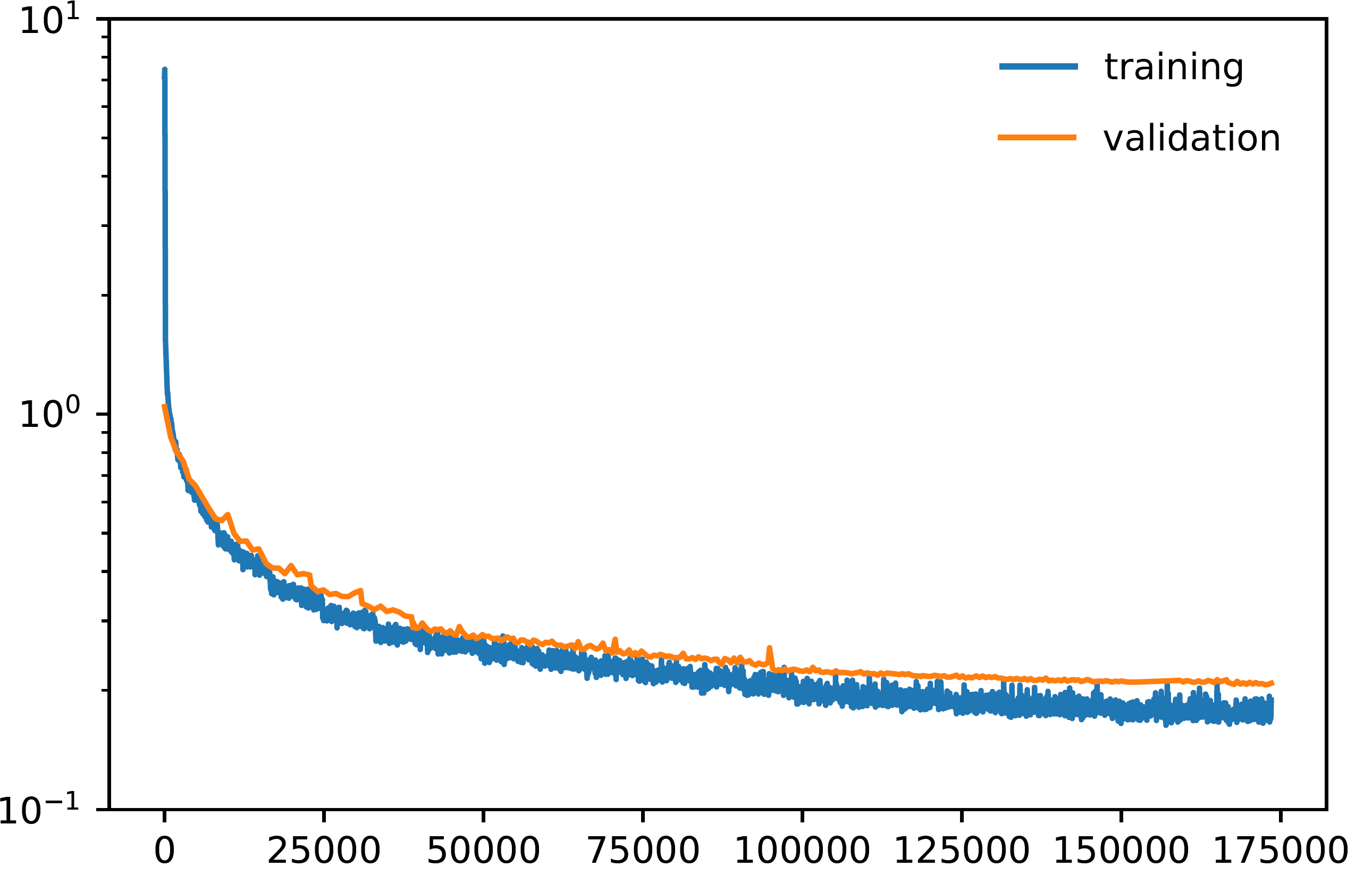}
 	\put(-285,65){\rotatebox{90}{\footnotesize Loss function}}
 	\put(-150,-10){\footnotesize Step $n$}
	\caption{Training and validation loss curves for a typical learning experiment. Very similar curves are observed for each of the learning experiments we performed. }
	\label{fig:training3channel0noise}
\end{figure}
To dive a little more into the training quality and the optimization landscape after applying our neural network preconditioner, we offer in Figure~\ref{fig:training3channel0noise} the training-validation loss curves for a typical learning experiment. We observe very similar curves for training and validating with the velocity model~\eqref{EQ:Velocity Model 1} with different total numbers of Fourier modes. We measure the training accuracy quantitatively with the size of the operator $\cI-\wh{\bff}_{\wh{\theta}}^{-1}\circ \bff$. More precisely, we evaluate the three main quantities for a data point $(\bg, m)$ in the testing dataset:\\[2ex]
(i) The error in the network prediction of Fourier modes of $m$:
\begin{equation*}
    \Delta \fm(\bk) := \fm(\bk)-{\bff}_{\wh{\theta}}^{-1}\circ \bff(m)\,.
\end{equation*}
(ii) The landscape of the classical functional $\Psi(m)$ evaluated along a line in the direction of a given Fourier mode of $m$, $\varphi_\bk$, passing through two different points $m=\bff^{-1}(\bg)$ and $m_{\rm net}=\wh\bff_{\wh\theta}^{-1}(\bg)$:
\begin{equation*}\label{EQ:Psi-New}
\Psi_{\cK}(h;\bk):=\| \bg - \bff(m_b + h\varphi_{\bk})\|^2_{[L^2((0,T]\times \Gamma)]^{N_s}},\ \ m_b = \cK(m),\ \ \ \cK\in \{\cI,\ \wh{\bff}_{\wh{\theta}}^{-1}\circ \bff\}\,.
\end{equation*}
(iii) The landscape under our preconditioner, the new mismatch function $\Phi(m)$ evaluated as in (ii):
\begin{equation*}\label{EQ:Phi-Classical}
    \Phi_{\cK}(h;\bk):=\|\wh{\bff}_{\wh{\theta}}^{-1}(\bg) - \wh{\bff}_{\wh{\theta}}^{-1}\circ \bff(m_b + h\varphi_\bk)\|^2_{L^2(\Omega)},\ \ m_b = \cK(m),\ \ \ \cK\in \{\cI,\ \wh{\bff}_{\wh{\theta}}^{-1}\circ \bff\}\,.
\end{equation*}
When a perfect learning is performed, we would have $\Delta \fm(\bk)=\bzero$ (the zero vector), $\Psi_{\cI}(h;\bk)=\Psi_{\wh{\bff}_{\wh{\theta}}^{-1}\circ \bff}(h; \bk)$, and $\Phi_{\cI}(h;\bk)=\Phi_{\wh{\bff}_{\wh{\theta}}^{-1}\circ \bff}(h; \bk)$ for any $(\bg, m)$ in the training dataset, and $\Delta \fm(\bk)$ small, $\Psi_{\cI}(h;\bk)\approx \Psi_{\wh{\bff}_{\wh{\theta}}^{-1}\circ \bff}(h; \bk)$, and $\Phi_{\cI}(h;\bk)\approx \Phi_{\wh{\bff}_{\wh{\theta}}^{-1}\circ \bff}(h; \bk)$ for any $(\bg, m)$ in the testing dataset.

\begin{figure}[!htb]
	\centering
	\begin{minipage}{0.32\textwidth}
	\includegraphics[width=\linewidth]{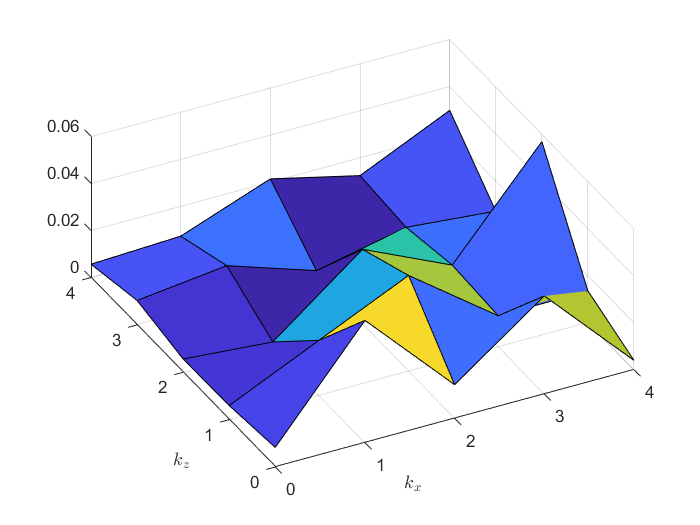}
	\subcaption[]{$\Delta\fm(\bk)$}
	\end{minipage}
	\hfill
	\begin{minipage}{0.32\textwidth}
	\includegraphics[width=\linewidth]{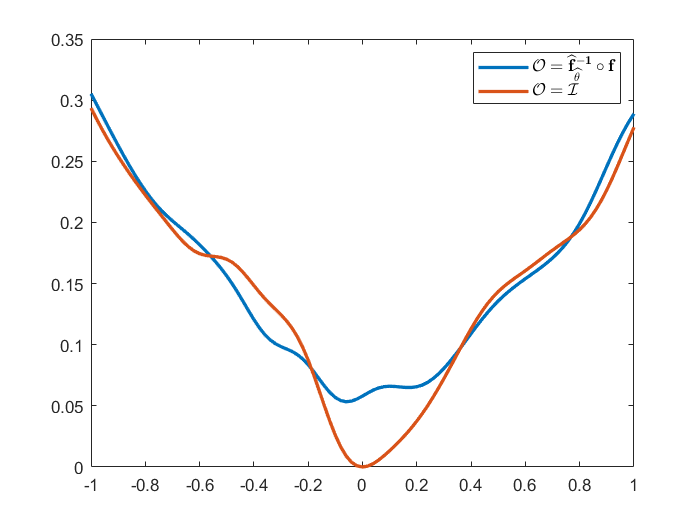}
	\subcaption[]{\scriptsize{$\Psi_{\cK}(h; (2,3))$}}
	\end{minipage}
	\hfill
	\begin{minipage}{0.32\textwidth}
	\includegraphics[width=\linewidth]{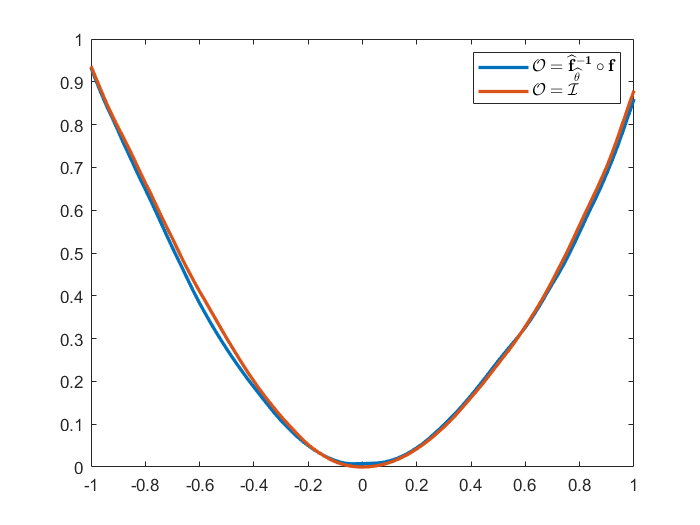}
	\subcaption[]{\scriptsize{$\Phi_{\cK}(h; (2,3))$}}
	\end{minipage}
	\\
	\begin{minipage}{0.32\textwidth}
	    \includegraphics[width=\linewidth]{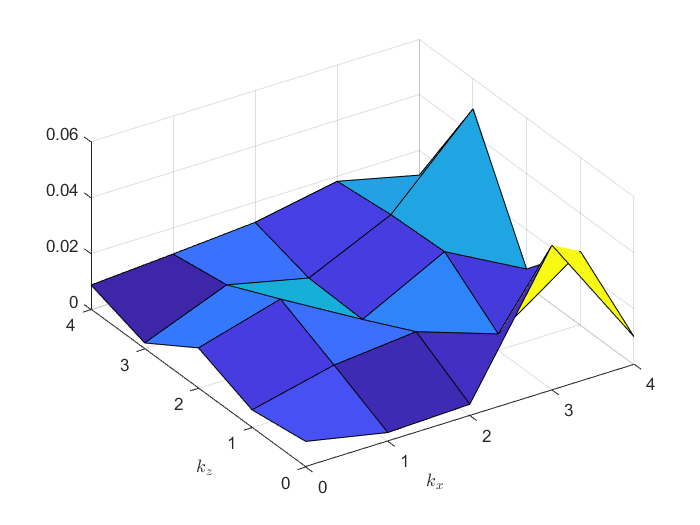}
    	\subcaption[]{$\Delta\fm(\bk)$}
    \end{minipage}
	\hfill
	\begin{minipage}{0.32\textwidth}
    	\includegraphics[width=\linewidth]{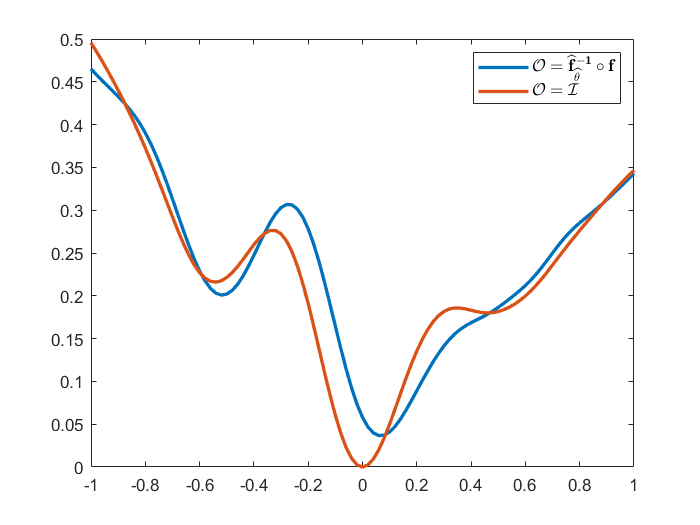}
    	\subcaption[]{\scriptsize{$\Psi_{\cK}(h; (1,1))$}}
    \end{minipage}
	\hfill
	\begin{minipage}{0.32\textwidth}
    	\includegraphics[width=\linewidth]{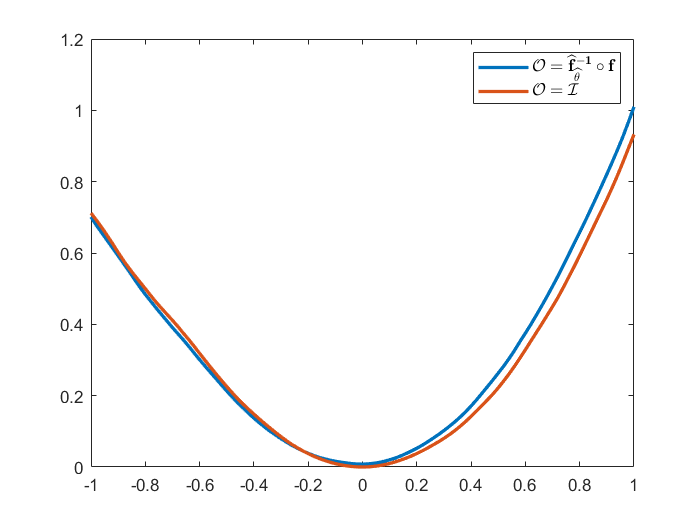}
    	\subcaption[]{\scriptsize{$\Phi_{\cK}(h; (1,1))$}}
	\end{minipage}
	\\
	\begin{minipage}{0.32\textwidth}
	    \includegraphics[width=\linewidth]{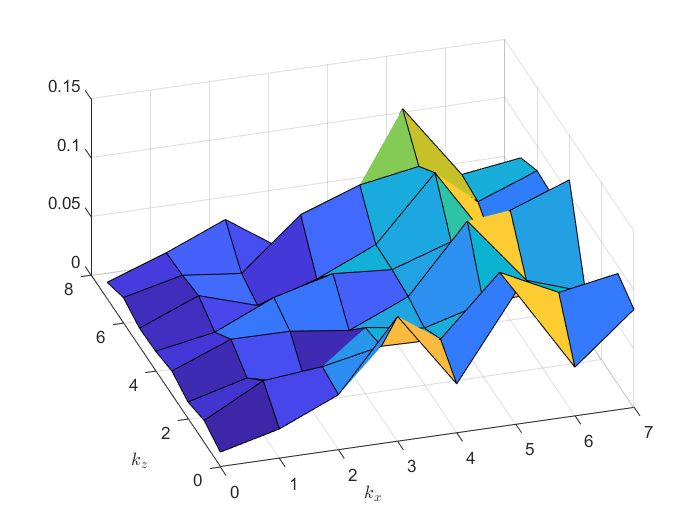}
    	\subcaption[]{$\Delta\fm(\bk)$}
    \end{minipage}
	\hfill
	\begin{minipage}{0.32\textwidth}
    	\includegraphics[width=\linewidth]{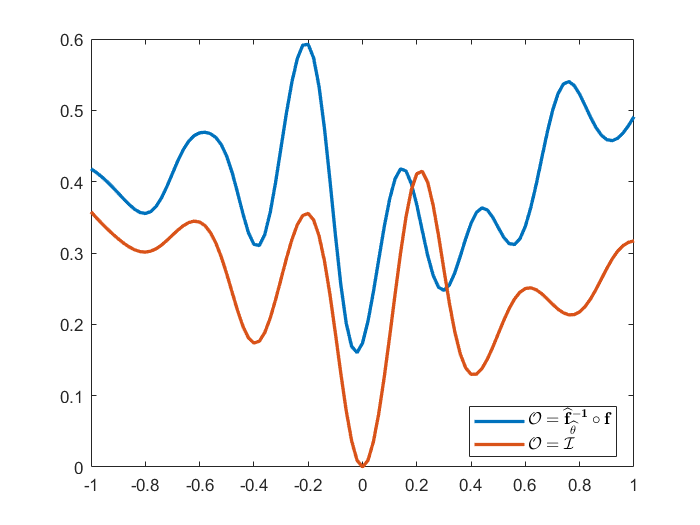}
    	\subcaption[]{\scriptsize{$\Psi_{\cK}(h; (2,3))$}}
    \end{minipage}
	\hfill
	\begin{minipage}{0.32\textwidth}
    	\includegraphics[width=\linewidth]{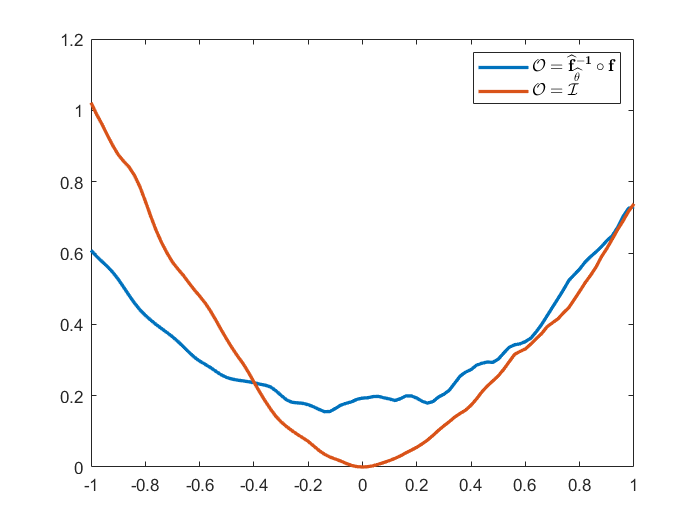}
    	\subcaption[]{\scriptsize{$\Phi_{\cK}(h; (2,3))$}}
	\end{minipage}\\
	\begin{minipage}{0.32\textwidth}
	    \includegraphics[width=\linewidth]{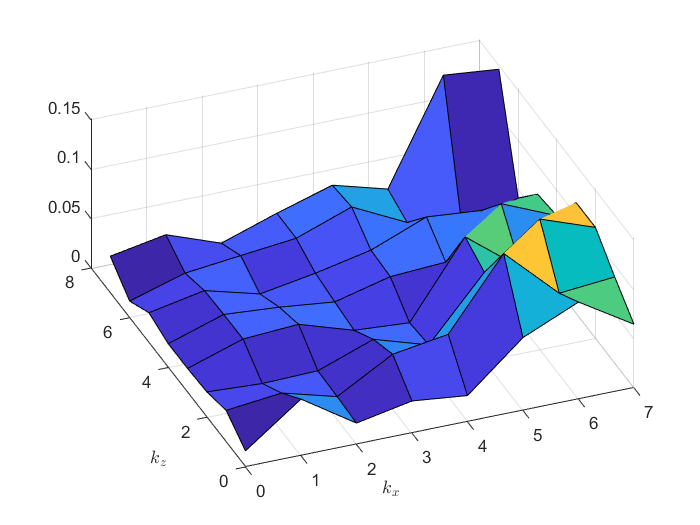}
    	\subcaption[]{$\Delta\fm(\bk)$}
    \end{minipage}
	\hfill
	\begin{minipage}{0.32\textwidth}
    	\includegraphics[width=\linewidth]{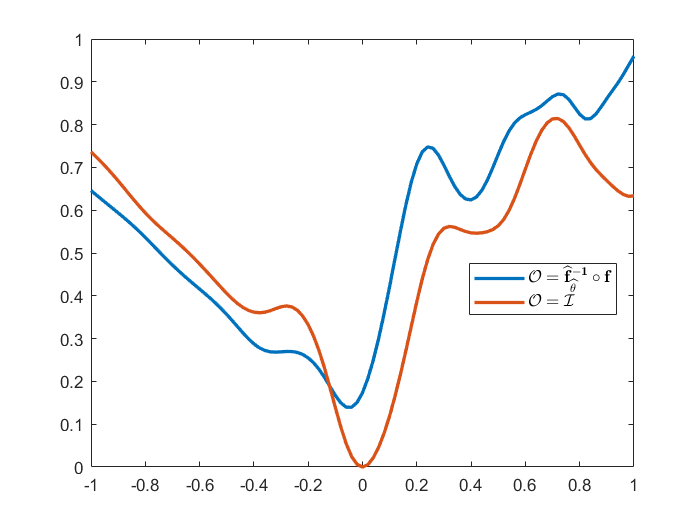}
    	\subcaption[]{\scriptsize{$\Psi_{\cK}(h; (1,1))$}}
    \end{minipage}
	\hfill
	\begin{minipage}{0.32\textwidth}
    	\includegraphics[width=\linewidth]{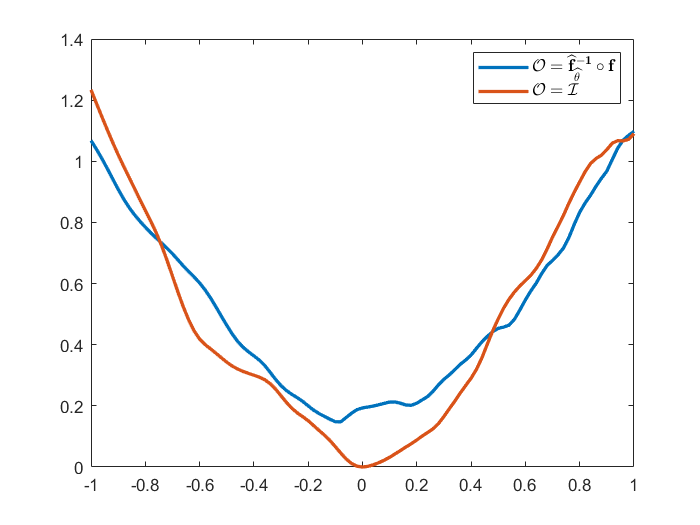}
    	\subcaption[]{\scriptsize{$\Phi_{\cK}(h; (1,1))$}}
	\end{minipage}
	\caption{Plots of $\Delta\fm(\bk)$ (first column), $\Psi_\cK(h;\bk)$ (second column), and $\Phi_\cK(h; \bk)$ (third column) for four different $(\bg, m)$ pairs in the testing dataset. The velocity model for rows 1-2 has $M=4$ and that for the plots in rows 3-4 has $M=7$.}
	\label{FIG:Psi Landscape}
\end{figure}
In Figure~\ref{FIG:Psi Landscape}, we show plots of $\Delta \fm(\bk)$ (left column), $\Psi_{\cI}(h;\bk)$ (red line) and $\Psi_{\wh\bff_{\wh{\theta}}^{-1}\circ \bff}(\bk)$ (blue line) (middle column), and 
$\Phi_{\cI}(h;\bk)$ (red line) and $\Phi_{\wh\bff_{\wh{\theta}}^{-1}\circ \bff}(\bk)$ (blue line) (right column), for four randomly selected $(\bg, m)$ pairs in the testing dataset. Shown are results for $\bk=(2,3)$ and $\bk=(1,1)$. Very similar behavior is observed along other coordinates $\varphi_\bk$.

The plots in Figure~\ref{FIG:Psi Landscape} provide a quantitative description of the accuracy of the trained network. They clearly indicates that the trained $\wh\bff_{\wh{\theta}}^{-1}$ is indeed a good approximation to $\bff^{-1}$. Moreover, a comparison of the second column and the third column gives the impression that along with the coordinates we plotted, the new objective functional $\Phi$ in~\eqref{EQ:Obj} has a much better landscape than the classical $\Psi$ in~\eqref{EQ:Obj Classical}. This is what we observed in other coordinates that are not shown here as well. Therefore, the trained neural network $\wh\bff_{\wh{\theta}}^{-1}$ can work as a nonlinear preconditioner to improve the convexity of the optimization landscape. Moreover, the plots provided a good indication that the trained network is fairly generalizable in the following sense. 
The Fourier coefficients (including $\fm_{(2,3)}$ and $\fm_{(1,1)}$ shown in the plots) in the training dataset are all randomly drawn in the interval $[-0.5, 0.5]$. Here in the plots, we consider the coefficient values in the range $[-1, 1]$. The agreement of the red and blue lines outside of the training value range $[-0.5, 0.5]$, that is, in the range $[-1, -0.5)\cup(0.5, 1]$, suggests that the trained neural network can be used in a region of coefficient values that is far larger than its training domain.
 
\subsubsection{Random Fourier velocity model: case of decaying coefficients}
While training the neural network for an approximate inverse $\wh\bff_{\wh\theta}^{-1}$ on a large space of velocity field is extremely useful for generalization purposes, it also poses great challenges when the number of Fourier modes included in the model gets very large. Not only will we need an exponentially larger training dataset, but also the training process takes exponentially growing computational power. This is what we observed in our numerical experiments. In this section, we show some training-validation results for the velocity model~\eqref{EQ:Velocity Model 1} with decaying Fourier coefficients following the pattern we imposed in~\eqref{EQ:decay_rule}. We present results from two different cases: the slow decay case with $\alpha=1/2$ and the fast decay case with $\alpha=1$. 
\begin{figure}[!htb]
	\centering
	    \includegraphics[width=0.8\linewidth]{8fouriernotdecay}
	    \includegraphics[width=0.8\linewidth]{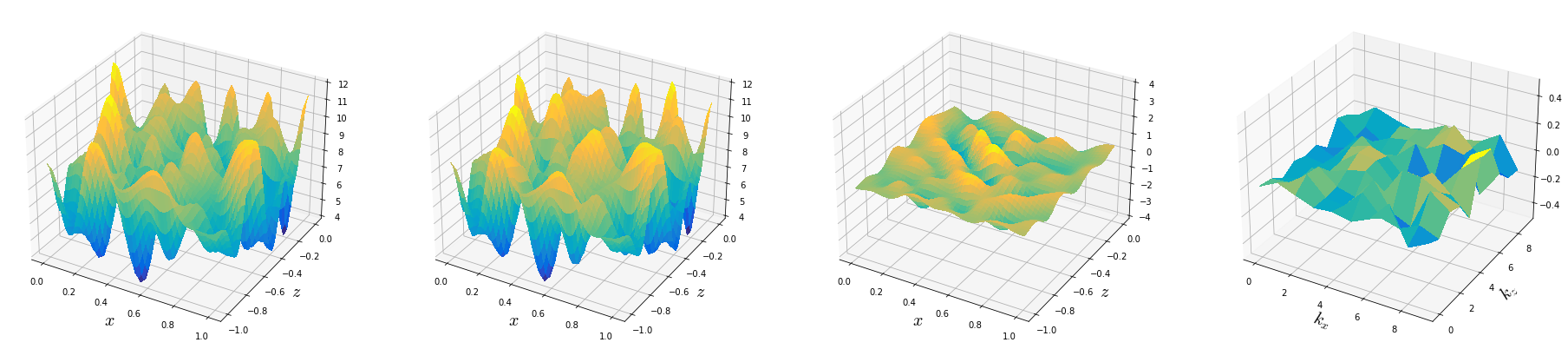}
		\includegraphics[width=0.8\linewidth]{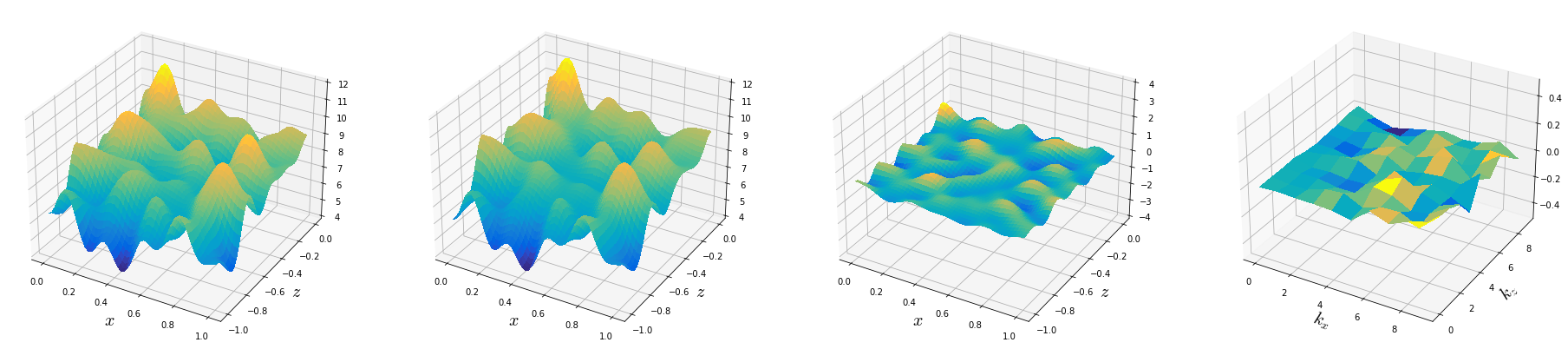} 
		\includegraphics[width=0.8\linewidth]{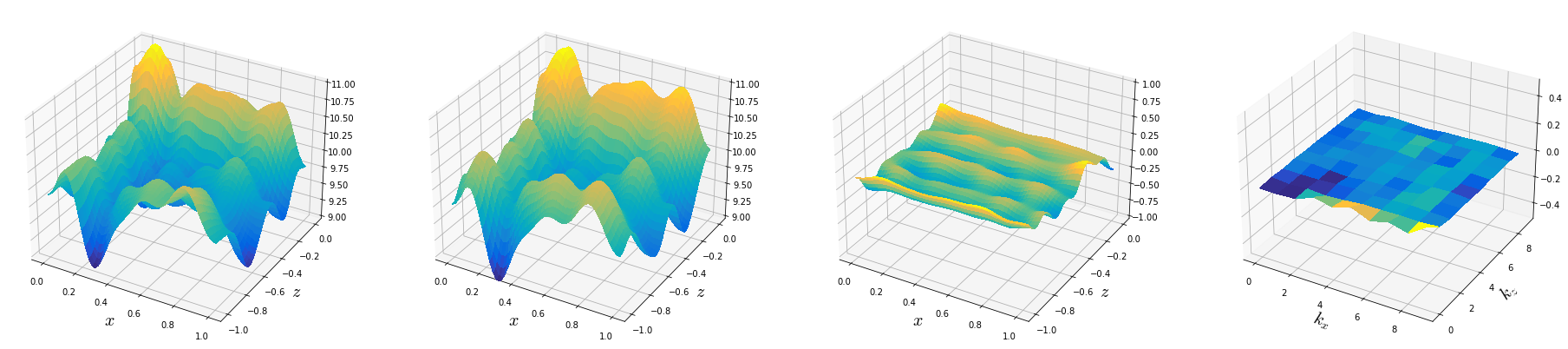}
        \includegraphics[width=0.8\linewidth]{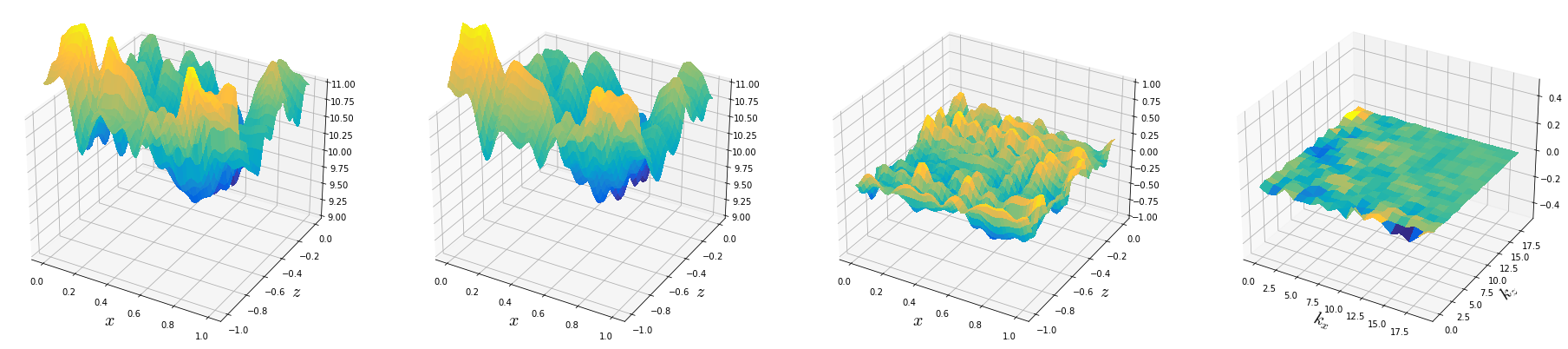}
		\caption{Validation results on four typical velocity fields in the testing dataset. Row 1 is the results for 8 Fourier velocity model with $\alpha=0$, Row 2 is the results for 10 Fourier velocity model with $\alpha=0$, row 3 is the results for 10 Fourier velocity model with $\alpha=1/2$, row 4 is the results for 10 Fourier velocity model with $\alpha=1$ while row 5 is are results for 20 Fourier velocity model with $\alpha=1$. From left to right are: the true velocity field, neural network prediction, the error of the prediction, and the error of the prediction in the Fourier domain.}\label{FIG:Decay Results}
\end{figure}

In Figure~\ref{FIG:Decay Results}, we show typical validation results on five randomly selected velocity profiles in the testing dataset. The top two rows are the results of the training of the velocity model with $\alpha=0$, the third row is the case of $\alpha=1/2$ while the bottom two rows are for the case of $\alpha=1$. In both cases, the training is successful, as can be seen from the relatively small errors in the predictions. Plots of the functionals $\Psi_\cI$ and $\Psi_{\wh\bff_{\wh{\theta}}^{-1}\circ \bff}$ show similar patterns as those in Figure~\ref{FIG:Psi Landscape}. We omit those to space. Moreover, prediction errors in the Fourier domain display very similar behavior as observed in the previous subsection: the error is higher for high-frequency components and lower for low-frequency components.
 
\begin{figure}[!htb]
	\centering
	\includegraphics[width=0.95\linewidth]{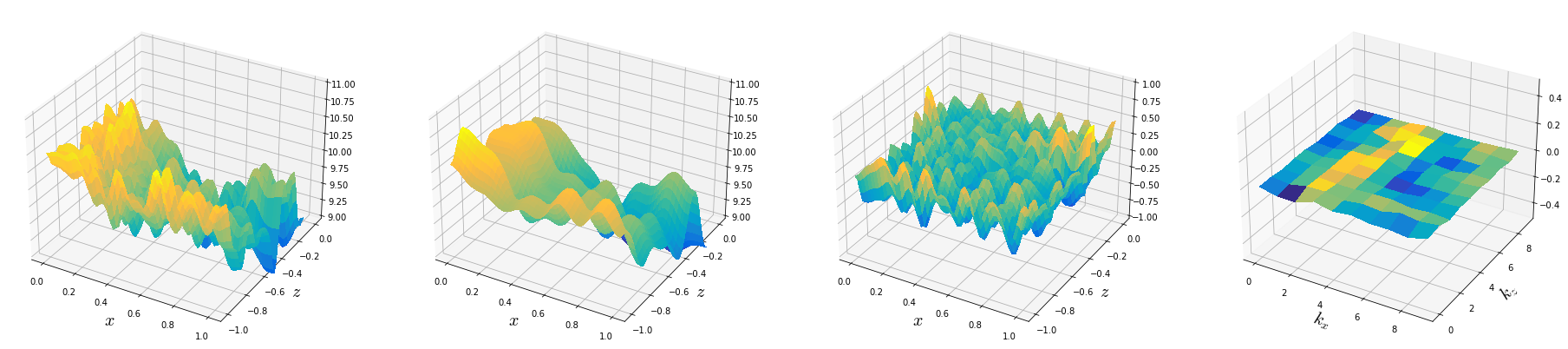}
	\caption{Instance of validation of learning results on a different class of velocity models for the case of $\alpha=1$. Shown from left to right are: the true velocity field, the neural network prediction, the error in the prediction, and the error in the Fourier domain. }
	\label{FIG:Decay Case Generalization}
\end{figure}
To study the generalization capability of the learned network, we validate the learning with on dataset generated from a different velocity model, that is, to consider the case where training and testing data samples are from different classes. In Figure~\ref{FIG:Decay Case Generalization}, we train a neural network to recover the first $10\times 10$ Fourier coefficient of the velocity field and validate the trained neural network on a dataset generated from velocity models that contain $20\times 20$ random Fourier modes. The decay rate in this particular case is $\alpha=1$, but similar results are observed for $\alpha=1/2$ as well. The validation results demonstrate that the trained network is reasonably generalizable in the setting that we considered.

\subsubsection{Mesh-based velocity model}
 
\begin{figure}
	\centering
	\includegraphics[width=0.7\linewidth]{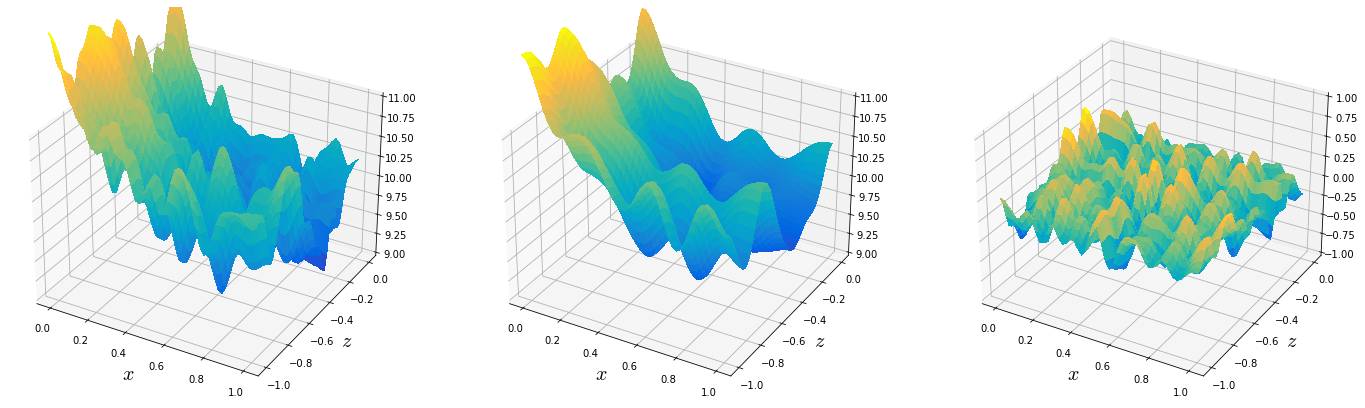}
    \caption{Out-of-domain validation of a training result with mesh-based velocity representation. Shown from left to right are: the true velocity field, the network prediction, and the error in the prediction. }
    \label{FIG:Non-Fourier Validation}
\end{figure}
In the last training-validation numerical experiment, we demonstrate that the phenomena observed in the previous subsections are not particularly due to the Fourier parameterization of the velocity field in~\eqref{EQ:Velocity Model 1} that we used. Indeed, the results are more related to our method of training. Here we perform the same type of training on a neural network whose output is the velocity field represented on a $51\times 51$ uniform mesh on the domain $\Omega$. The output space is therefore much larger compared to the training in the case of the random Fourier velocity model. However, the training result, after projecting into the Fourier space, has almost identical properties to what we observed in the random Fourier model. In Figure~\ref{FIG:Non-Fourier Validation}, we show the out-of-domain validation result for the training. The velocity fields that generated the training dataset have $10\times 10$ Fourier modes while the velocity fields in the validation dataset have $20\times 20$ Fourier modes (but represented on a $51\times 51$ uniform mesh), both generated with $\alpha=1$. The relatively small validation errors again indicate that the training is fairly successful and reasonably generalizable. The computational cost, in this case, is much larger than those in the previous subsections since the neural network has a larger size due to the increased size of the network output.

\subsection{Learning-assisted main feature} 

In this section, we present inversion results for some simulated datasets to verify the efficiency and stability of the proposed coupling method. All simulations on the inversion stage are conducted on a quad-core Intel Core i$7$ with $16$ GB RAM.
 
\subsubsection{Convexity of the new objective function} 

Lemma~\ref{fact:landscape} indicates that if we have relatively accurate training, the new objective function for our coupled reconstruction scheme behaves similarly to the functional $ \|m-m_0\|_{L^2(\Omega)}^2$, $m_0$ being the true solution. Figure~\ref{FIG:Psi Landscape} provided some evidence of this in the training of the random Fourier model. In the one-coefficient case, plots in Figure~\ref{FIG:Psi Landscape} clearly show that the new objective function is almost convex. We now present some numerical evidence in the case of the Gaussian mixture velocity model. In particular, we are interested in seeking convexity with respect to the location of a Gaussian perturbation. More precisely, the velocity field $m({\bf x})$ is set to be a single Gaussian model with $M= 1$ in (\ref{EQ:Velocity Model 2}), that is,
\[
m({\bf x}) = m_0 + c_1e^{-\frac{1}{2} ({\bf x} - {\bf x}_0^1)^\fT \Sigma_1^{-1}({\bf x} - {\bf x}_0^1)},\quad {\bf x}_0^1 = (x_0^1, z_0^1),\quad \Sigma_1 = \begin{pmatrix} \sigma_1^2 & 0\\ 0 & \sigma_1^2 \end{pmatrix}.
\]
where the background velocity $m_0$, the amplitude $c_1$, and the variance $\sigma_1$ are fixed to be $(m_0, c_1, \sigma_1) = (10, 5, 0.1)$. We then present the objective functions $\Psi(m)$ and $\Phi(m)$ ($\gamma=0$) with respect to the location $(x_0^1, z_0^1)$ in Figure \ref{FIG:Landscape Gaussian}. The setting of the offline training stage for generating $\wh{\bf f}_{\wh{\theta}}^{-1}$ is the same as those in Section \ref{SUBSEC:Data Generation}.

Figure~\ref{FIG:Landscape Gaussian} presents the landscapes of objective functions $\Psi(m)$ and $\Phi(m)$ $(\gamma=0)$ with fixed $(m_0, c_1, \sigma_1)$. In particular, we set $(x_0^1, z_0^1) = (0.5, -0.5)$ as the ground true velocity model, which generates the wave signal ${\bf g}$. From Figure~\ref{FIG:Landscape Gaussian}, we observe that (i) the classical objective function $\Psi(m)$ is not a convex function, and its landscape shows that the optimization can be easily trapped into a local minimum if the initial model is not carefully chosen; (ii) the new objective function $\Phi(m)$ ($\gamma=0$) for the proposed coupling method becomes more convex which is consistent with Lemma~\ref{fact:landscape}. In addition, we note that when the initial model is close enough to the exact model (located at the convex region of the misfit function), the global minimum is guaranteed, and one can also expect a fast convergence. In fact, a good initial model under the setting of the proposed coupling scheme can be easily obtained by adding a small perturbation to $\wh{\bf f}_{\wh{\theta}}^{-1}({\bf g})$ as indicated by Neumann series (\ref{EQ:Neumann}).

\begin{figure}[!htb]
	\centering
	\begin{minipage}{0.49\textwidth}
		\centering
		\includegraphics[width=0.48\textwidth]{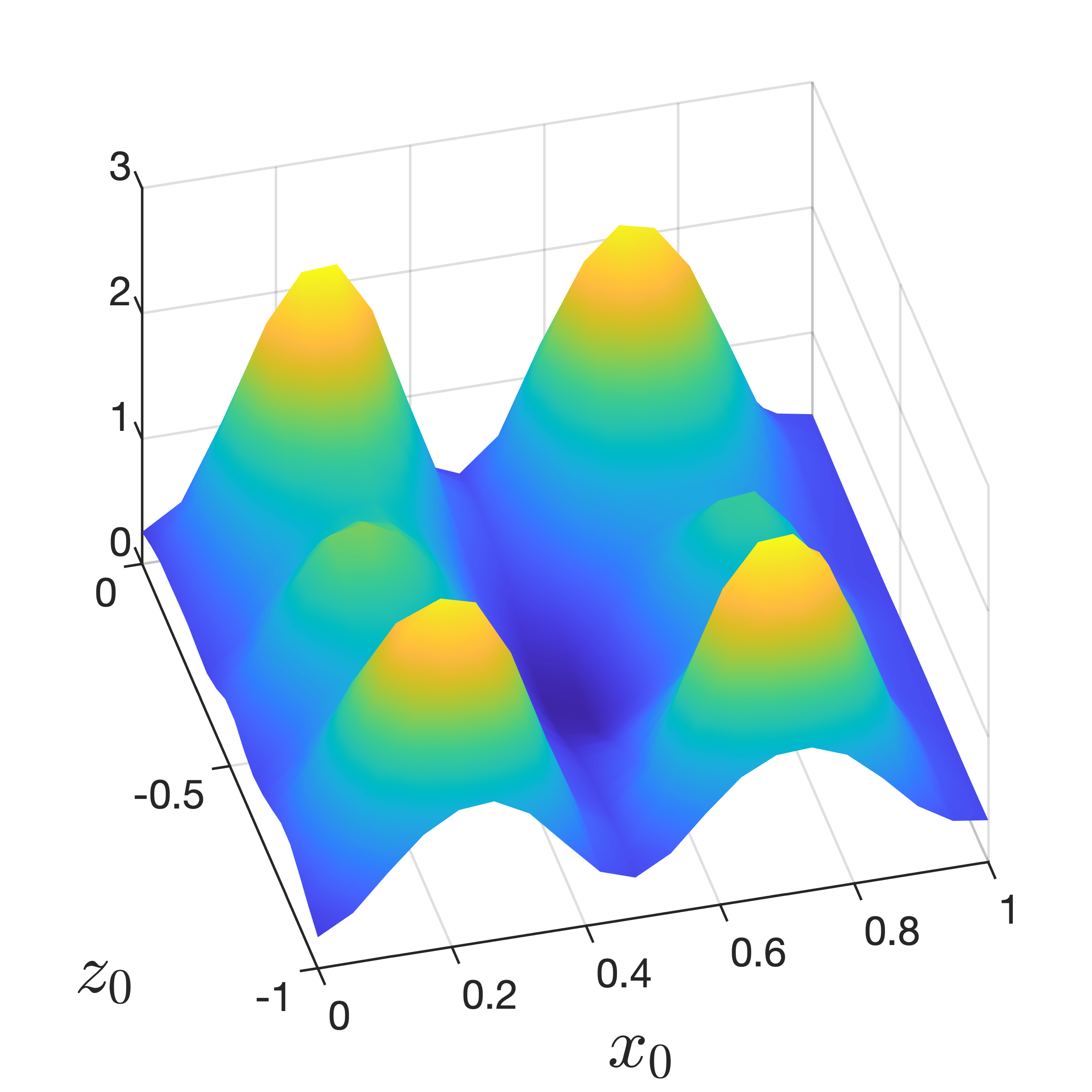}
		\includegraphics[width=0.48\textwidth]{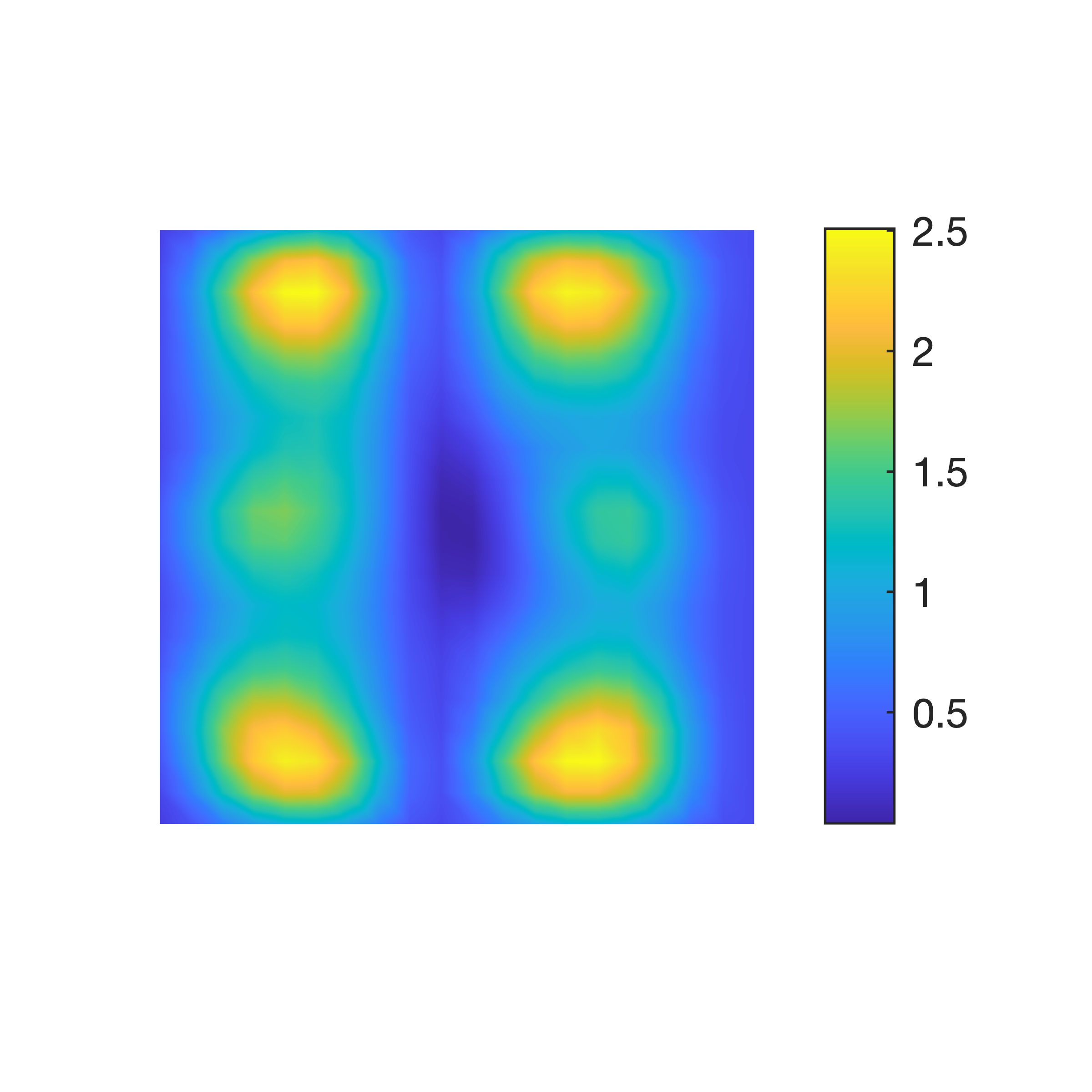} \subcaption[fig:a]{\footnotesize{Landscape of $\Psi(m)$}}
		\label{fig:training}
	\end{minipage}
	\hfill
	\begin{minipage}{0.49\textwidth}
		\centering
		\includegraphics[width=0.48\linewidth]{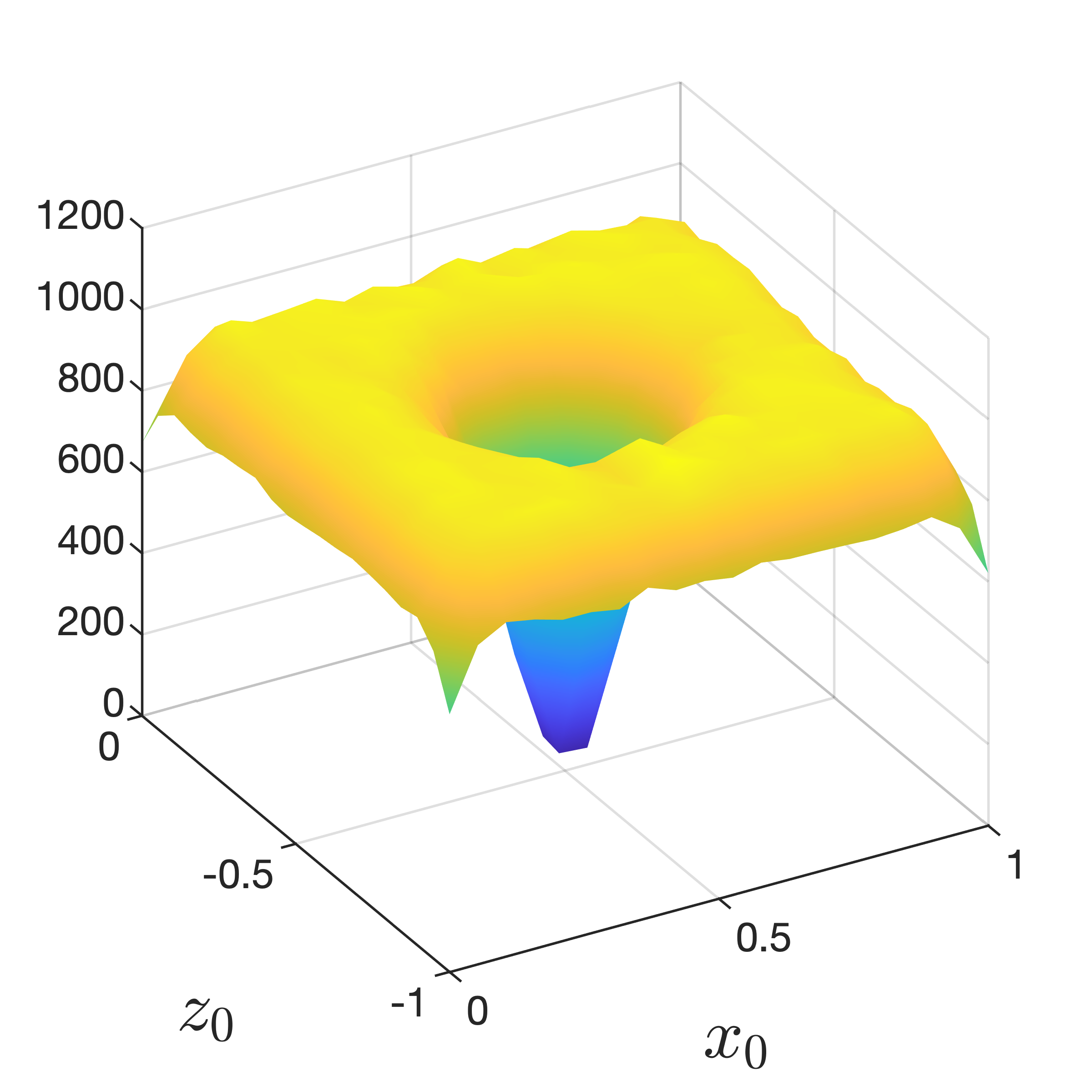}
		\includegraphics[width=0.48\linewidth]{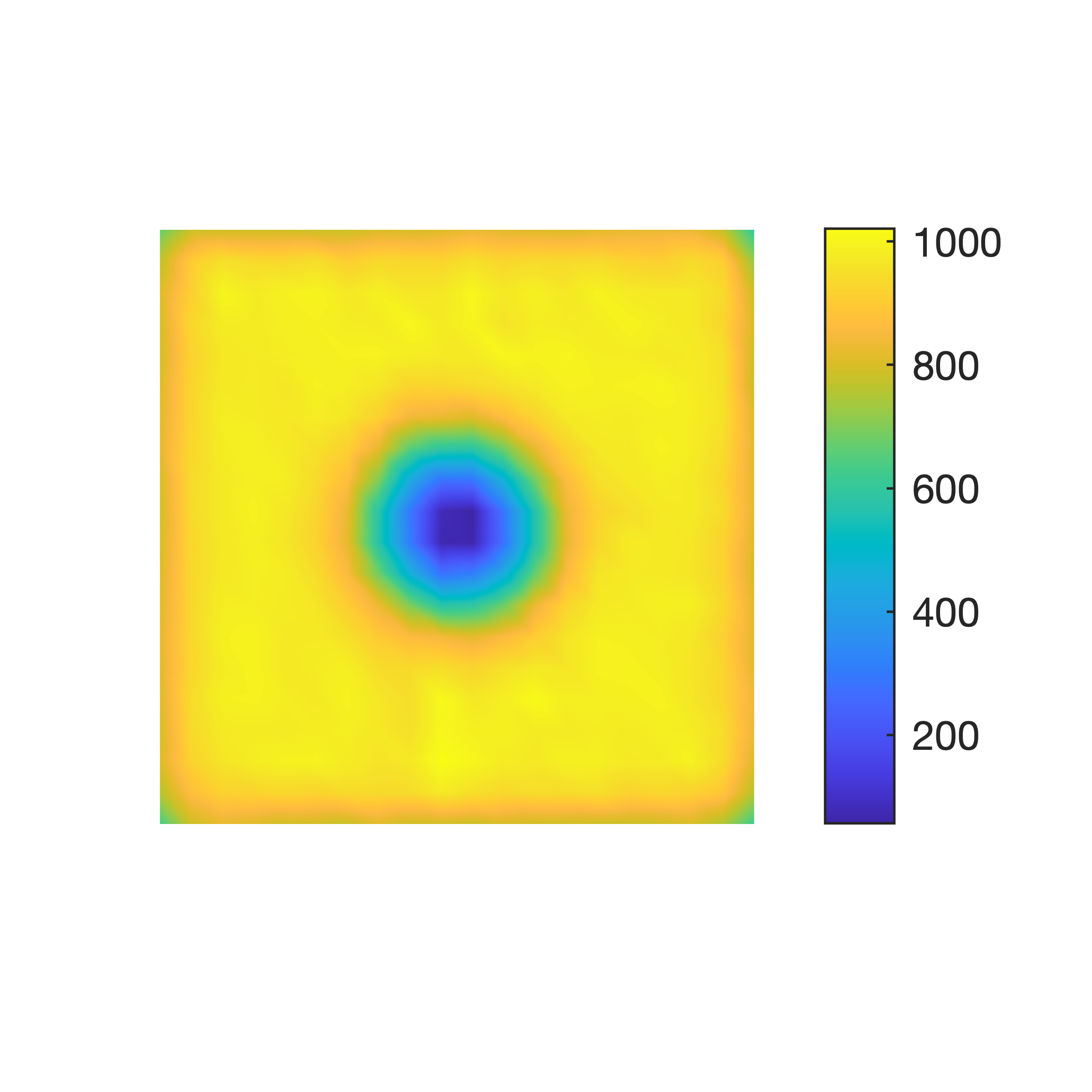}
		\subcaption[fig:b]{\footnotesize{Landscape of $\Phi(m)$}}
		\label{fig:prediction}
	\end{minipage}
	\caption{\small{The landscape of the classical (left) and new (right) objective functions for the location of a Gaussian perturbation of the velocity field. }}
	\label{FIG:Landscape Gaussian}
\end{figure}

\subsubsection{Inversion for the velocity model \texorpdfstring{\eqref{EQ:Velocity Model 2}}{} with \texorpdfstring{$M = 2$}{}}
\label{Sec:two_gaussian}

The first inversion example was performed to recover the following mixed Gaussian velocity model (\ref{EQ:Velocity Model 2}) with $M = 2$ and $m_0 = 10$,
\begin{equation}\label{model1_M2}
m({\bf x}) = 10 + \sum_{k=1}^2c_ke^{-\frac{1}{2} ({\bf x} - {\bf x}_0^k)^T \Sigma_k^{-1}({\bf x} - {\bf x}_0^k)},\quad {\bf x}_0^k = (x_0^k, z_0^k),\quad \Sigma_k = \begin{pmatrix}
\sigma_k^2 & 0\\
0 & \sigma_k^2
\end{pmatrix}.
\end{equation}
 We use the same offline training settings as those for the mixed Gaussian wave signal generation in Section \ref{SUBSEC:Data Generation} to construct $\wh{\bf f}_{\wh{\theta}}^{-1}$ for the online inversion stage. However, to generate versatile wave signals at the bottom surface to recover the features $\{c_1,c_2,\sigma_1,\sigma_2,x_0^1,x_0^2,z_0^1,z_0^2\}$, we enforce three different top sources $h_i(x), i = 1,2,3$ with
\[
h_1(x) = e^{-\frac{(x-0.8)^2}{0.01}} -e^{-\frac{(x-0.2)^2}{0.01}}, \quad
h_2(x) = e^{-\frac{(x-0.4)^2}{0.01}} -e^{-\frac{(x-0.7)^2}{0.01}}, 
\]
and
\[
h_3(x) = e^{-\frac{(x-0.6)^2}{0.01}} -e^{-\frac{(x-0.3)^2}{0.01}},
\]
rather than one single external top source in Section \ref{SUBSEC:Data Generation}.
 
For the inversion stage, we implement a $J$-term truncated Neumann series approximation~\eqref{EQ:Neumann J} to obtain the reconstructed velocity image. Note that $J = 1$ corresponds to the reconstructed velocity image from the offline training stage. We also add the Gaussian noise with zero mean and $10\%$ standard deviation to test the stability of the proposed coupling scheme. Figure \ref{two_gaussian} presents the reconstructed images. Precisely, the first three columns show the surface plots of the exact velocity field, the neural network prediction velocity field from the offline training stage, and the reconstructed velocity field with $J = 20$ from the online inversion stage, while the last column displays the difference between the exact velocity field (first column) and the reconstructed velocity field (third column). From the top row to the bottom row of Figure \ref{two_gaussian}, we present the results from the noise-free wave signal, the wave signal with $10\%$ multiplication Gaussian noise, and the wave signal with $10\%$ additive Gaussian noise, respectively. We see that the online inversion stage improves the accuracy of the reconstructions for all cases. Table \ref{model1_M2_errors} lists the $L^2/L^\infty$ errors on the velocity field for the entire computational domain, as well as the CPU time for various implementations with different values of $J$. There, we note that for the wave signals without noise and with $10\%$ multiplication Gaussian noise, both $L^2$ and $L^\infty$ reconstruction errors dropped by a factor $\sim 10^4$ within $30$ seconds; for the wave signal with $10\%$ additive Gaussian noise, it seems that there is no improvement to add more Neumann terms in (\ref{EQ:Neumann J}) once the $L^2$ error reduces to $5.89\times 10^{-3}$ and $L^\infty$ error reduces to $4.28\times 10^{-2}$; for this type of the situation, we can use the reconstruction from adding Neumann terms as an initial guess for a gradient-based optimization scheme to further improve the resolution of the reconstruction, see Section~\ref{SEC:algorithms}. 

\begin{figure}[!htb]
	\centering
		\includegraphics[width=0.24\textwidth,trim=0cm 1.5cm 0cm 0.5cm,clip]{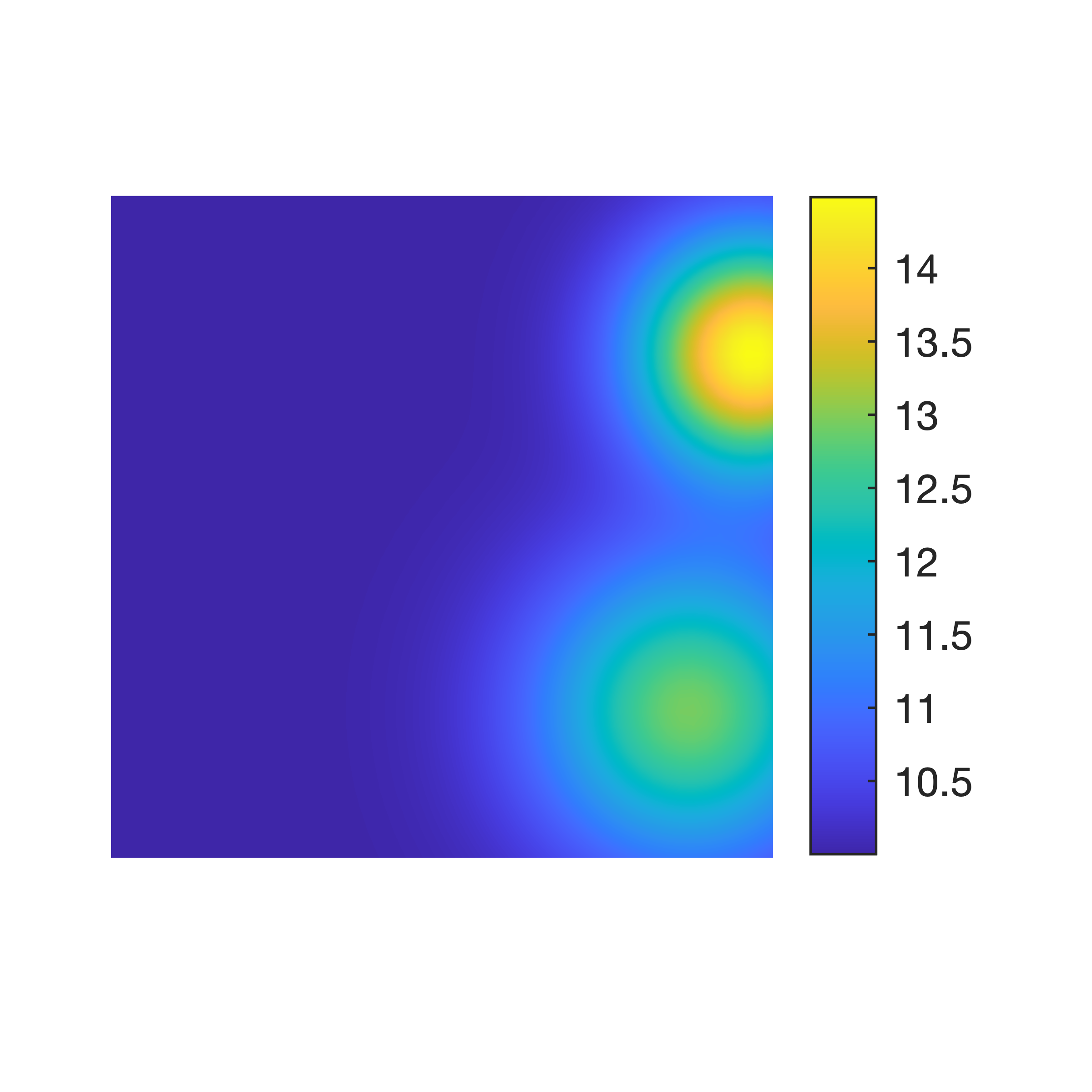}
		\includegraphics[width=0.24\textwidth,trim=0cm 1.5cm 0cm 0.5cm,clip]{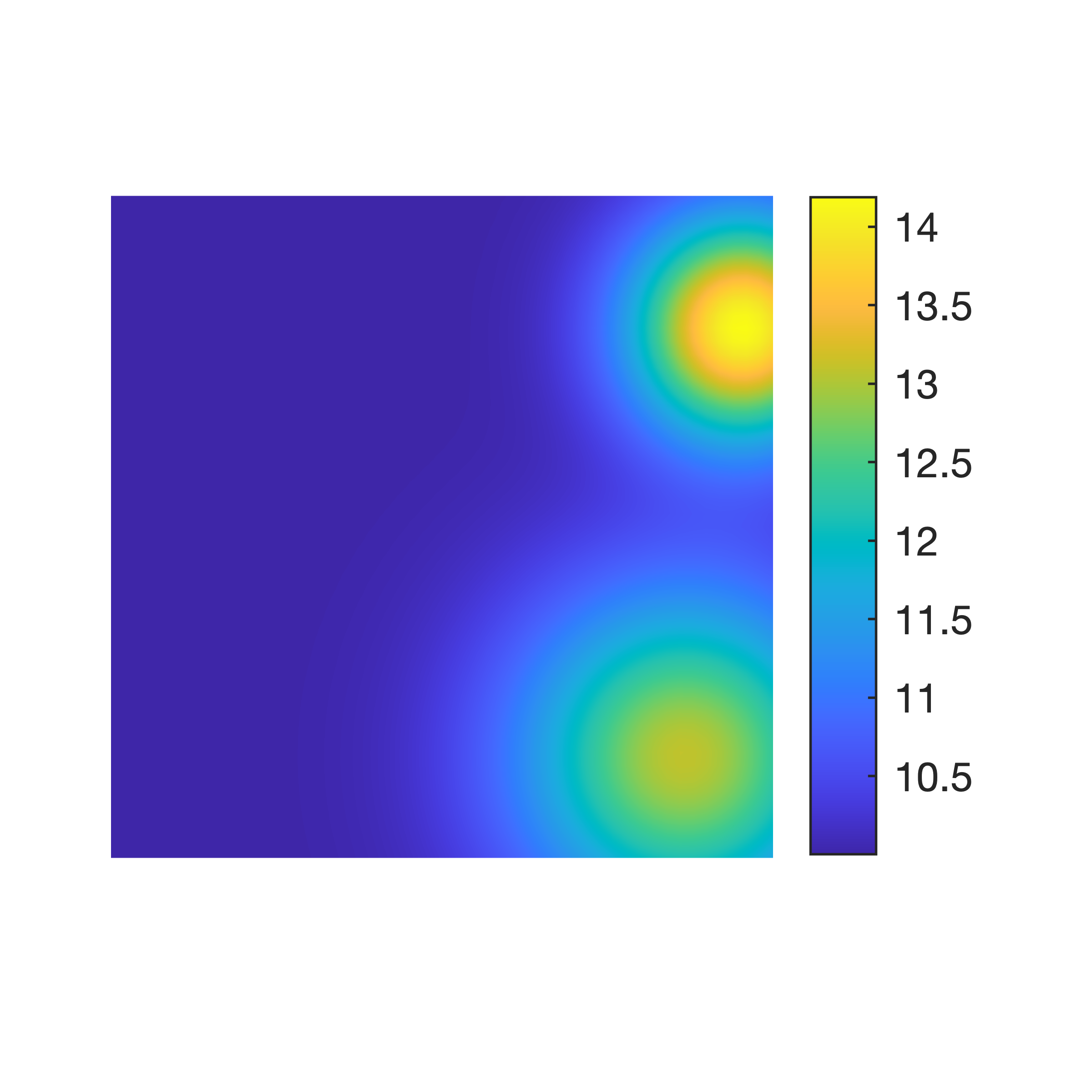} 
		\includegraphics[width=0.24\textwidth,trim=0cm 1.5cm 0cm 0.5cm,clip]{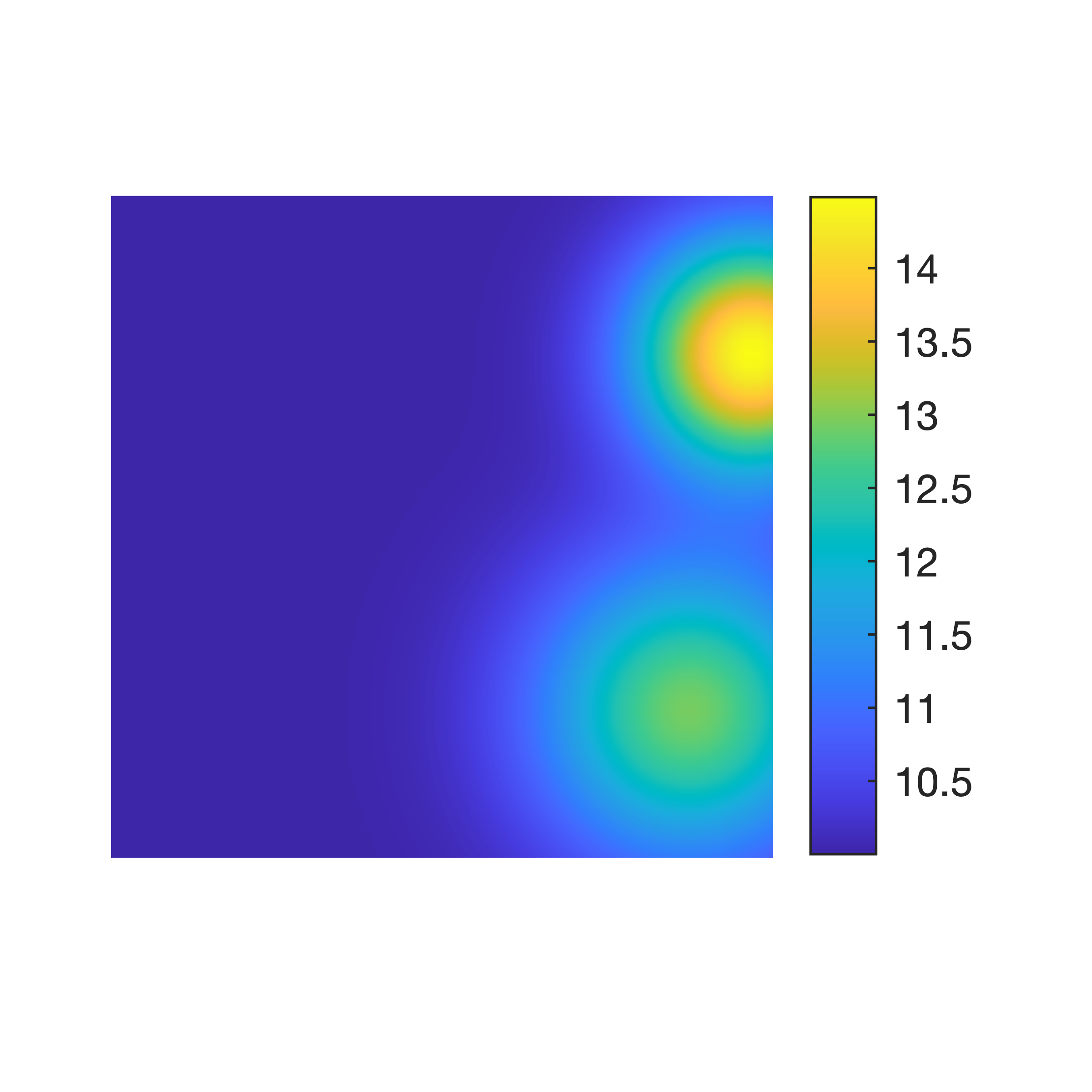}
		\includegraphics[width=0.24\textwidth,trim=0cm 1.5cm 0cm 0.5cm,clip]{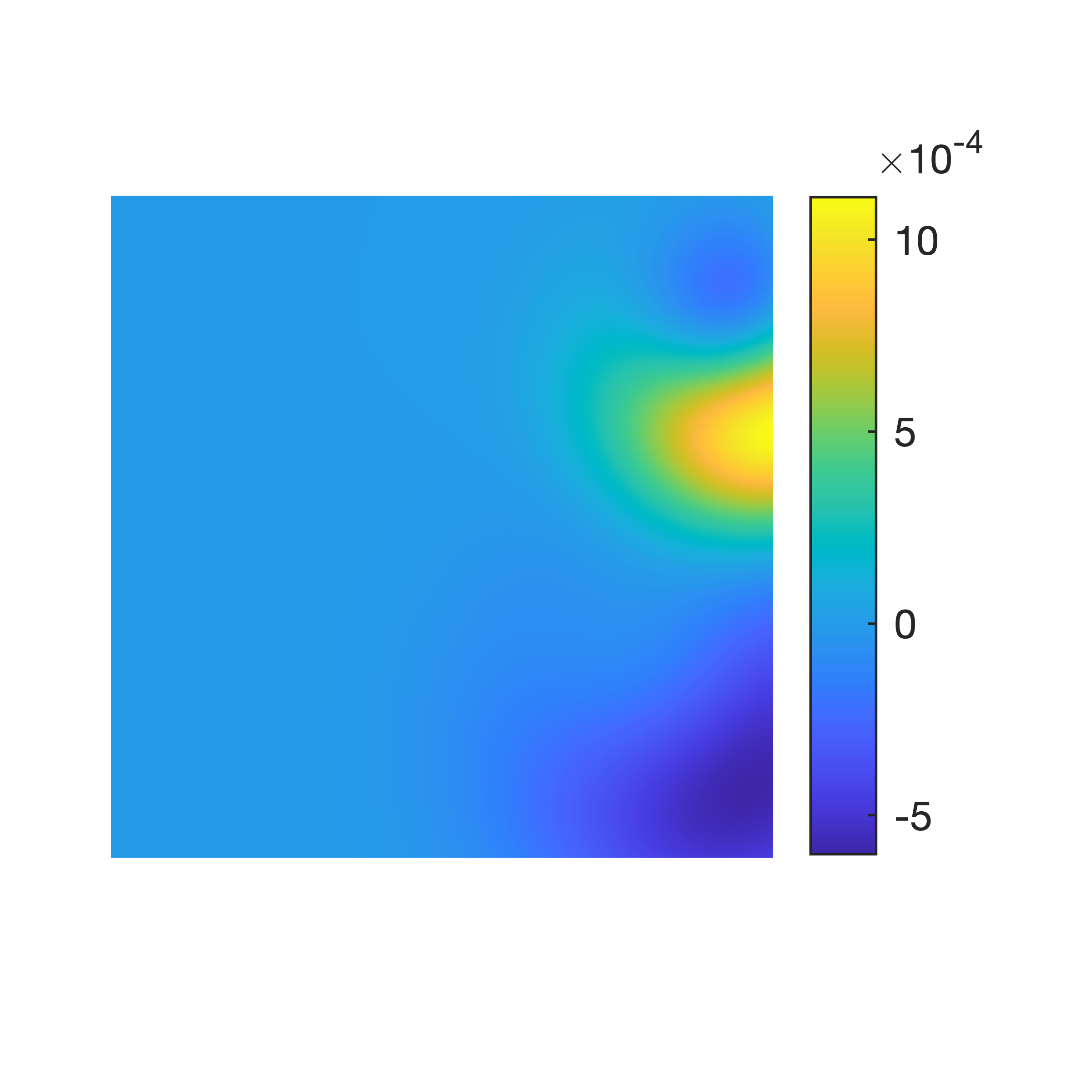} \\
		\includegraphics[width=0.24\textwidth,trim=0cm 1.5cm 0cm 0.5cm,clip]{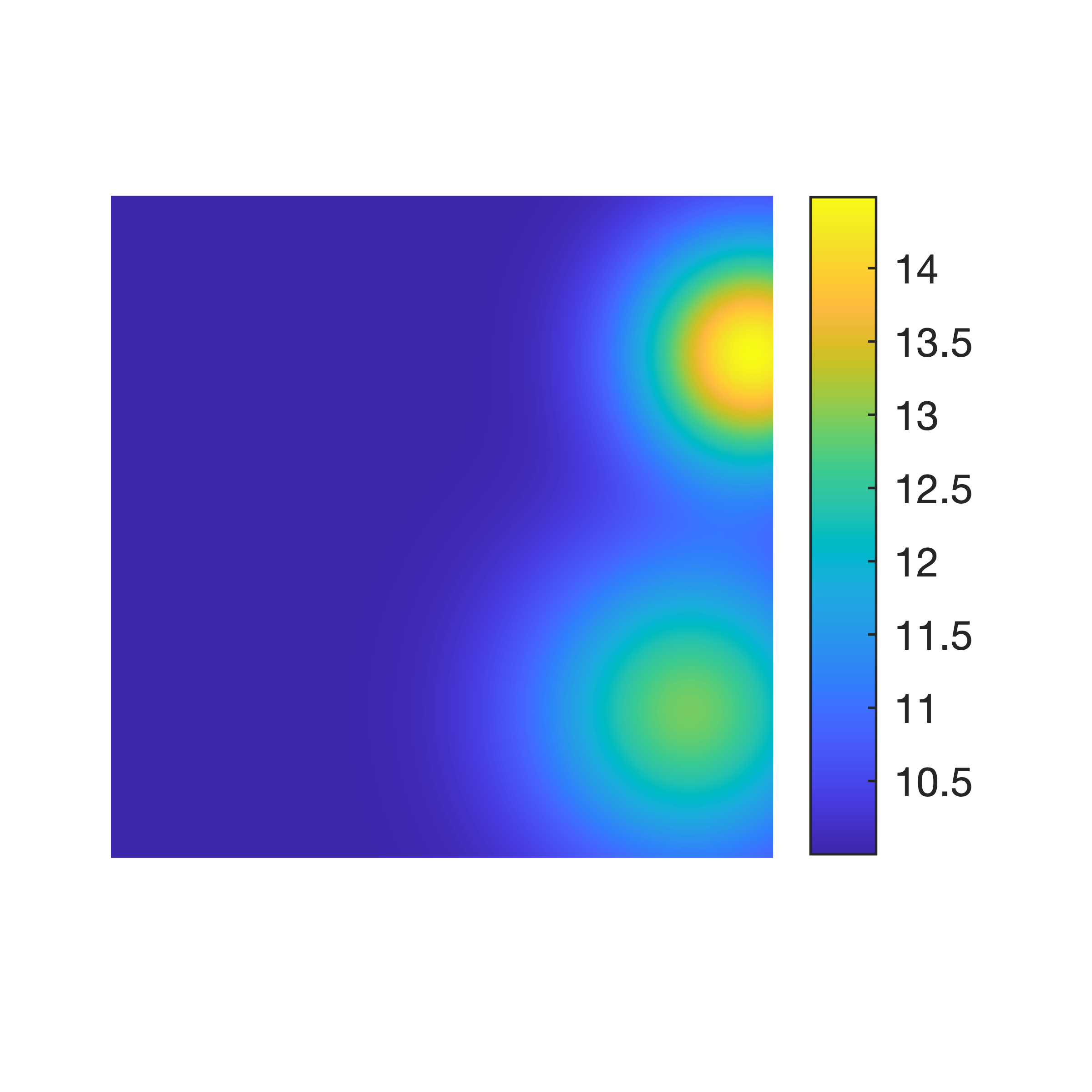}
		\includegraphics[width=0.24\textwidth,trim=0cm 1.5cm 0cm 0.5cm,clip]{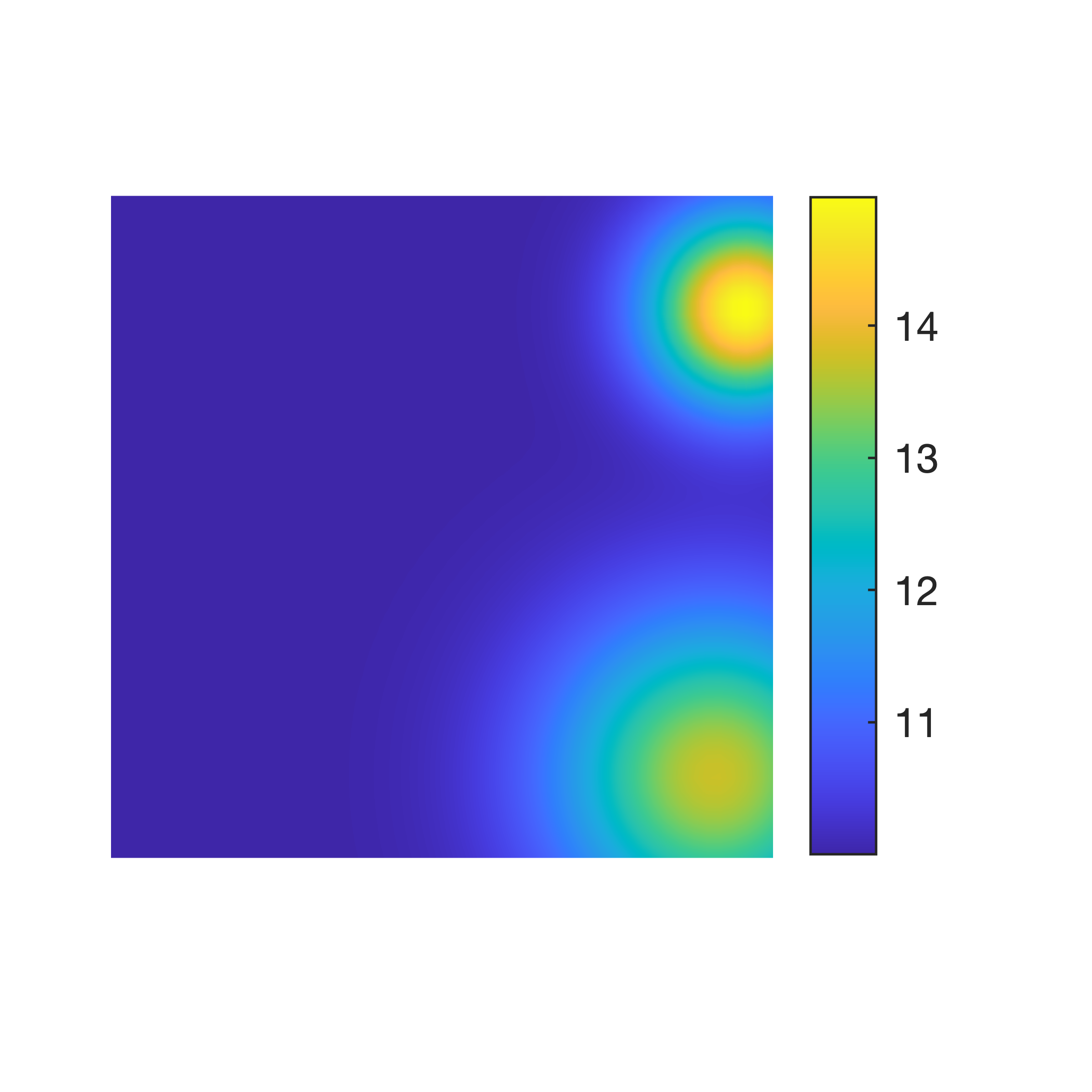} 
		\includegraphics[width=0.24\textwidth,trim=0cm 1.5cm 0cm 0.5cm,clip]{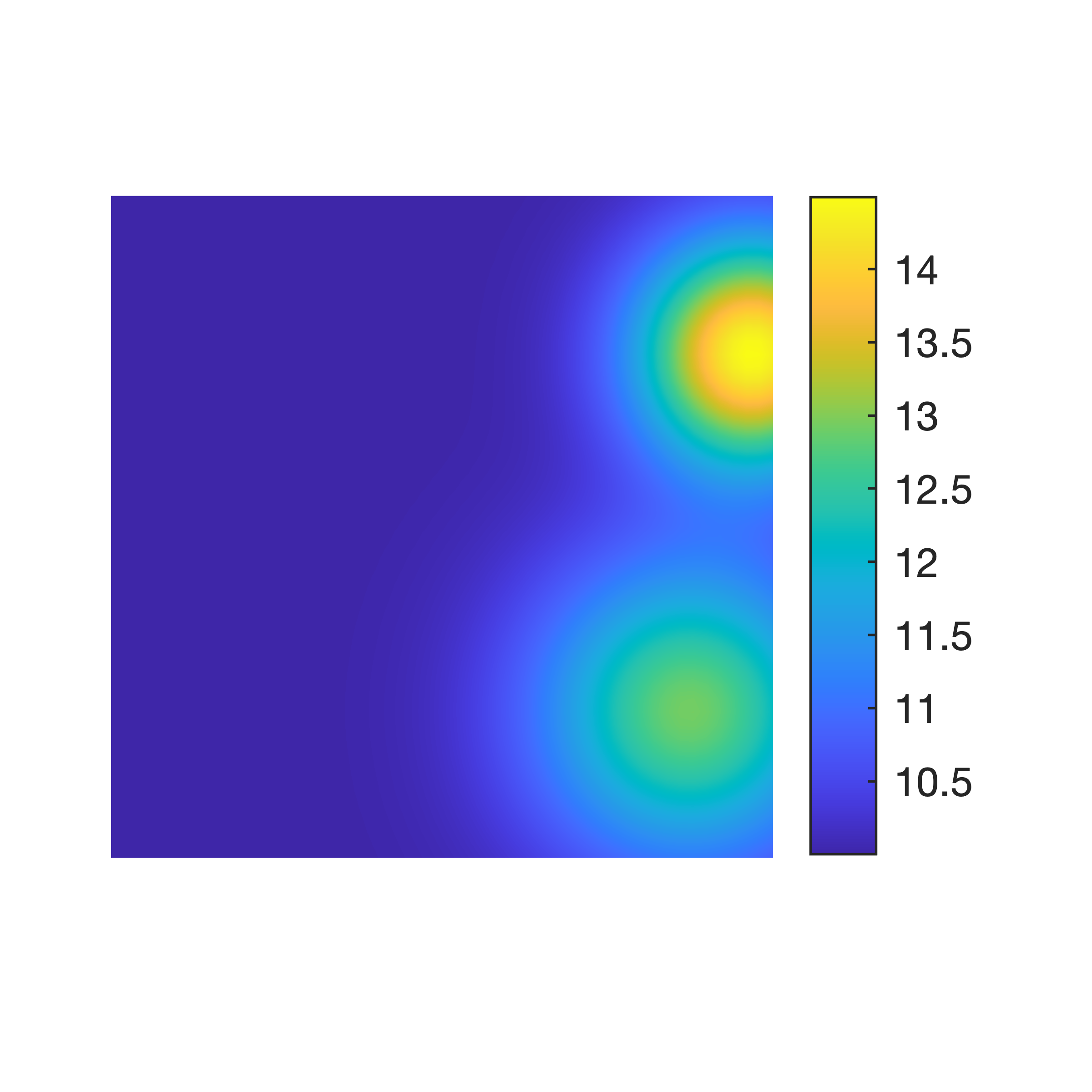}
		\includegraphics[width=0.24\textwidth,trim=0cm 1.5cm 0cm 0.5cm,clip]{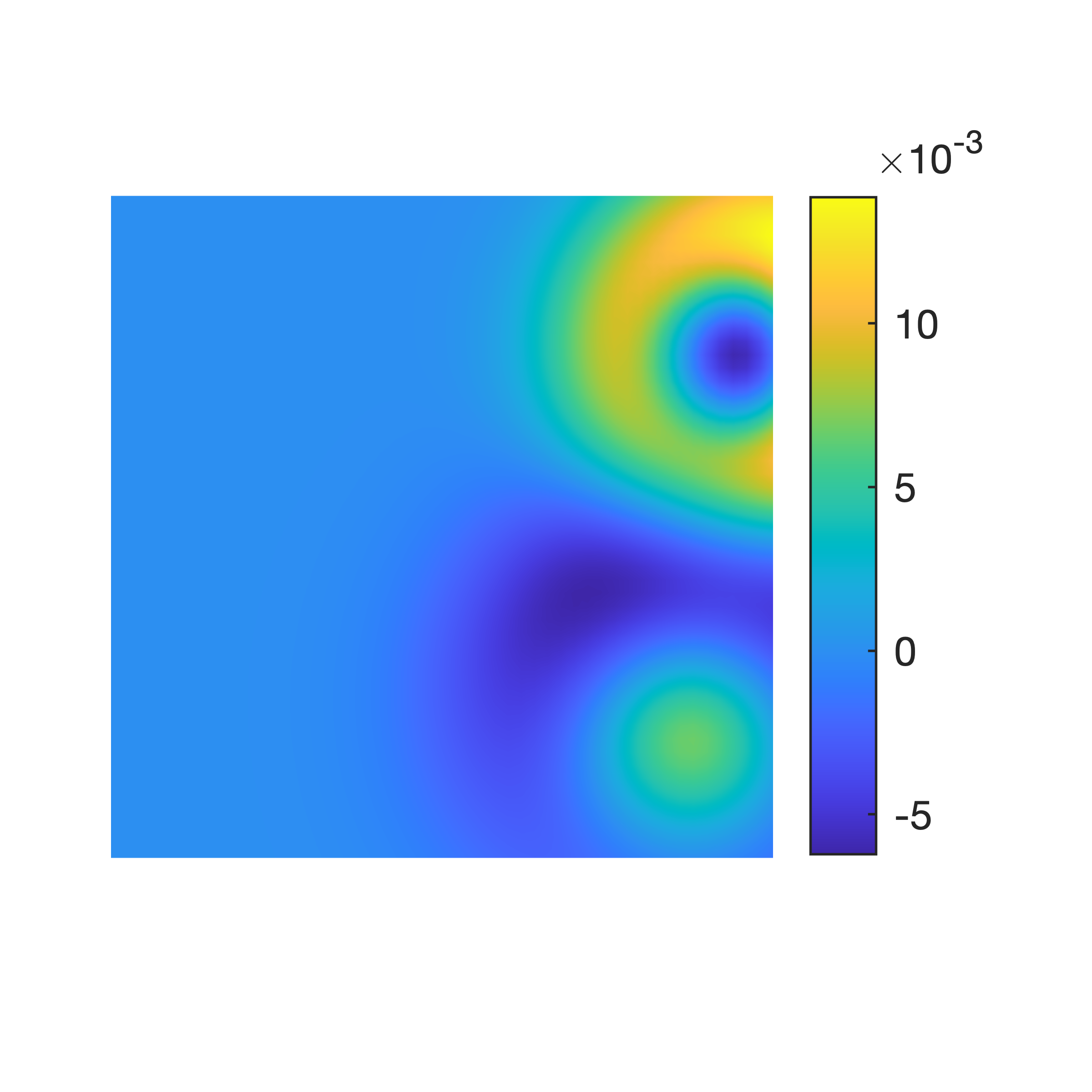} \\
			\includegraphics[width=0.24\textwidth,trim=0cm 1.5cm 0cm 0.5cm,clip]{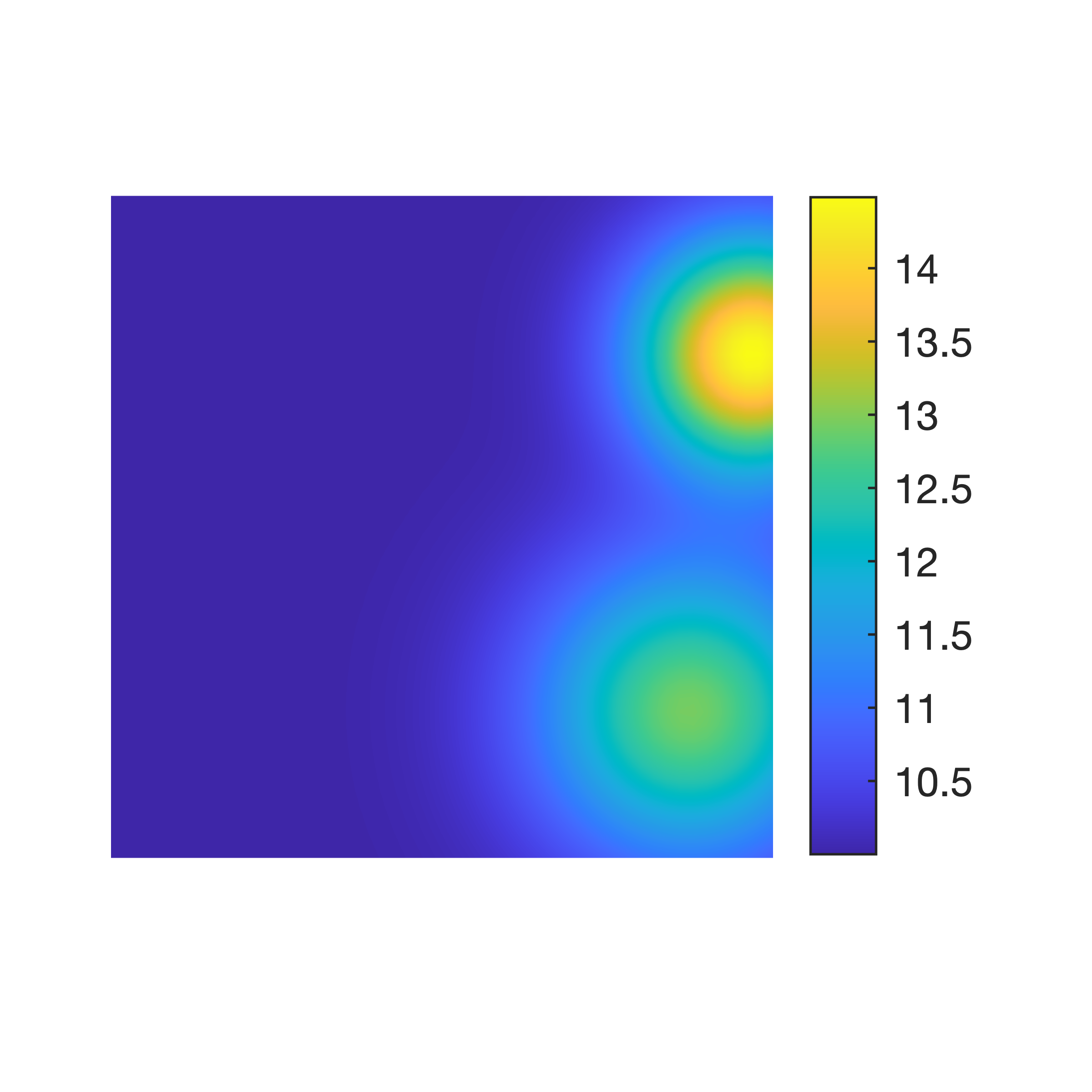}
		\includegraphics[width=0.24\textwidth,trim=0cm 1.5cm 0cm 0.5cm,clip]{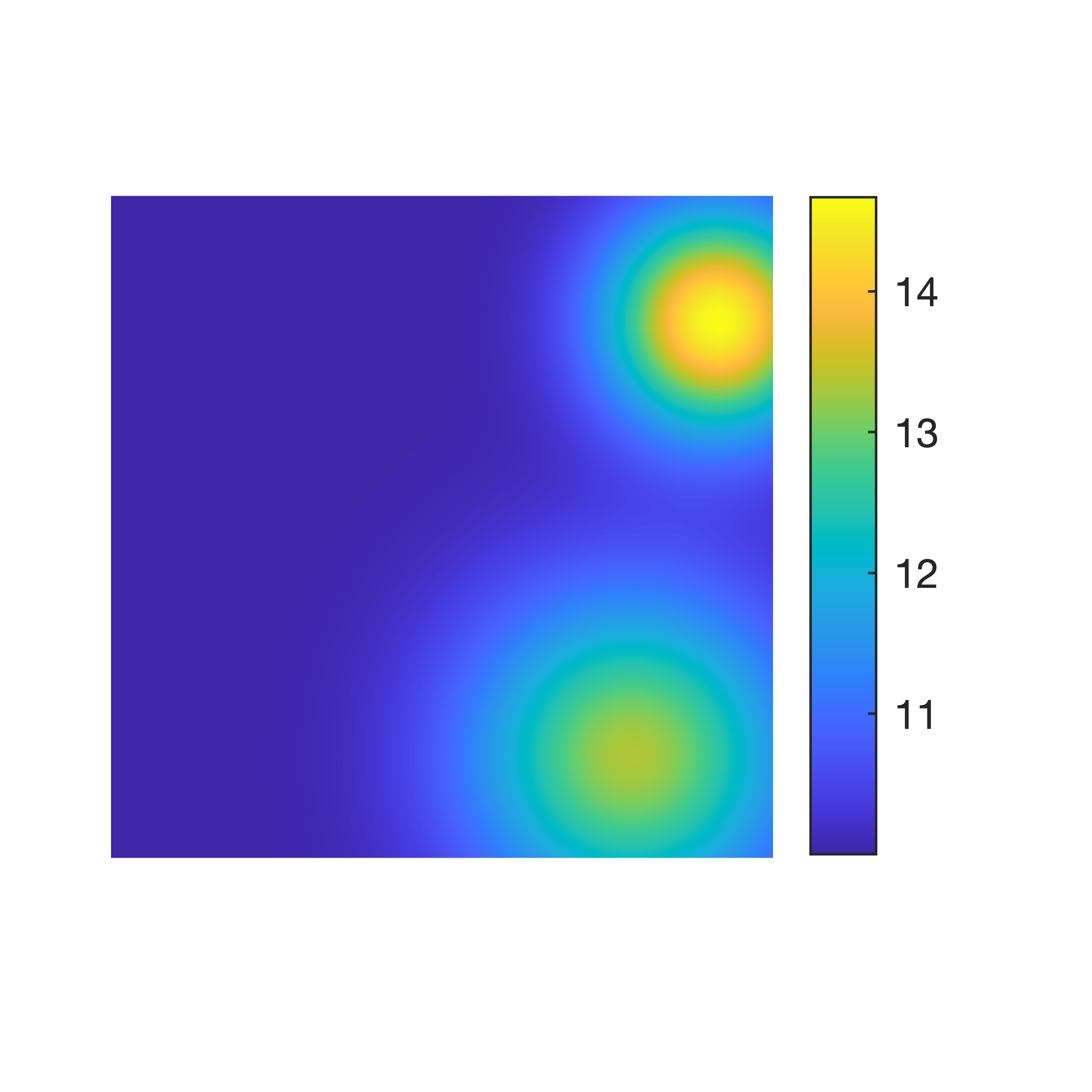} 
		\includegraphics[width=0.24\textwidth,trim=0cm 1.5cm 0cm 0.5cm,clip]{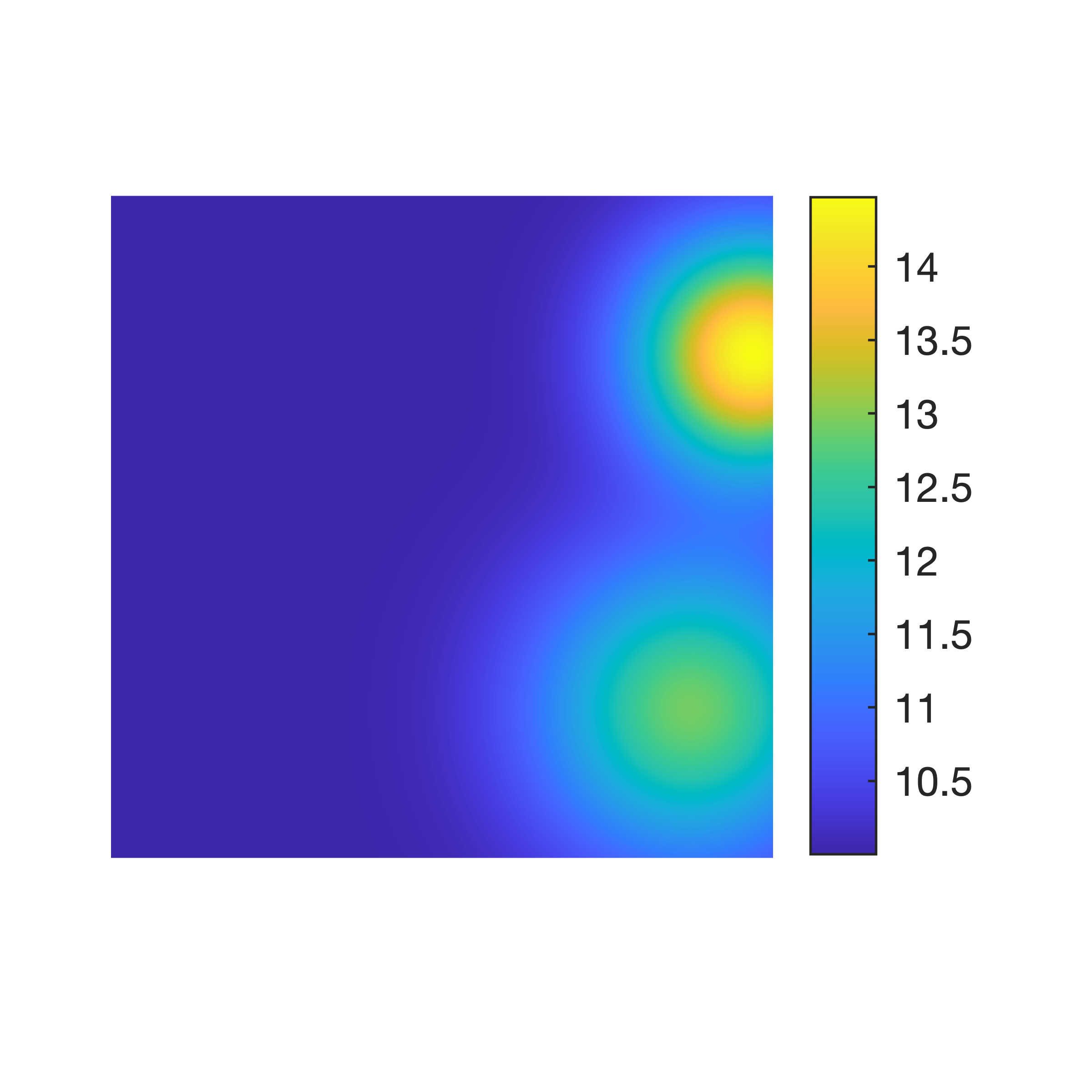}
		\includegraphics[width=0.24\textwidth,trim=0cm 1.5cm 0cm 0.5cm,clip]{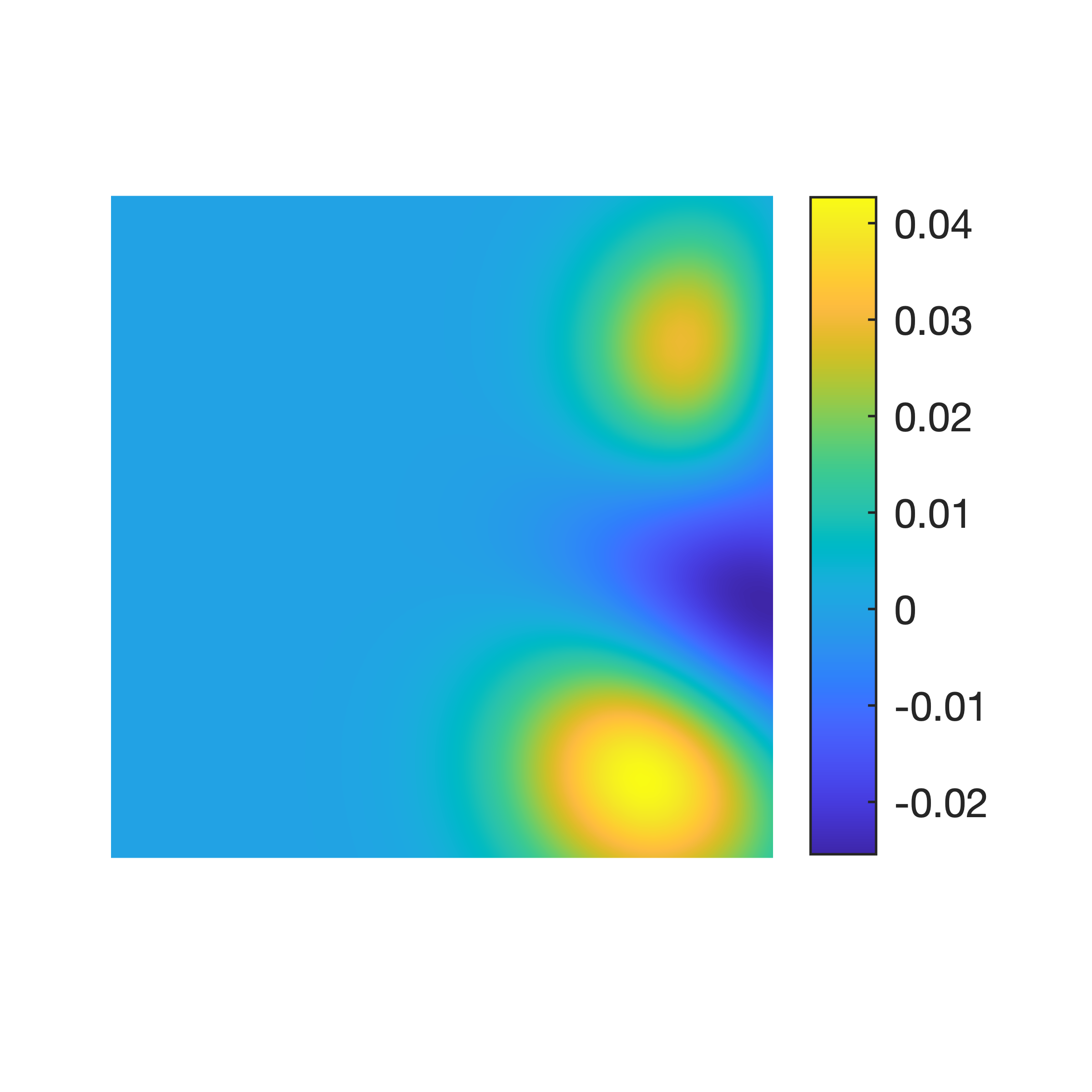} 
	\caption{\small{The reconstructed velocity images for the mixed Gaussian (\ref{model1_M2}). From the left to the right are the ground true velocity field, the reconstructed velocity field with $J = 1$, the reconstructed velocity field with $J = 20$, and the difference between the ground true velocity field and the reconstructed velocity field with $J = 20$ (first column - third column ). From the top to the bottom are the results from the noise-free wave signal, the wave signal with $10\%$ multiplication Gaussian noise, and the wave signal with $10\%$ additive Gaussian noise.}}\label{two_gaussian}
\end{figure}

\begin{table}
 	\footnotesize
 	\begin{center}
 		\scalebox{0.98}{
 			\begin{tabular}{c c c c c c c c c c}
 				\hline\\[-1ex]
                & \multicolumn{3}{c}{no noise} & \multicolumn{3}{c}{$10\%$ multiplicative noise} & \multicolumn{3}{c}{$10\%$ additive noise}\\[1ex]
 				\cline{2-10}\\[-1ex]
 				~& $L^2$  & $L^{\infty}$ &  CPU& $L^2$ & $L^{\infty}$  &  CPU & $L^2$ & $L^{\infty}$ &  CPU \\
 				$J$ &error & error &  time($s$)& error & error &  time($s$)&  error & error &  time($s$) \\
 				\hline
 				1 & 1.48e-01&  1.12e-00&  0 & 2.54e-01 & 1.98e-00 & 0 & 2.60e-01 & 1.60e-00 & 0\\
 				20& 1.10e-04 & 1.11e-03 & 6.83 & 1.71e-03 & 1.39e-02 & 6.71 & 5.79e-03 & 4.27e-02 &6.89 \\
 				40&  3.21e-06&  2.60e-05& 13.80& 4.84e-05 & 5.48e-04 & 13.60 & 5.88e-03 & 4.28e-02 & 13.38\\
 				60& 3.09e-06 & 2.41e-05 & 20.34& 5.78e-06 & 4.38e-05 & 20.97 & 5.89e-03 & 4.28e-02 & 20.51\\
 				80& 2.86e-06 & 2.66e-05 & 27.74& 4.63e-06 & 4.36e-05 & 28.52 & 5.89e-03 & 4.28e-02 & 27.56\\
 				\hline
 			\end{tabular}
 		}
 	\end{center}
 	\caption{\small{$L^2/L^\infty$ reconstruction errors, and the CPU time for the inversion stage with different $J$-term truncated Neumann series approximation, as well as different noise level/form for the reconstruction of the mixed Gaussian~\eqref{model1_M2}.}}\label{model1_M2_errors}
 \end{table}

\subsubsection{Inversion for the velocity model \texorpdfstring{\eqref{EQ:Velocity Model 1}}{} with \texorpdfstring{$M = 4$}{}}
\label{numerical_ex_4modes}

For the second inversion example, we work on reconstructing the features of the following velocity model
\begin{equation}\label{model2_M4}
m({\bf x}) = \sum_{k_x,k_z = 0}^4 \fm(\bk) \cos(k_x\pi x)\cos(k_z\pi z),\quad {\bk} = (k_x, k_z)
\end{equation}
with $\alpha = 0$ in \eqref{EQ:decay_rule}, namely, we don't consider any decay on the coefficients for this example. In addition, we use the same training settings as those in Section~\ref{SUBSEC:Data Generation} for the Fourier wave signal generation. But for the external top sources $h_i(x)$, we choose them to be the same as the sources in Section \ref{Sec:two_gaussian} to generate resourceful training samples for the construction of $\wh{\bf f}_{\wh{\theta}}^{-1}$.  

  For the online inversion stage, we also implement a $J$-term truncated Neumann series approximation \eqref{EQ:Neumann J} to recover the velocity model. To test the stability of the proposed coupled scheme, as in Section \ref{Sec:two_gaussian}, we add the Gaussian noise with zeros mean and $10\%$ standard derivation to the wave signals. Figure \ref{five_fourier} presents the surface plots of the reconstructed velocity images with $J = 20$, as well as the surface plots for the difference between the reconstructed image and the ground true velocity model. The layout of Figure \ref{five_fourier} is the same as the one in Figure \ref{two_gaussian}. We observe that the training prediction is stable with respect to the noise, see the second column of Figure \ref{five_fourier} and $L^2/L^\infty$ errors when $J = 1$ in Table \ref{model2_M4_errors}. In addition, we note that the inversion stage can significantly improve the accuracy of the reconstruction. For the data without noise, the errors dropped by a factor $\sim 10^7$ within $30$ seconds; even for the data with $10\%$ Gaussian noise, the errors also dropped by a factor $\sim 10^3$ within $30$ seconds.

\begin{figure}[!hbt]
	\centering
	\includegraphics[width=0.24\textwidth,trim=0cm 1.5cm 0cm 0cm,clip]{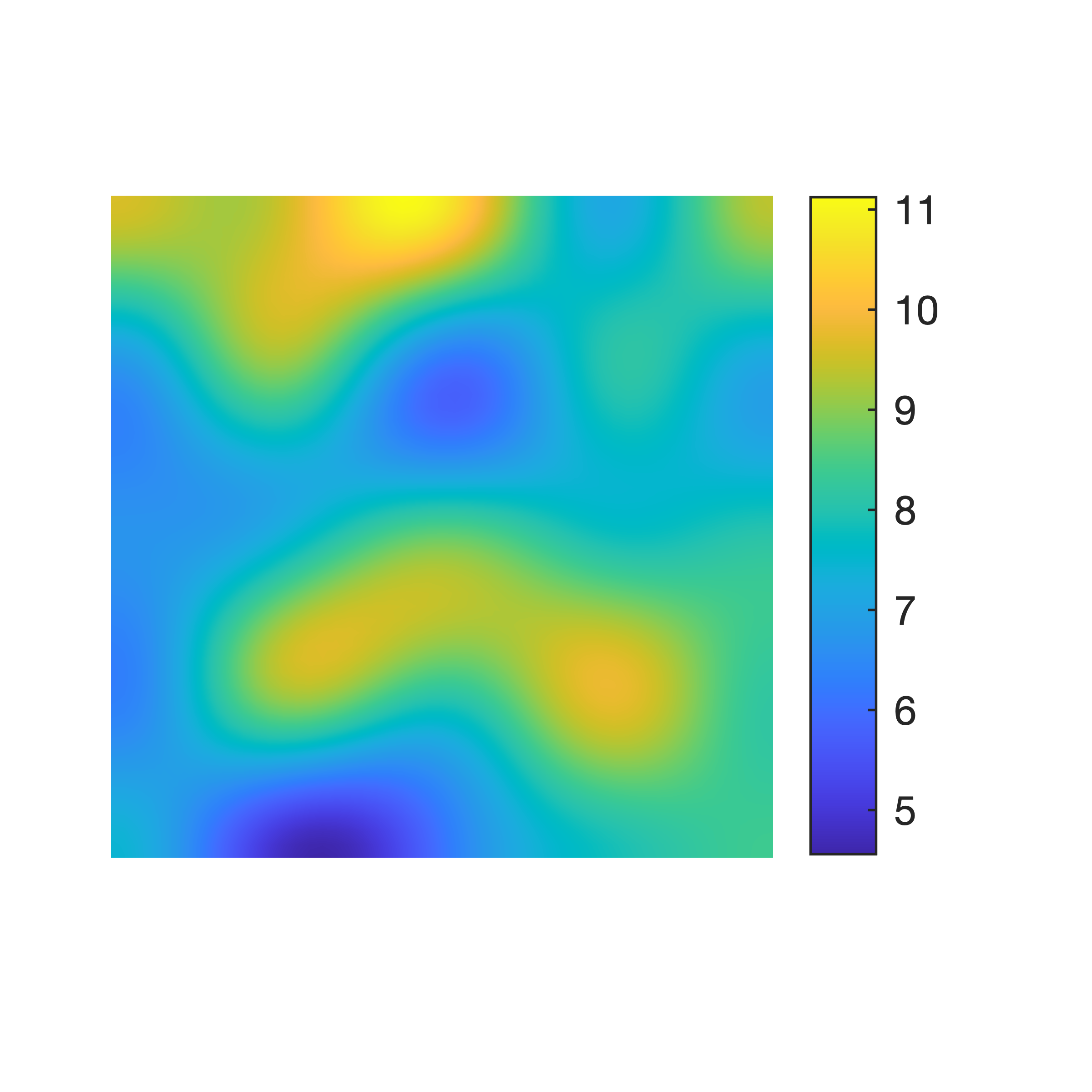}
	\includegraphics[width=0.24\textwidth,trim=0cm 1.5cm 0cm 0cm,clip]{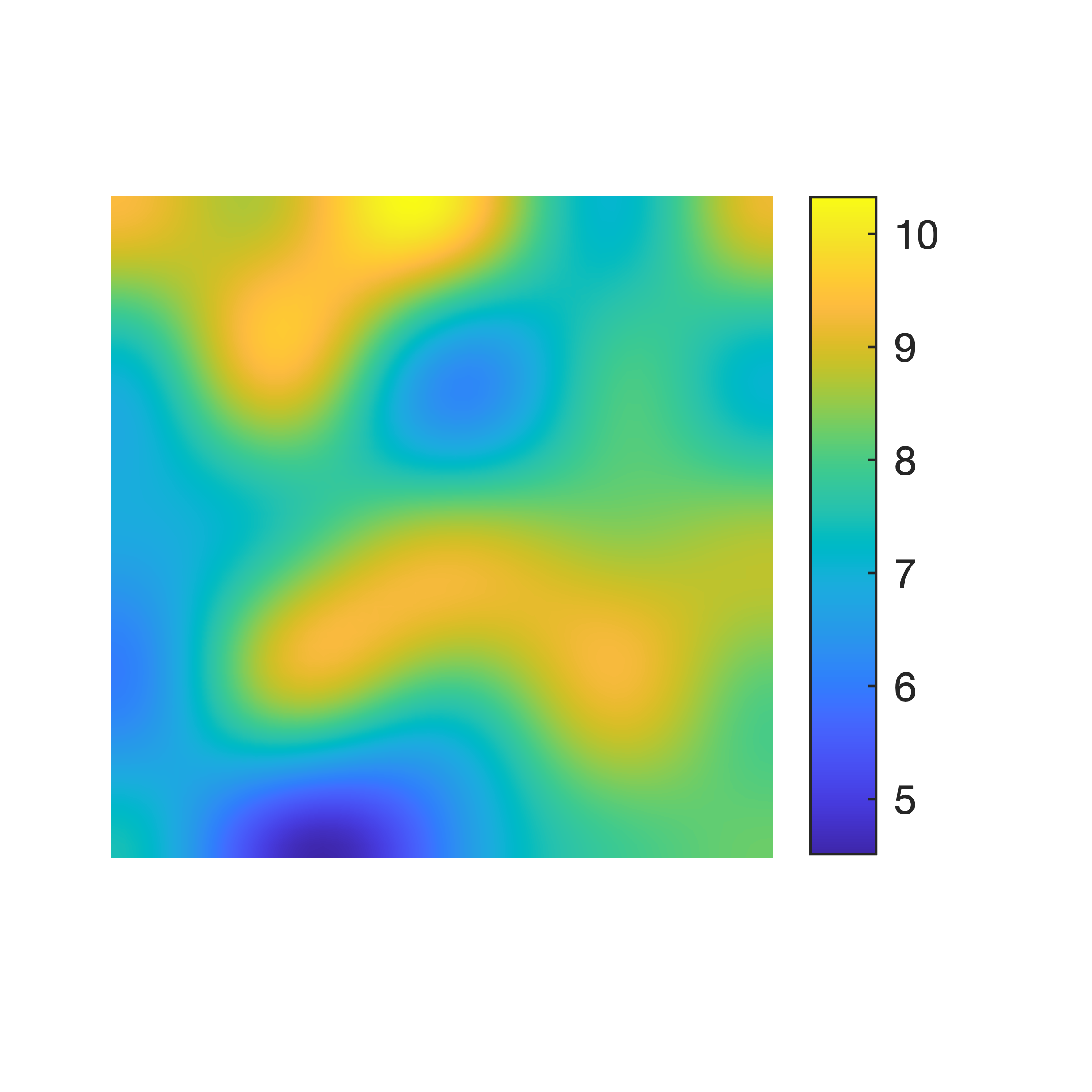} 
	\includegraphics[width=0.24\textwidth,trim=0cm 1.5cm 0cm 0cm,clip]{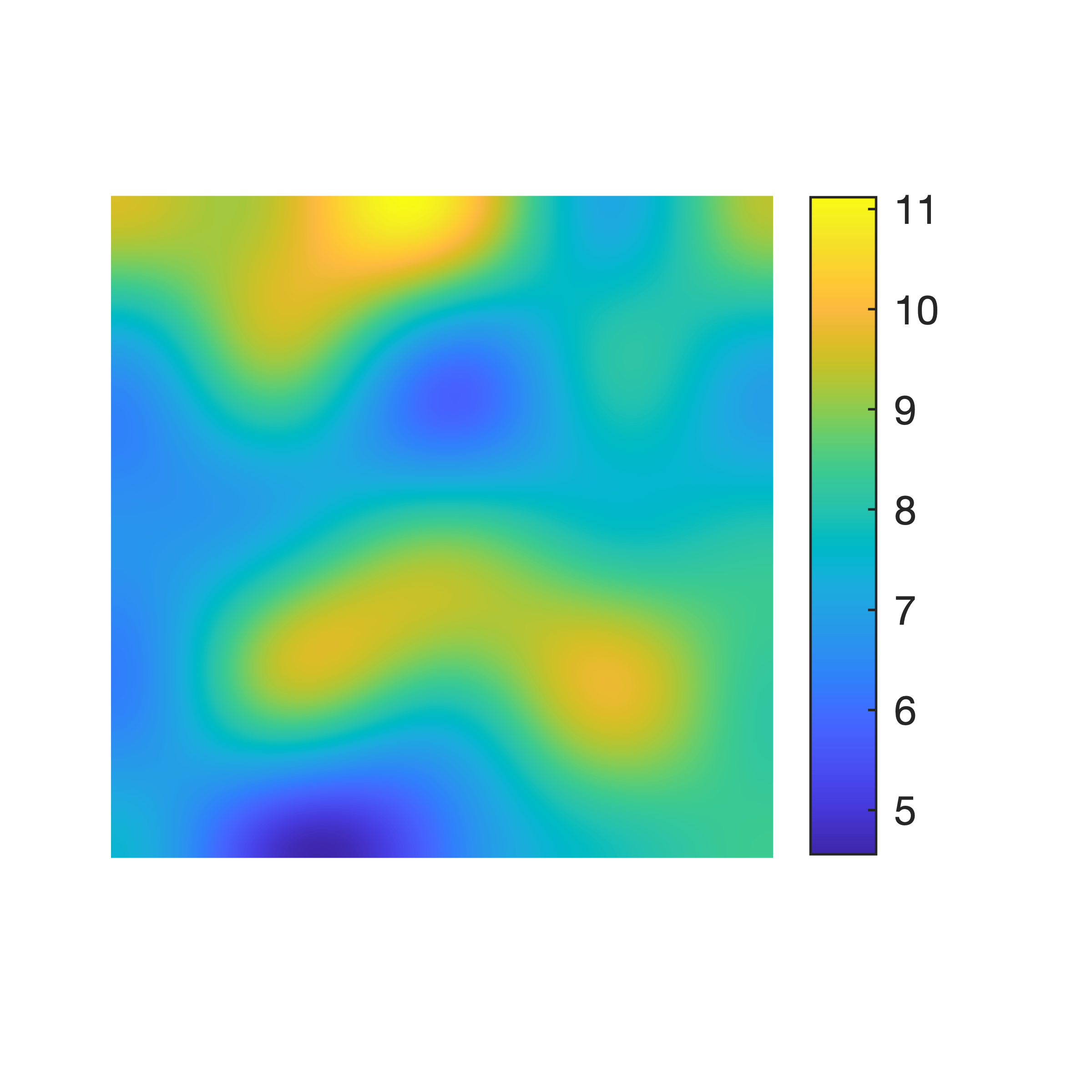}
	\includegraphics[width=0.24\textwidth,trim=0cm 1.5cm 0cm 0cm,clip]{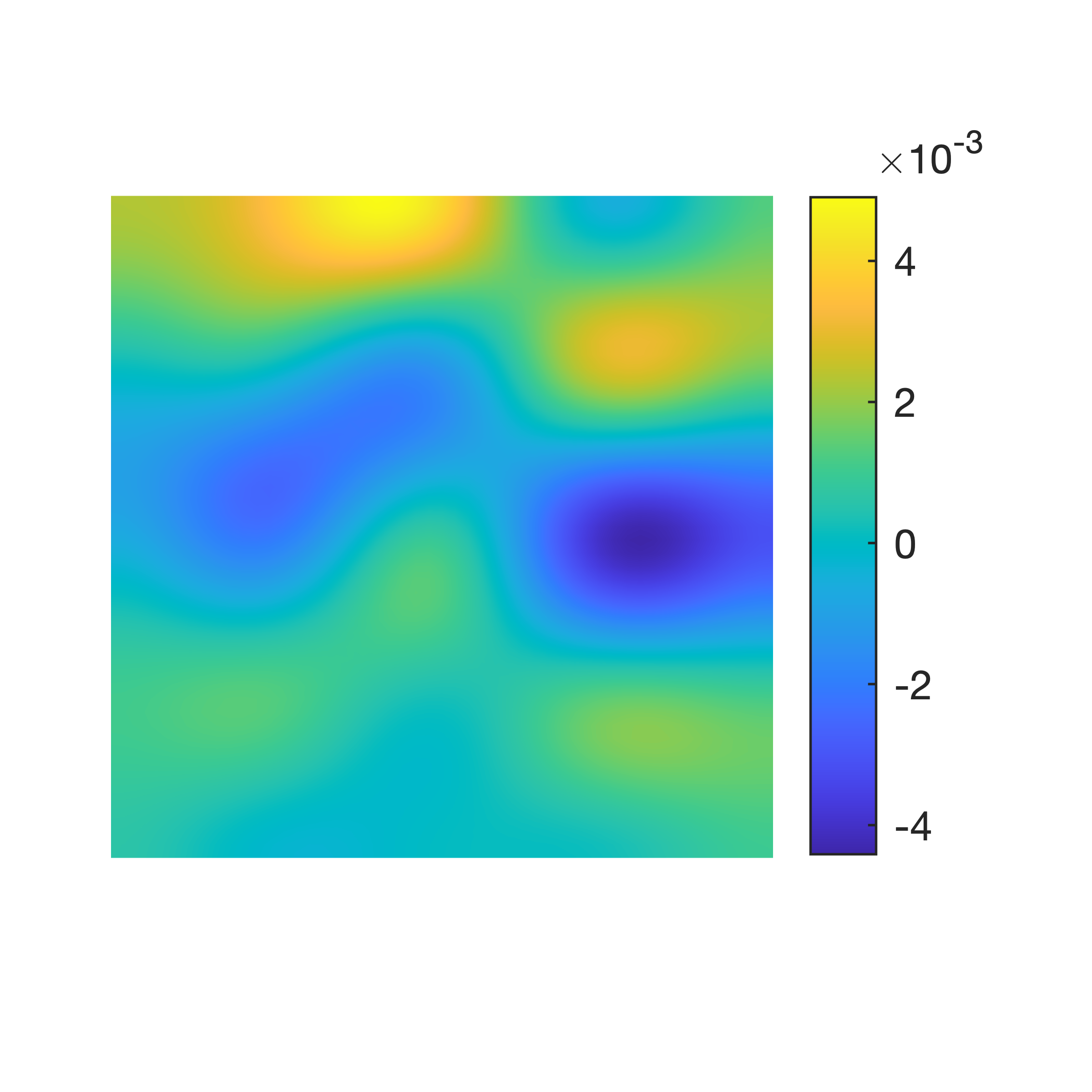}\\
	\includegraphics[width=0.24\textwidth,trim=0cm 1.5cm 0cm 1cm,clip]{true_mode_F5}
	\includegraphics[width=0.24\textwidth,trim=0cm 1.5cm 0cm 1cm,clip]{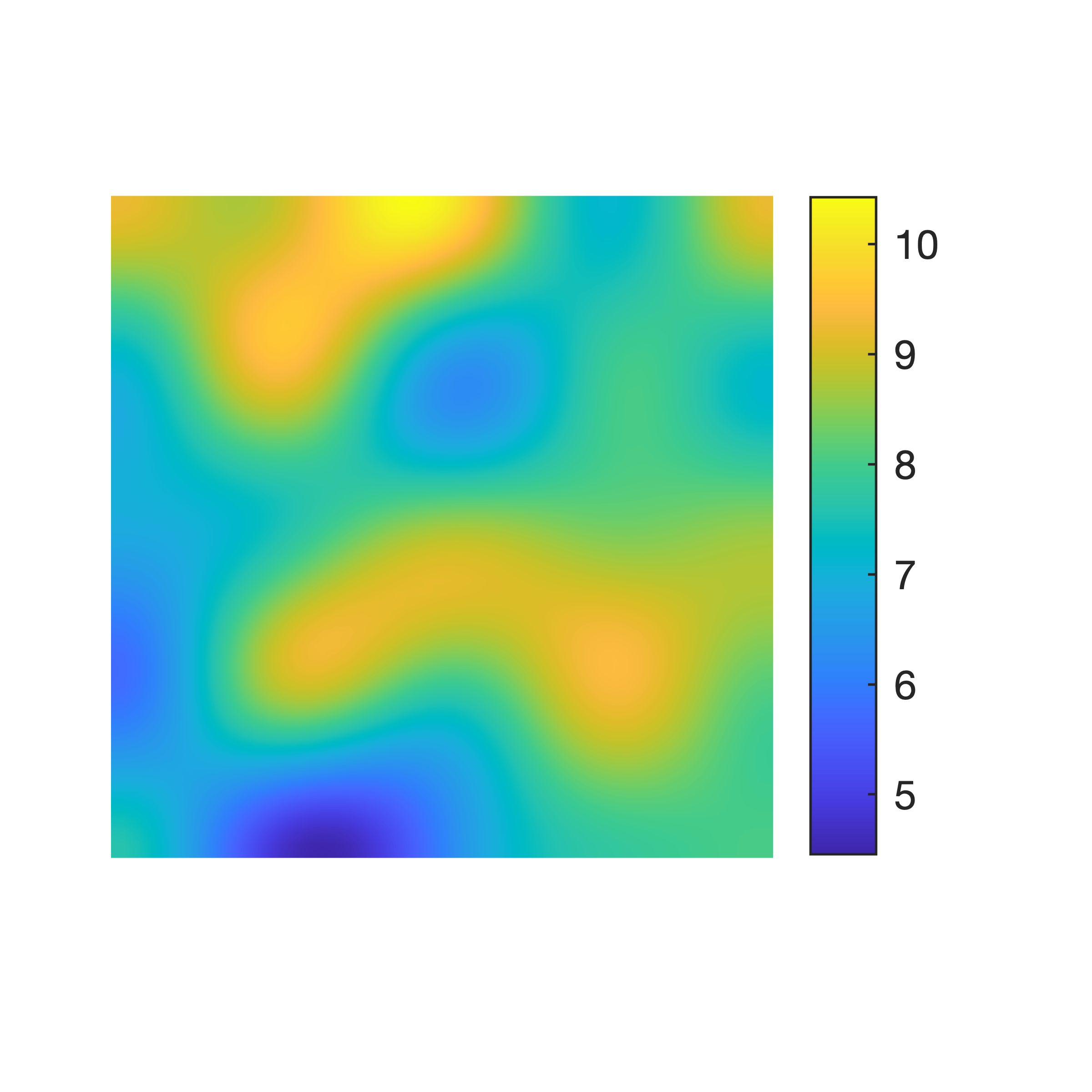} 
	\includegraphics[width=0.24\textwidth,trim=0cm 1.5cm 0cm 1cm,clip]{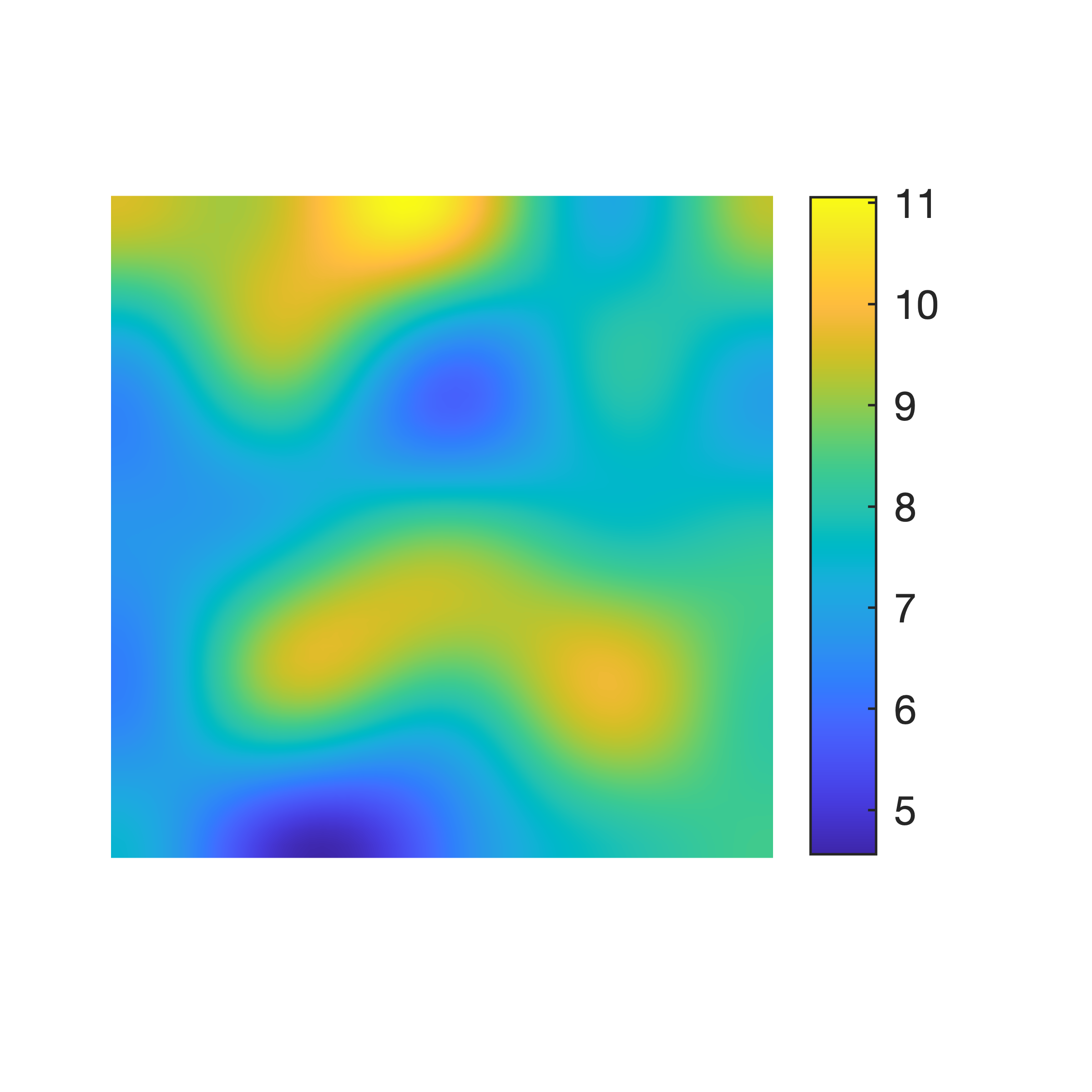}
	\includegraphics[width=0.24\textwidth,trim=0cm 1.5cm 0cm 1cm,clip]{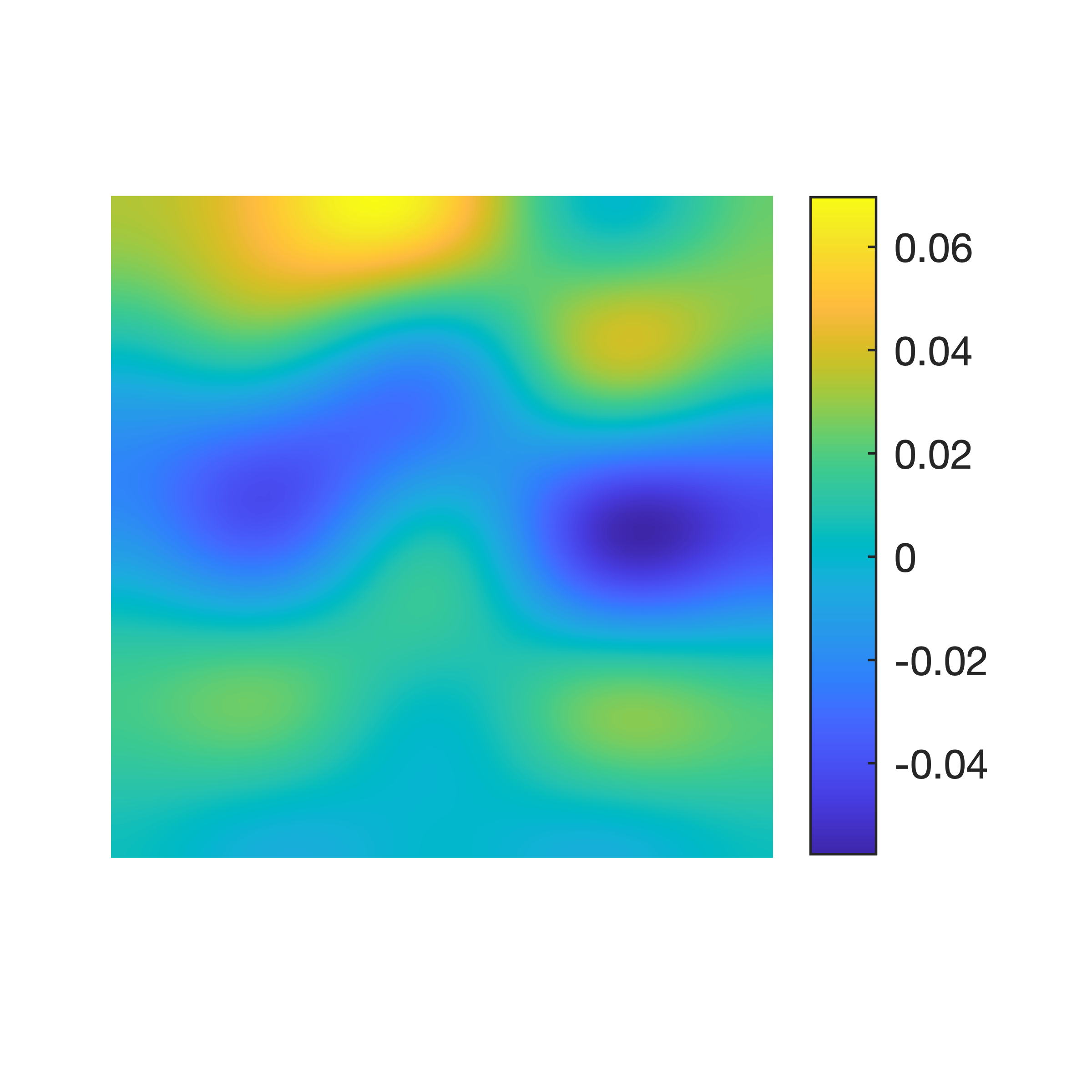}\\
	\includegraphics[width=0.24\textwidth,trim=0cm 1.5cm 0cm 1cm,clip]{true_mode_F5}
	\includegraphics[width=0.24\textwidth,trim=0cm 1.5cm 0cm 1cm,clip]{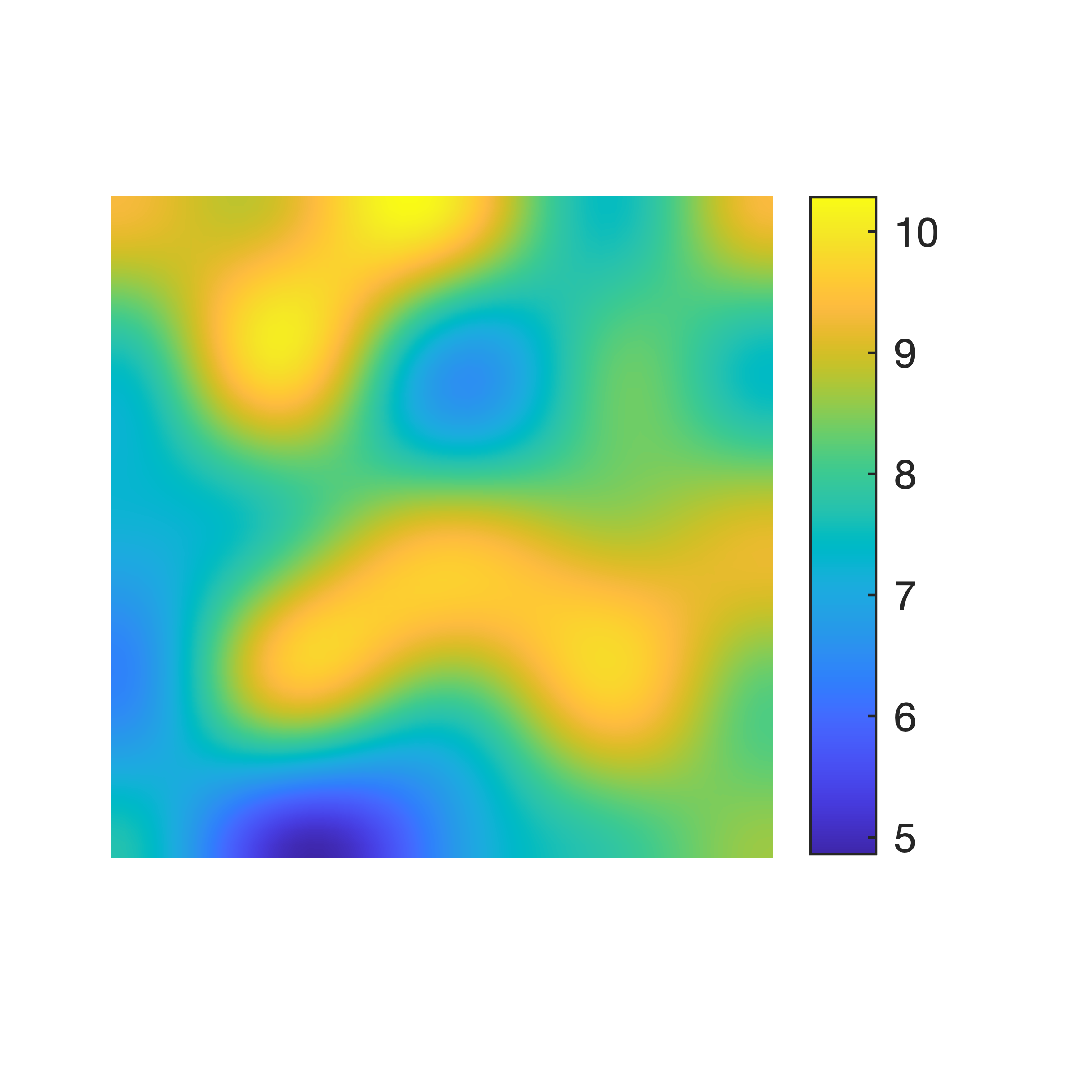} 
	\includegraphics[width=0.24\textwidth,trim=0cm 1.5cm 0cm 1cm,clip]{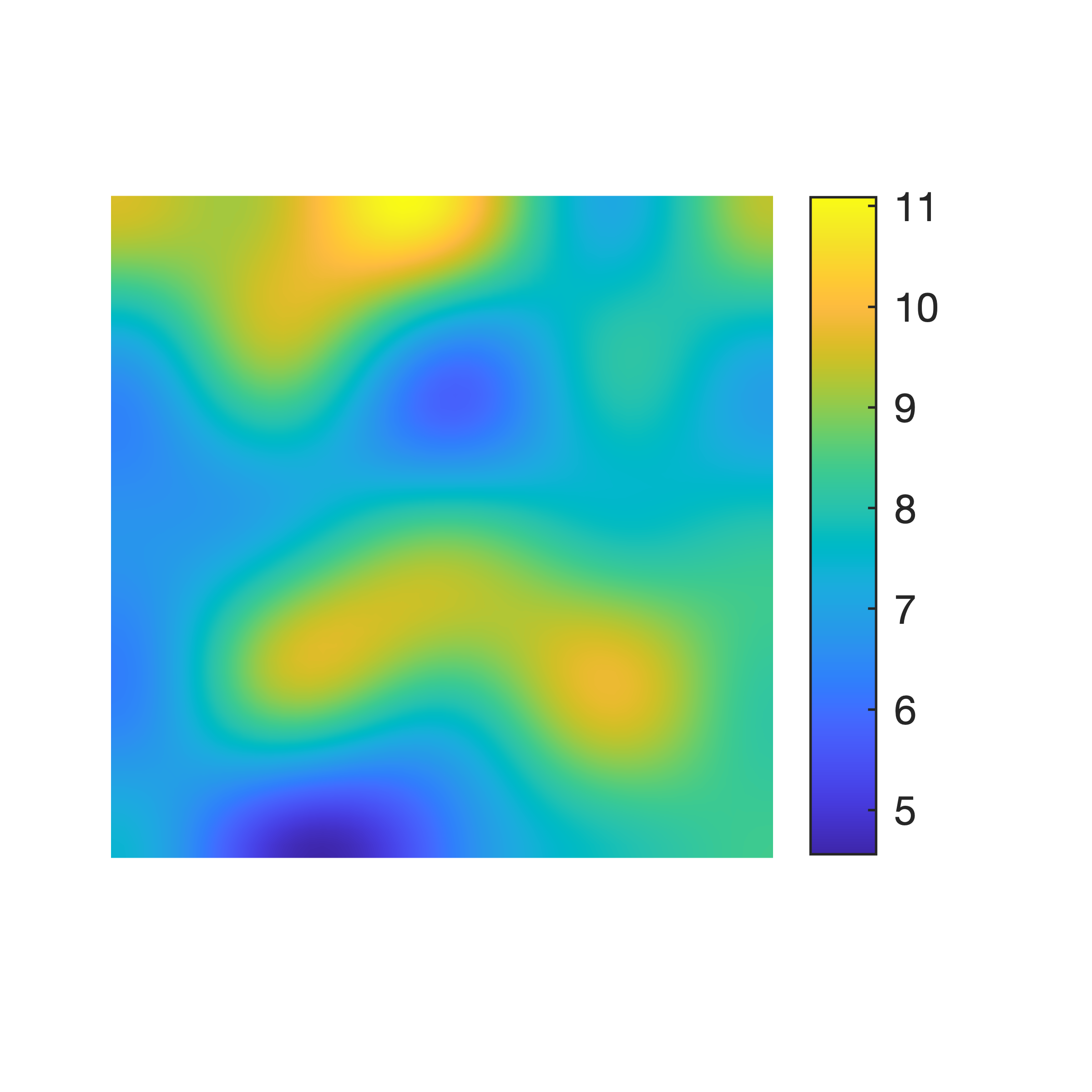}
	\includegraphics[width=0.24\textwidth,trim=0cm 1.5cm 0cm 1cm,clip]{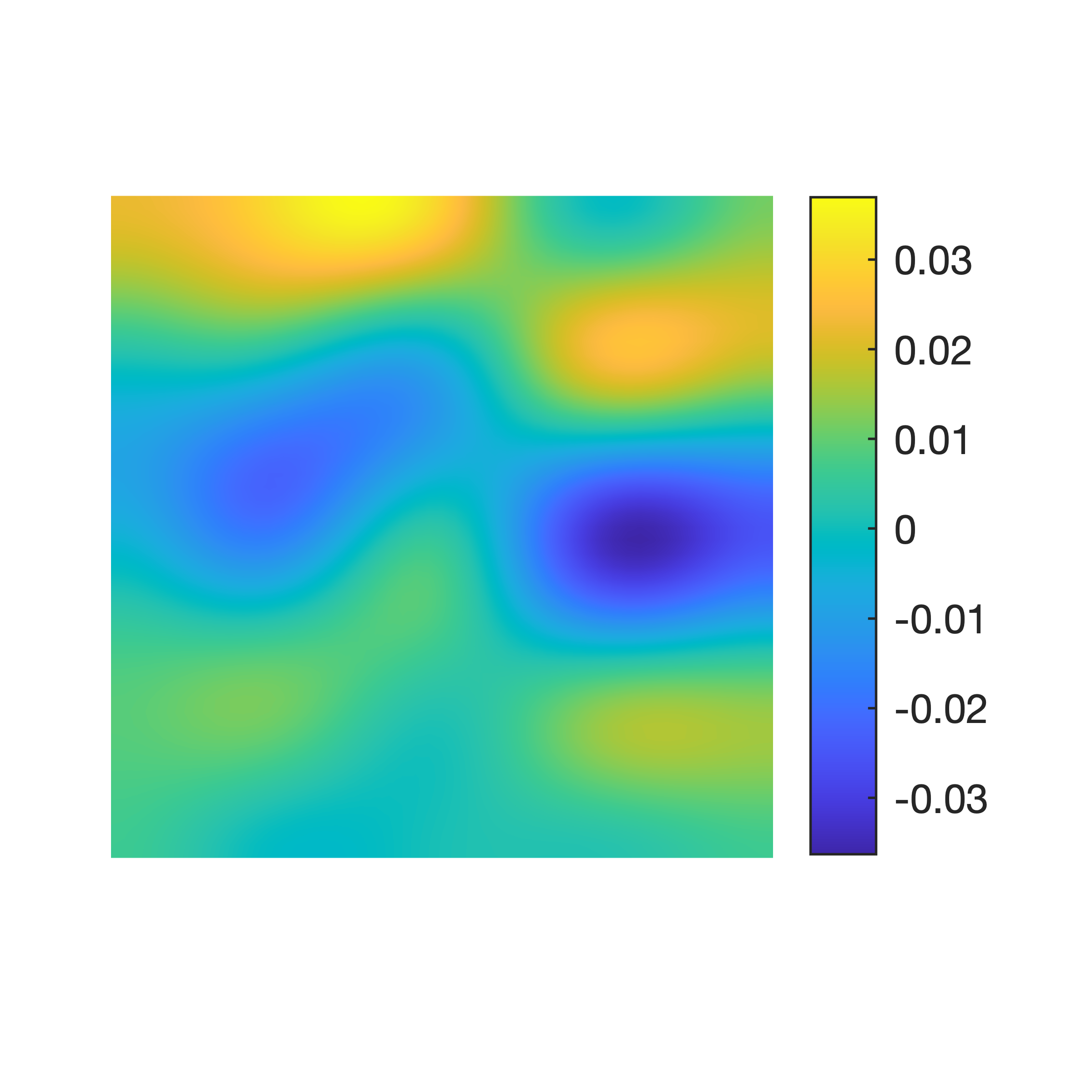}\\
	\caption{\small{The reconstructed velocity images for the general Fourier type (\ref{EQ:Velocity Model 1}) with $M = 4$. From the top to the bottom are for the velocity reconstruction without noise, with $10\%$ multiplication Gaussian noise, with $10\%$ additive Gaussian noise, respectively. While from the left to the right are the ground true velocity field, the reconstructed velocity field from the neural network in the offline training stage, and the reconstructed velocity image with $J = 20$, error for the reconstructed velocity image with $J = 20$, respectively.}}\label{five_fourier}
\end{figure}

 \begin{table}
 	\footnotesize
 	\begin{center}
 		\scalebox{1.0}{
 			\begin{tabular}{c c c c c c c c c c}
 				\hline\\[-1ex] 
                & \multicolumn{3}{c}{no noise} & \multicolumn{3}{c}{$10\%$ multiplicative noise} & \multicolumn{3}{c}{$10\%$ additive noise}\\[1ex]
 				\cline{2-10}\\[-1ex]
                ~& $L^2$  & $L^{\infty}$ &  CPU& $L^2$ & $L^{\infty}$  &  CPU & $L^2$ & $L^{\infty}$ &  CPU \\
 				$J$ &error & error &  time($s$)& error & error &  time($s$)&  error & error &  time($s$) \\
 				\hline
 				1 & 1.78e-01&  8.46e-01&  0 & 1.79e-00 & 7.25e-00 & 0 & 2.64e-01 & 1.16e-00 & 0\\
 				20& 8.52e-04 & 4.90e-03 & 6.08 & 1.23e-02 & 6.97e-02 & 5.81 & 7.13e-03 & 3.70e-02 &5.87 \\
 				40&  1.49e-05&  8.41e-05& 11.15& 4.93e-03 & 2.76e-02 & 11.24 & 1.69e-03 & 9.26e-03 & 11.73\\
 				60& 2.60e-07 & 1.53e-06 & 17.97& 3.06e-03 & 1.71e-02 & 17.42 & 6.81e-04 & 3.77e-03 & 17.67\\
 				80& 2.16e-08 & 1.34e-07 & 22.74& 2.19e-03 & 1.22e-02 & 23.25 & 2.05e-04 & 1.12e-03 & 23.13\\
 				\hline
 			\end{tabular}
 		}
 	\end{center}
 	\caption{\small{$L^2/L^\infty$ reconstruction errors, and the CPU time for the inversion stage with different $J$-term truncated Neumann series approximation, as well as different noise level/form for the reconstruction of the Fourier model (\ref{model2_M4}).}}\label{model2_M4_errors}
 \end{table}

\subsubsection{Inversion for the velocity model \texorpdfstring{\eqref{EQ:Velocity Model 1}}{} with \texorpdfstring{$M = 7$}{}}

For the third example, we consider a velocity model that contains $8$ Fourier modes in each direction, namely,
\begin{equation}\label{model2_M7}
m({x,z}) = \sum_{k_x,k_z = 0}^{7} \fm(\bk) \cos(k_x\pi x)\cos(k_z\pi z).
\end{equation}
The spatial and temporal discretization, as well as the rules for data generation, the choice of the top source $h_i(x)$ are the same as the example in Section~\ref{numerical_ex_4modes}.
 
For the inversion stage, we again implement a $J$-term truncated Neumann series approximation~\eqref{EQ:Neumann J} to obtain the reconstructed velocity image. Figure \ref{eight_fourier} presents the surface plots of the reconstructed velocity images with various values of $J$ in the online inversion stage. Precisely, each row of Figure \ref{eight_fourier} corresponds to one velocity model; from the left to the right are the ground true velocity field, the reconstructed velocity image with $J = 1$, the reconstructed velocity image with $J = 20$, and the reconstructed velocity image with $J = 50$, respectively. We note that the online inversion stage improves the accuracy of the reconstruction for all cases, which verifies the effectiveness of the proposed coupling scheme.

\begin{figure}[!htb]
	\centering
	\includegraphics[width=0.24\textwidth,trim=0cm 1.5cm 0cm 1cm,clip]{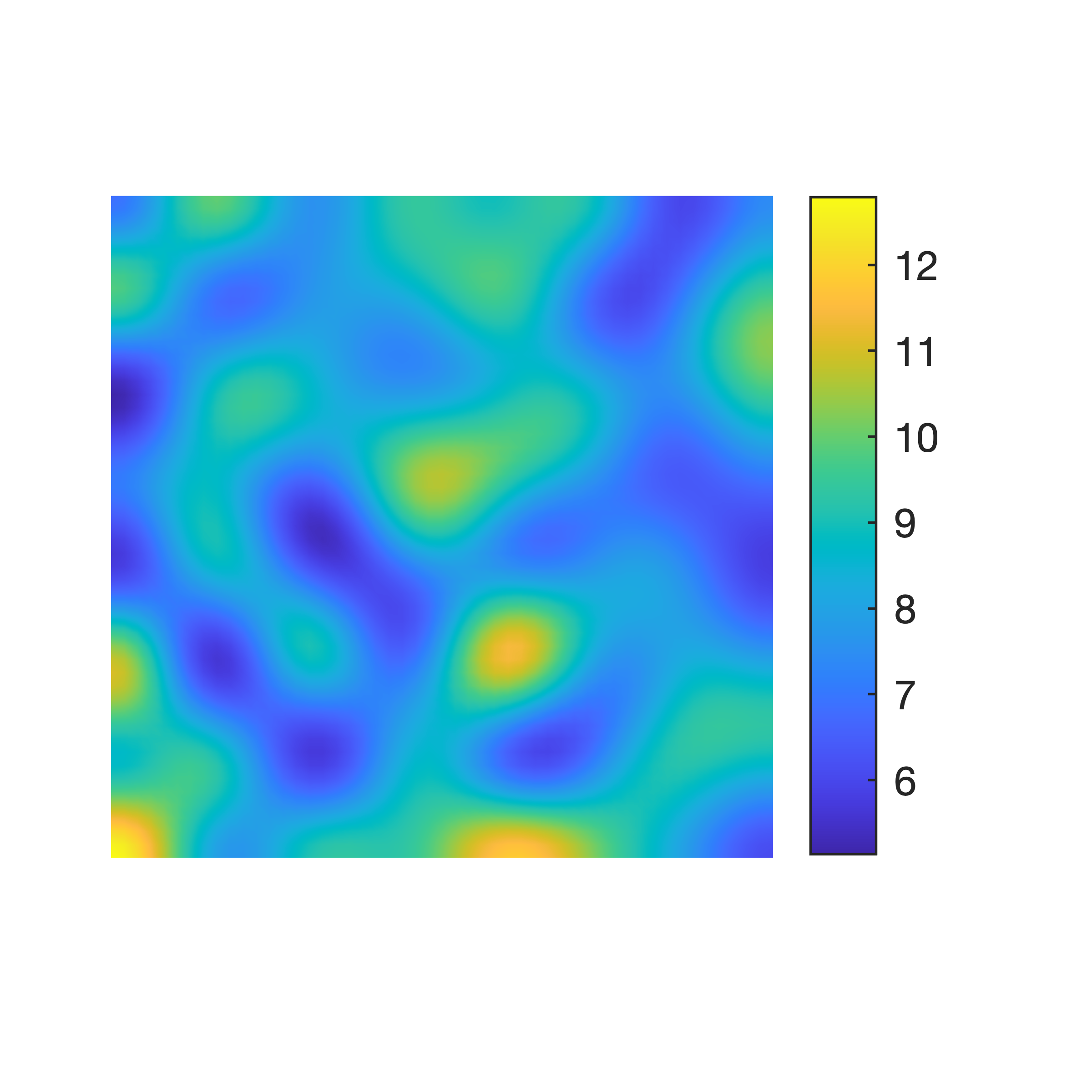}
	\includegraphics[width=0.24\textwidth,trim=0cm 1.5cm 0cm 1cm,clip]{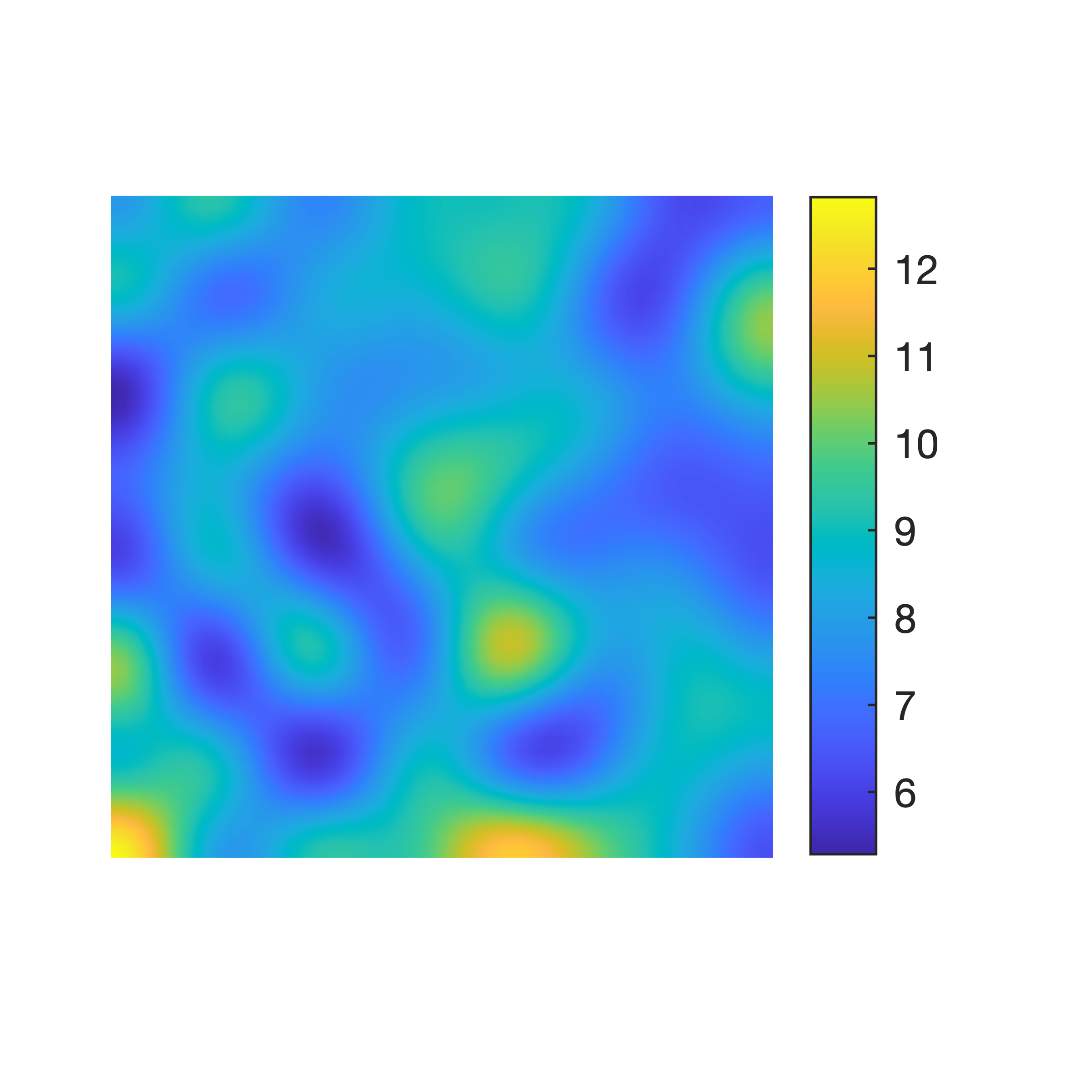} 
	\includegraphics[width=0.24\textwidth,trim=0cm 1.5cm 0cm 1cm,clip]{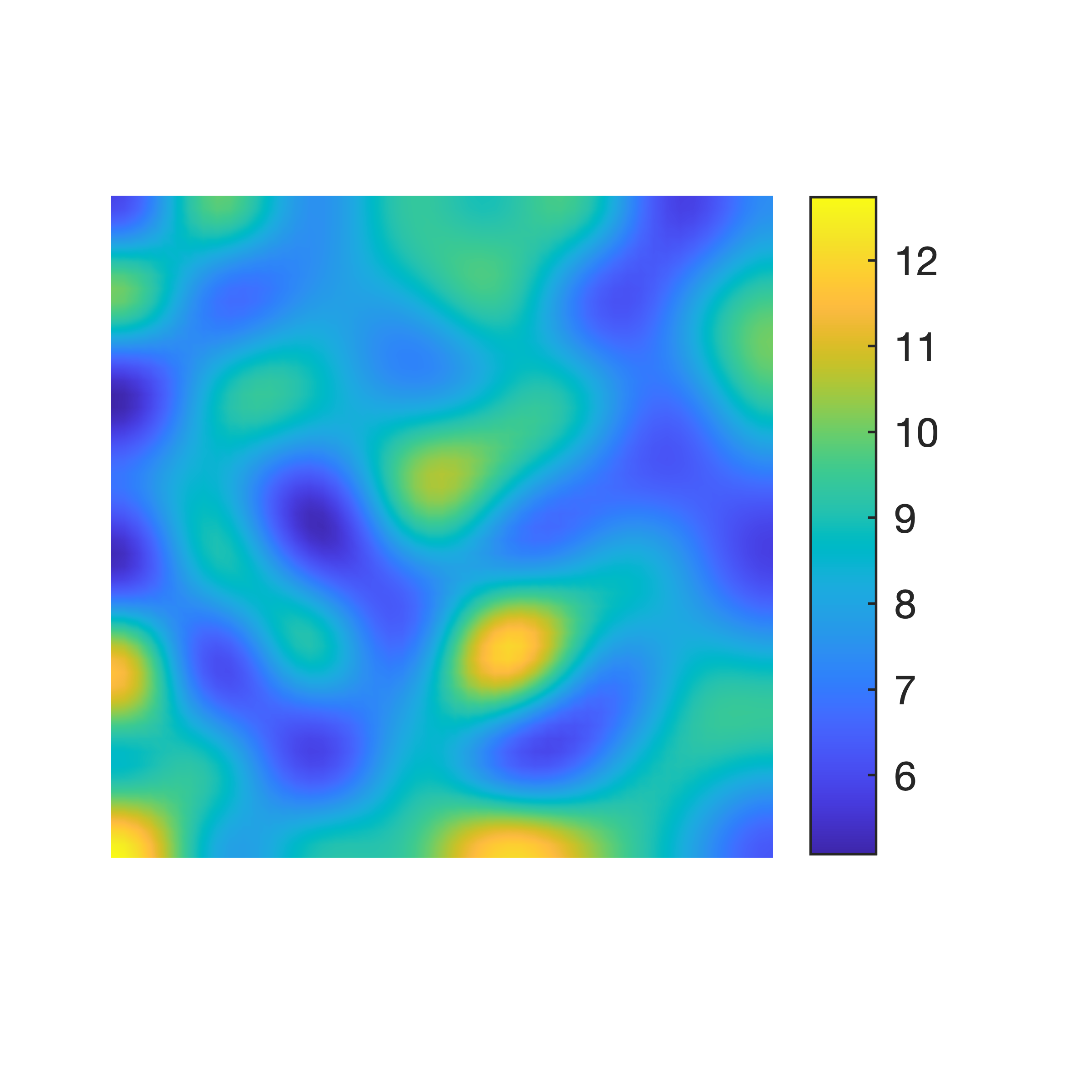}
	\includegraphics[width=0.24\textwidth,trim=0cm 1.5cm 0cm 1cm,clip]{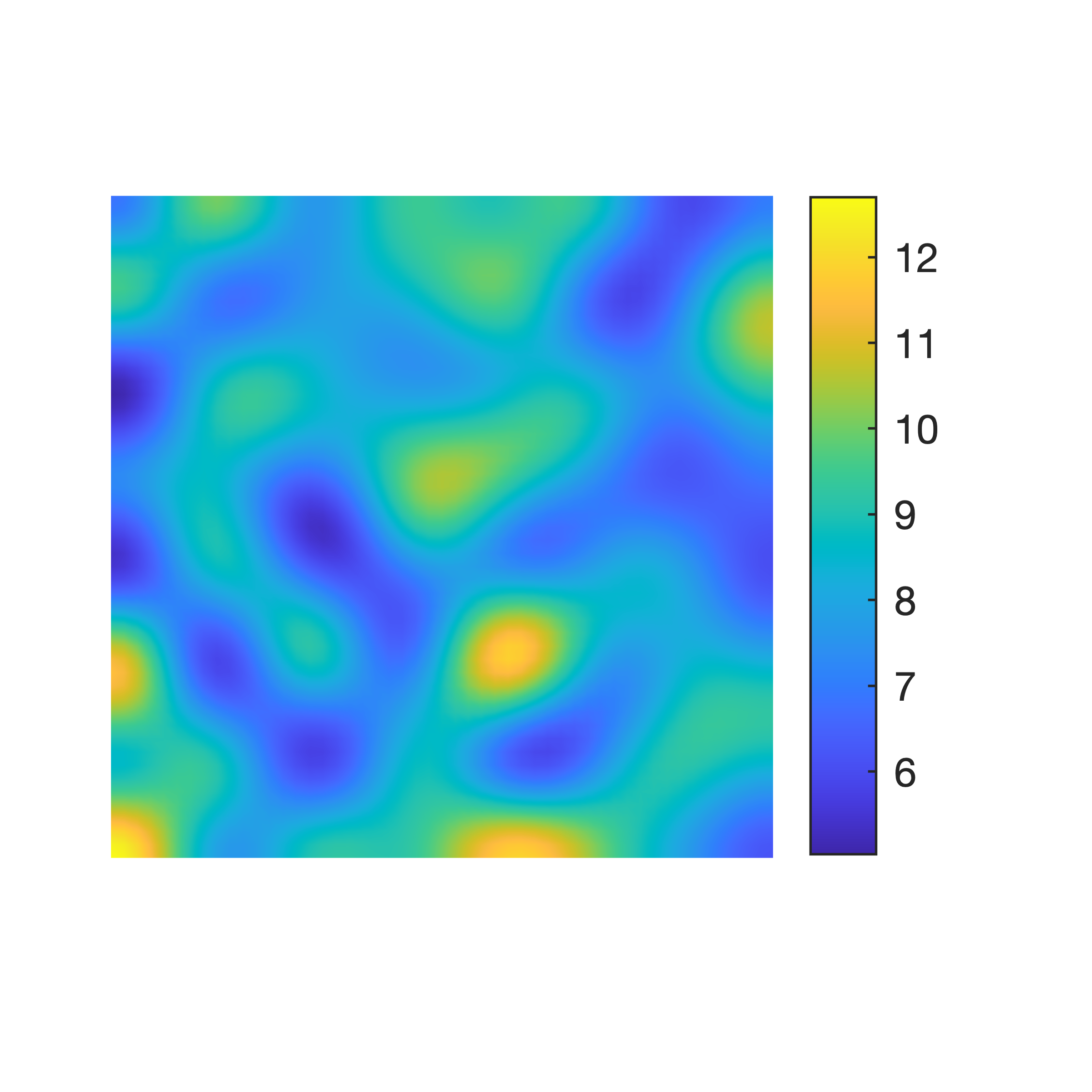} \\
	\includegraphics[width=0.24\textwidth,trim=0cm 1.5cm 0cm 1cm,clip]{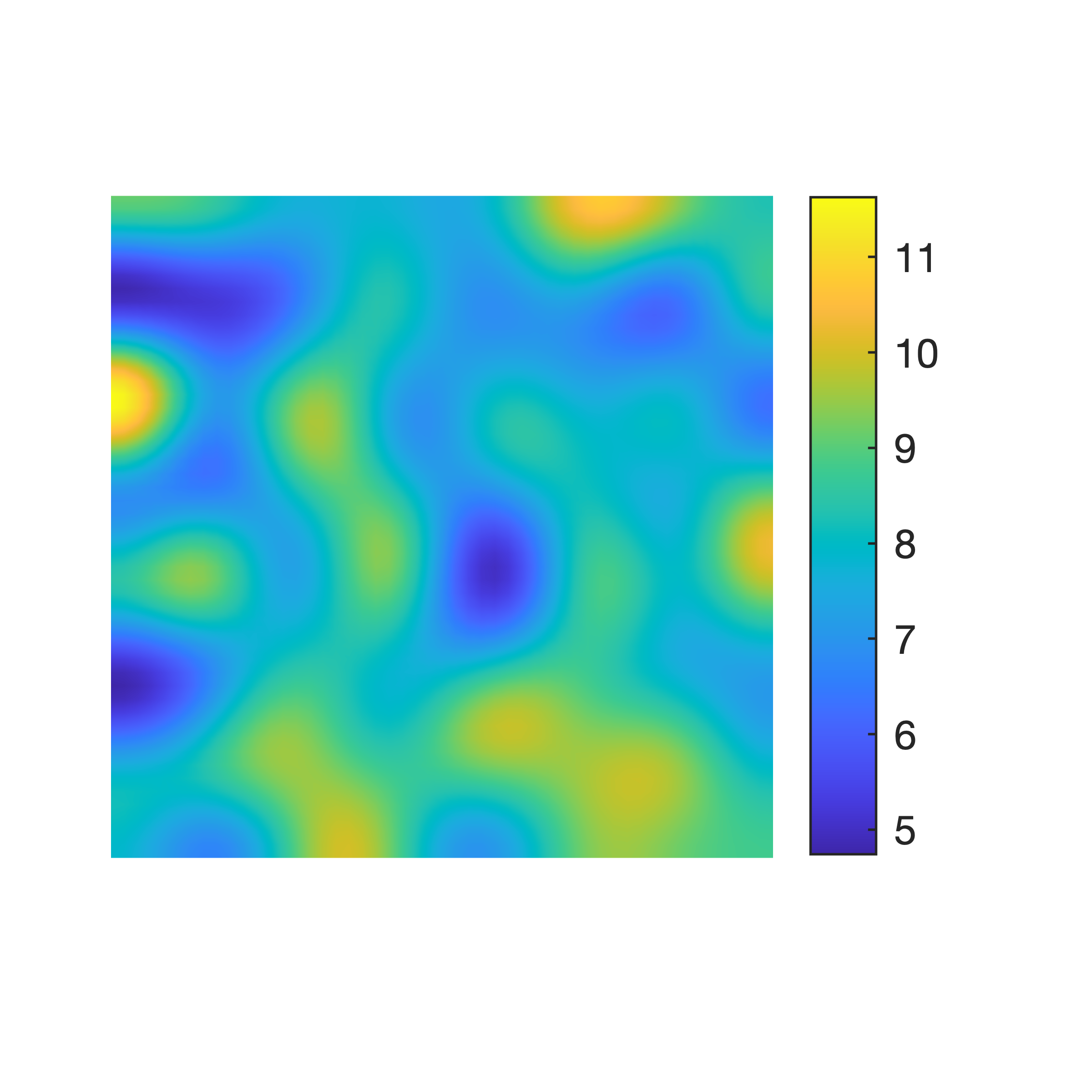}
	\includegraphics[width=0.24\textwidth,trim=0cm 1.5cm 0cm 1cm,clip]{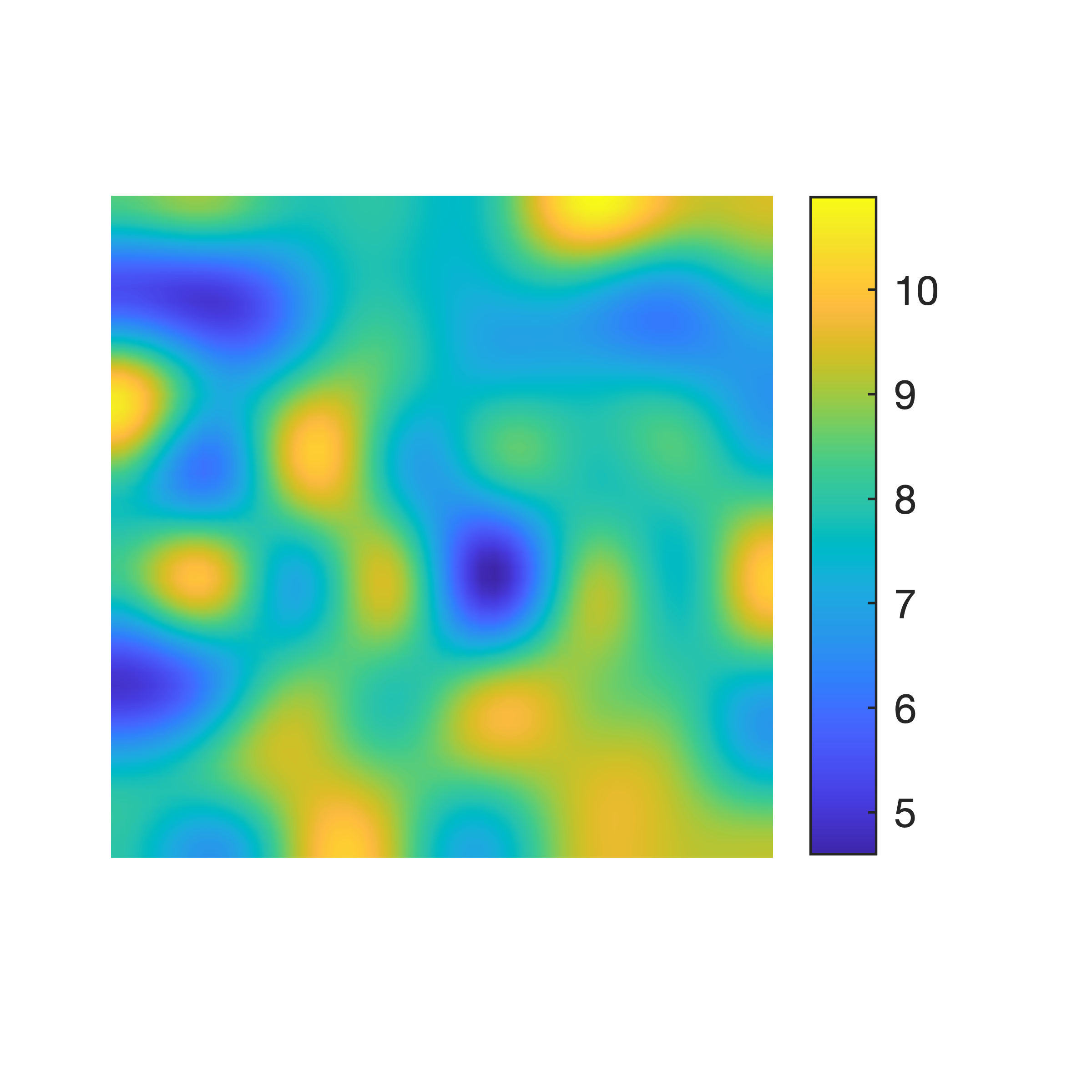} 
	\includegraphics[width=0.24\textwidth,trim=0cm 1.5cm 0cm 1cm,clip]{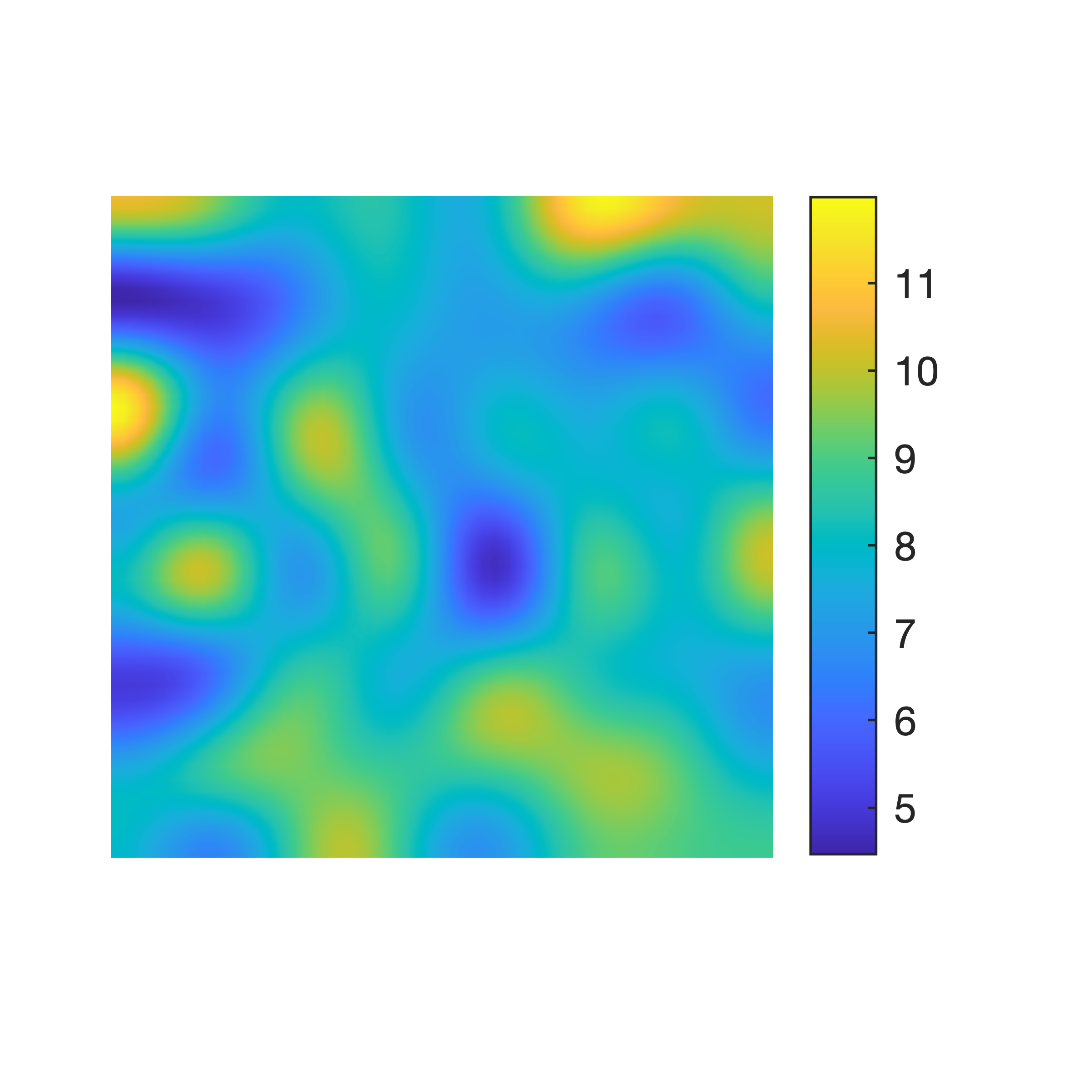}
	\includegraphics[width=0.24\textwidth,trim=0cm 1.5cm 0cm 1cm,clip]{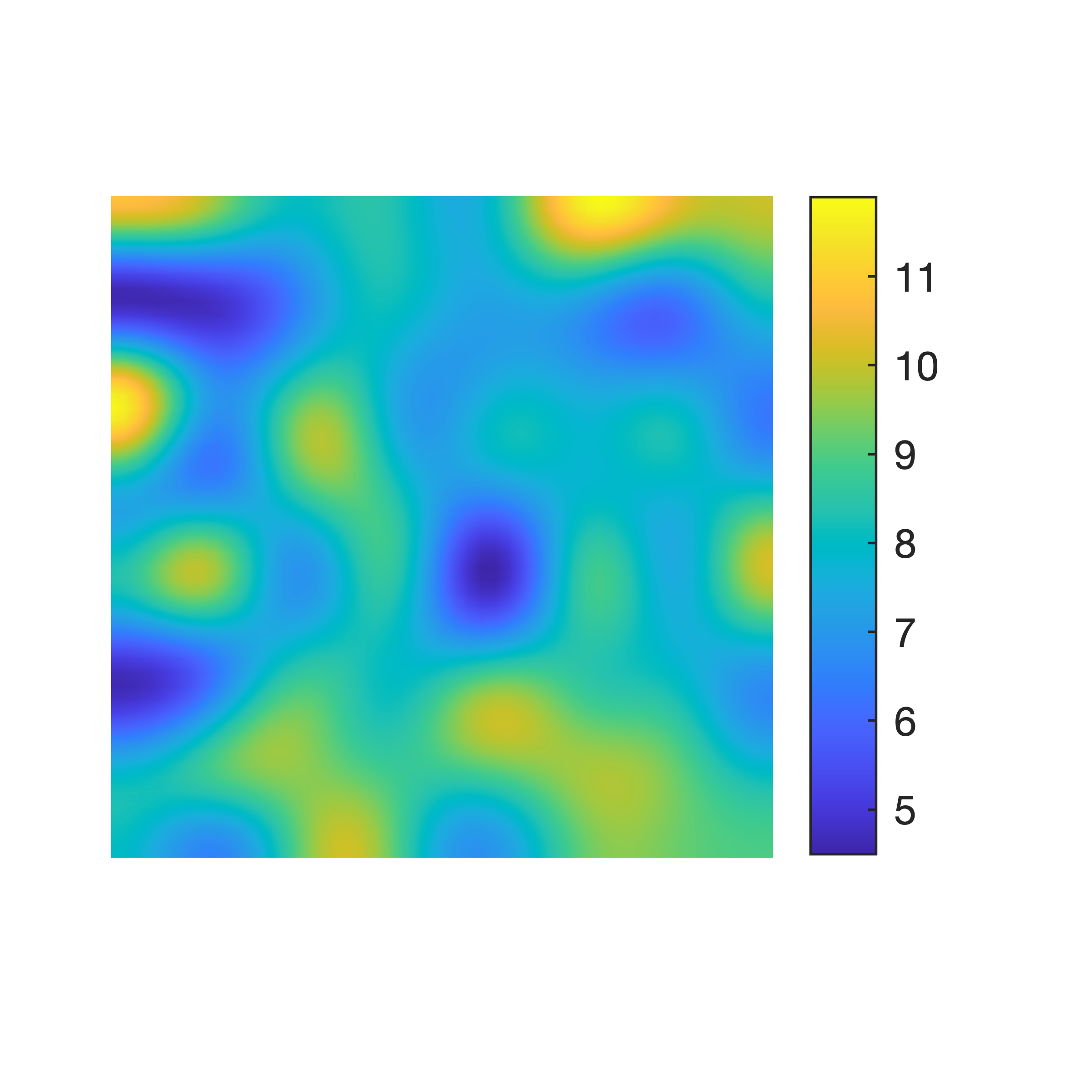} \\
	\includegraphics[width=0.24\textwidth,trim=0cm 1.5cm 0cm 1cm,clip]{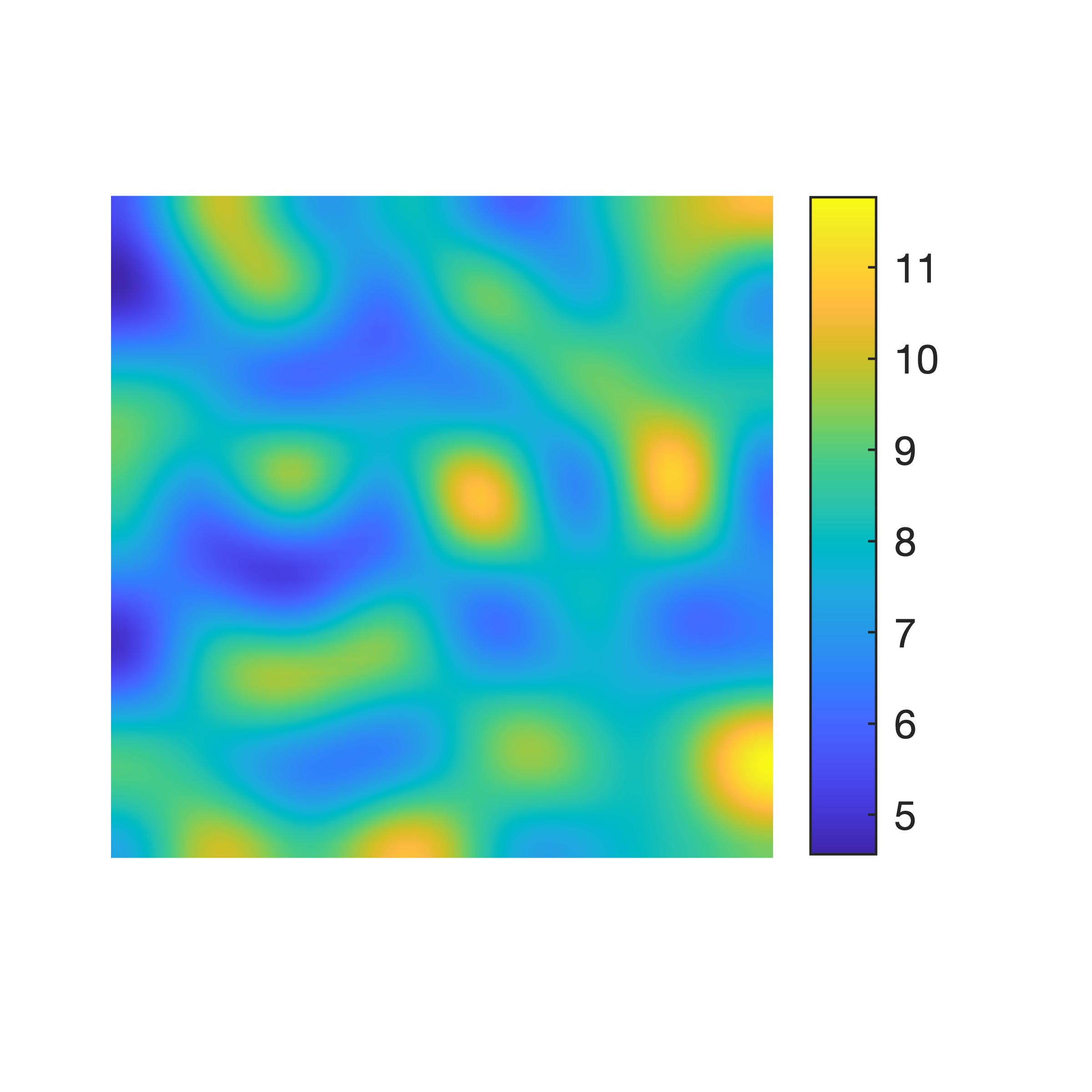}
	\includegraphics[width=0.24\textwidth,trim=0cm 1.5cm 0cm 1cm,clip]{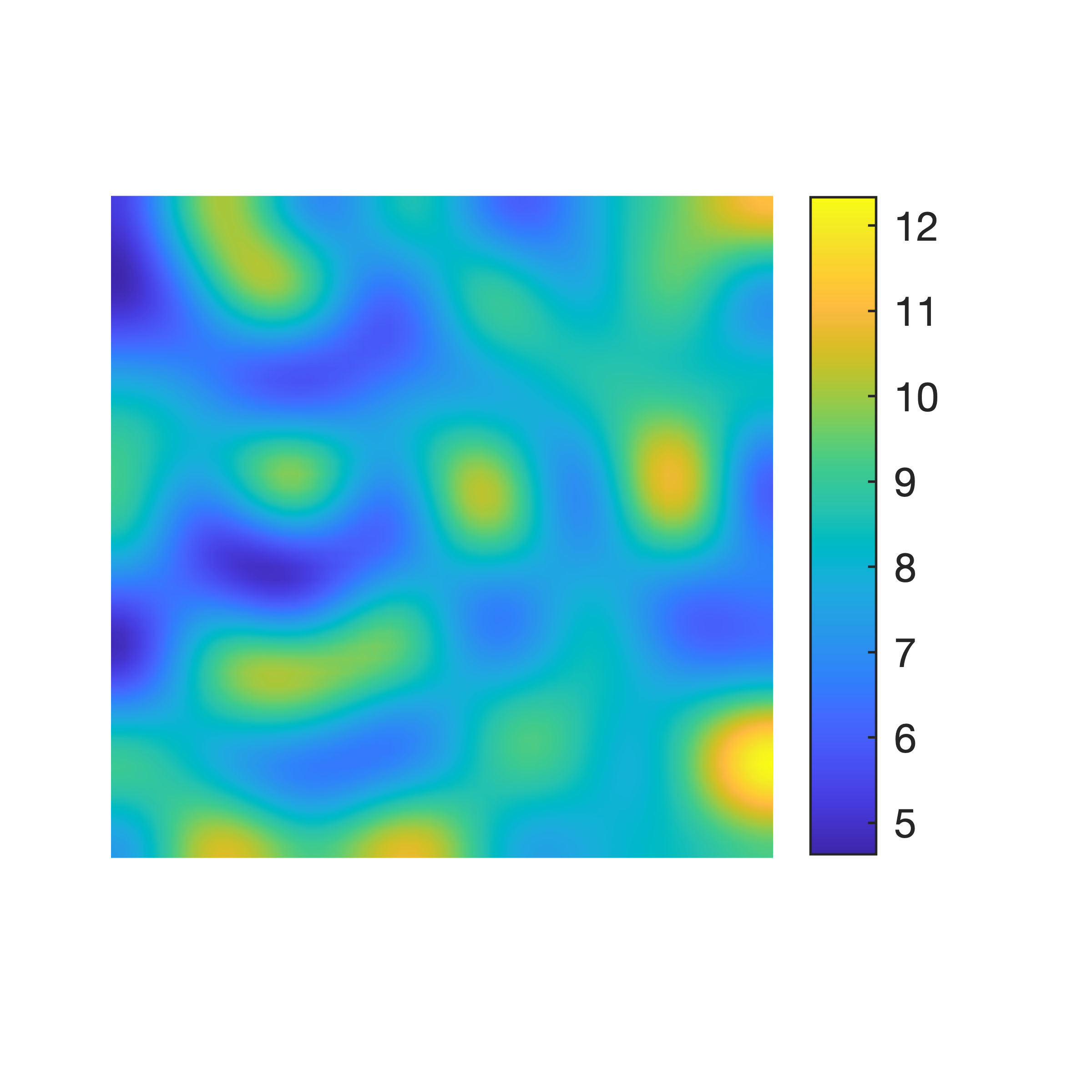} 
	\includegraphics[width=0.24\textwidth,trim=0cm 1.5cm 0cm 1cm,clip]{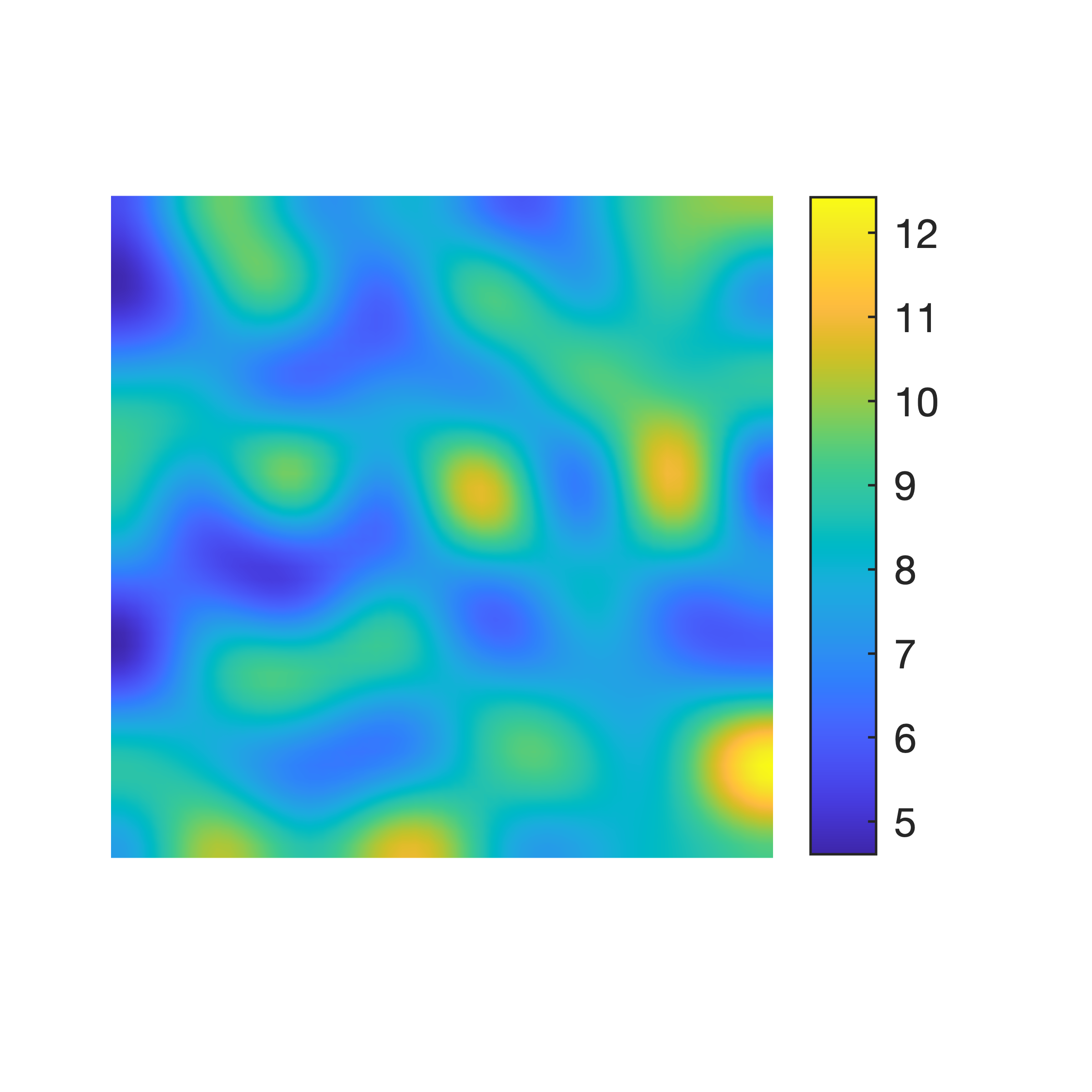}
	\includegraphics[width=0.24\textwidth,trim=0cm 1.5cm 0cm 1cm,clip]{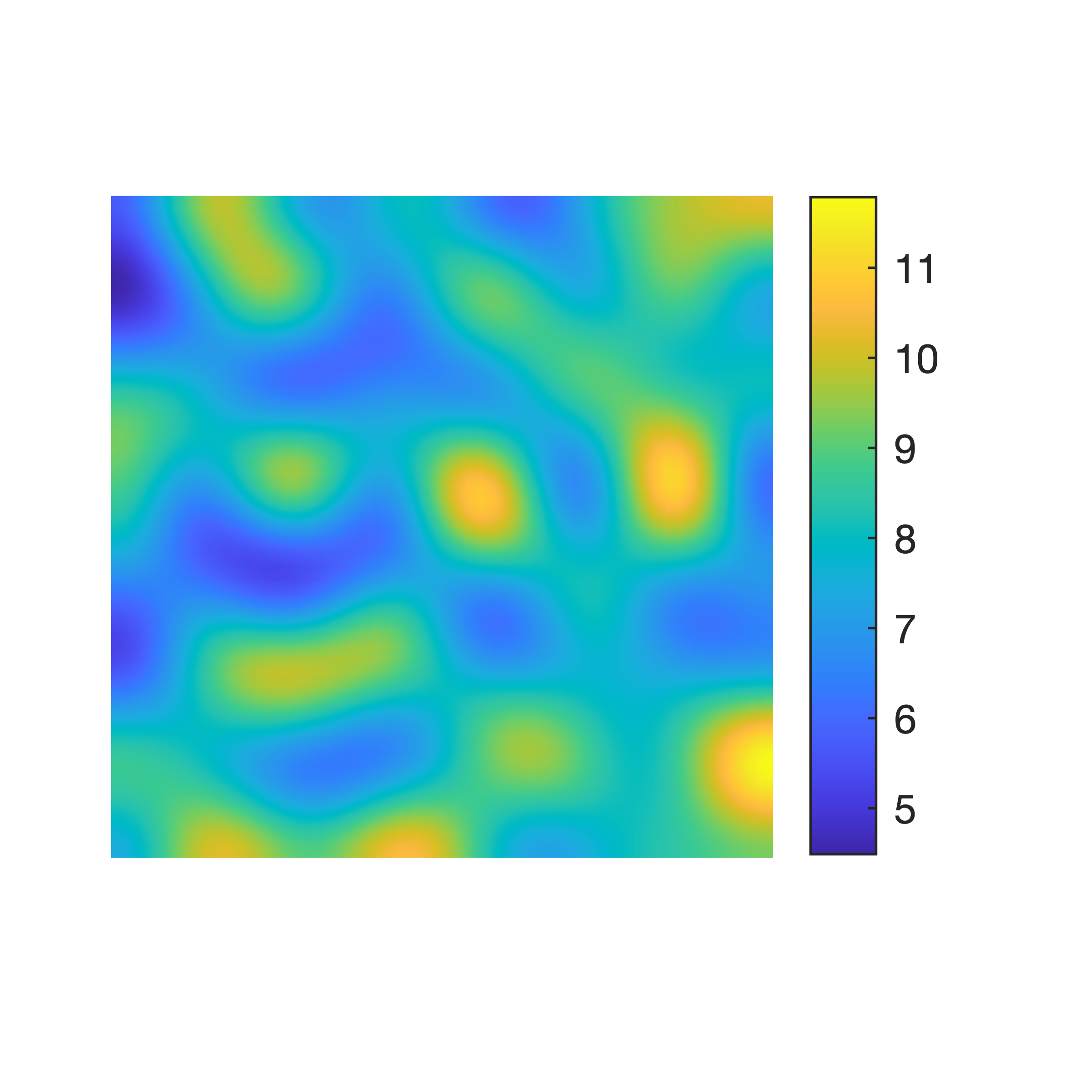} 
	\caption{\small{The reconstructed velocity images for the Fourier model (\ref{model2_M7}). Each row corresponds to the reconstruction of one velocity field. From the left to the right are the ground true velocity field, the reconstructed velocity field with $J = 1$, the reconstructed velocity field with $J = 20$, and the reconstructed velocity field with $J = 50$, respectively.}}
	\label{eight_fourier}
\end{figure}

\subsection{Warm start with learned inverse in stage (ii)}

In the previous numerical examples, the velocity models we intend to reconstruct are mostly within the distribution of velocity models used to generate the training data set. Therefore, the reconstruction given by online stage (i) is sufficient.

In the last example, we test the proposed coupling scheme on a velocity model that is outside of the training domain. Precisely, the design of the offline training stage is the same as the one in Section~\ref{numerical_ex_4modes}, namely, we focus on learning the first $5$ Fourier modes along each direction during the training. However, our goal in this example is to reconstruct the following velocity model,
\begin{equation}\label{velocity_model_outside}
m(x,z) = \left\{
\begin{aligned}
    &8.4,\quad (x,z)\in[0.22, 0.74]\times[-0.52, -0.5],\\
    &7.6, \quad \mbox{otherwise},
\end{aligned}
    \right. , \quad (x,z)\in[0,1]\times[-1,0],
\end{equation}
which is apparently outside of the training domain containing many high-frequency components. To reconstruct (\ref{velocity_model_outside}), we first implement the $J$-term truncated Neumann series approximation (\ref{EQ:Neumann}) with $J = 20$ to obtain the low-frequency part of the velocity model (\ref{velocity_model_outside}), then use it as the initial guess of a quasi-Newton algorithm based on the BFGS gradient update rule for the minimization problem in online stage (ii)~ \eqref{EQ:Online Stage II}. In addition, to recover the high-frequency components of the velocity field, except the $51$ receivers at the bottom surface in the training stage, we place another $51$ receivers at the top surface when minimizing ~\eqref{EQ:Obj Classical}, and enforce $7$ different top sources $h_i(x), i = 1,2\cdots,7$ with $h_1, h_2, h_3$ being the same as the top sources in the training stage, and
\[
h_4(x) = e^{-\frac{(x-0.7)^2}{0.01}} -e^{-\frac{(x-0.2)^2}{0.01}}, \quad
h_5(x) = e^{-\frac{(x-0.3)^2}{0.01}} -e^{-\frac{(x-0.9)^2}{0.01}}, 
\]
\[
h_6(x) = e^{-\frac{(x-0.2)^2}{0.01}} -e^{-\frac{(x-0.5)^2}{0.01}}, \quad
h_7(x) = e^{-\frac{(x-0.1)^2}{0.01}} -e^{-\frac{(x-0.6)^2}{0.01}}.
\]

Figure~\ref{outside_domian_eg} presents the surface plots of the reconstructed velocity images with both noise-free data and the data with Gaussian noise. Precisely, the top row shows the reconstructed velocity from noise-free data, the middle row displays the reconstructed velocity from the data with $10\%$ multiplication Gaussian noise, and the bottom row presents the reconstructed velocity from the data with $10\%$ additive Gaussian noise, while from the left to the right columns are the ground true velocity field, the reconstructed velocity image with $J = 1$, the reconstructed velocity image with $J = 20$ (initial guess), and the reconstructed velocity image by minimizing~\eqref{EQ:Obj Classical}, respectively. We note that adding several terms to the Neumann series approximation can lead to a relatively good reconstruction for the low-frequency components of the velocity field (\ref{velocity_model_outside}) for all cases (noise-free data and the data with Gaussian noise) by comparing the reconstruction results in column $2$ and column $3$. Then, solving an extra classical minimization problem as documented in Section~\ref{Sec:outside_domain} helps grab the high-frequency components of the velocity field, as shown in the last column.

\begin{figure}[!htb]
	\centering
	\includegraphics[width=0.24\textwidth,trim=0cm 1.5cm 0cm 1cm,clip]{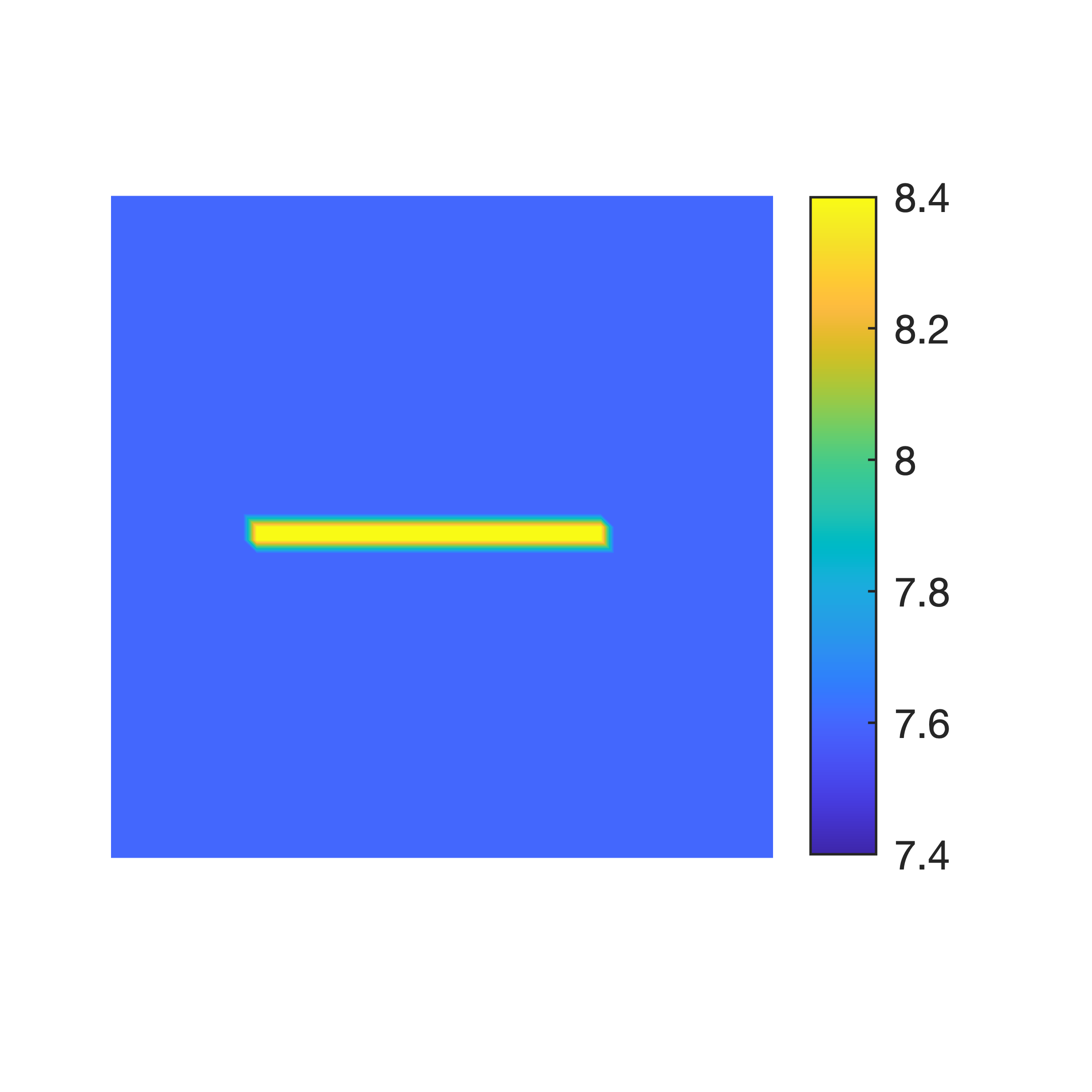}
	\includegraphics[width=0.24\textwidth,trim=0cm 1.5cm 0cm 1cm,clip]{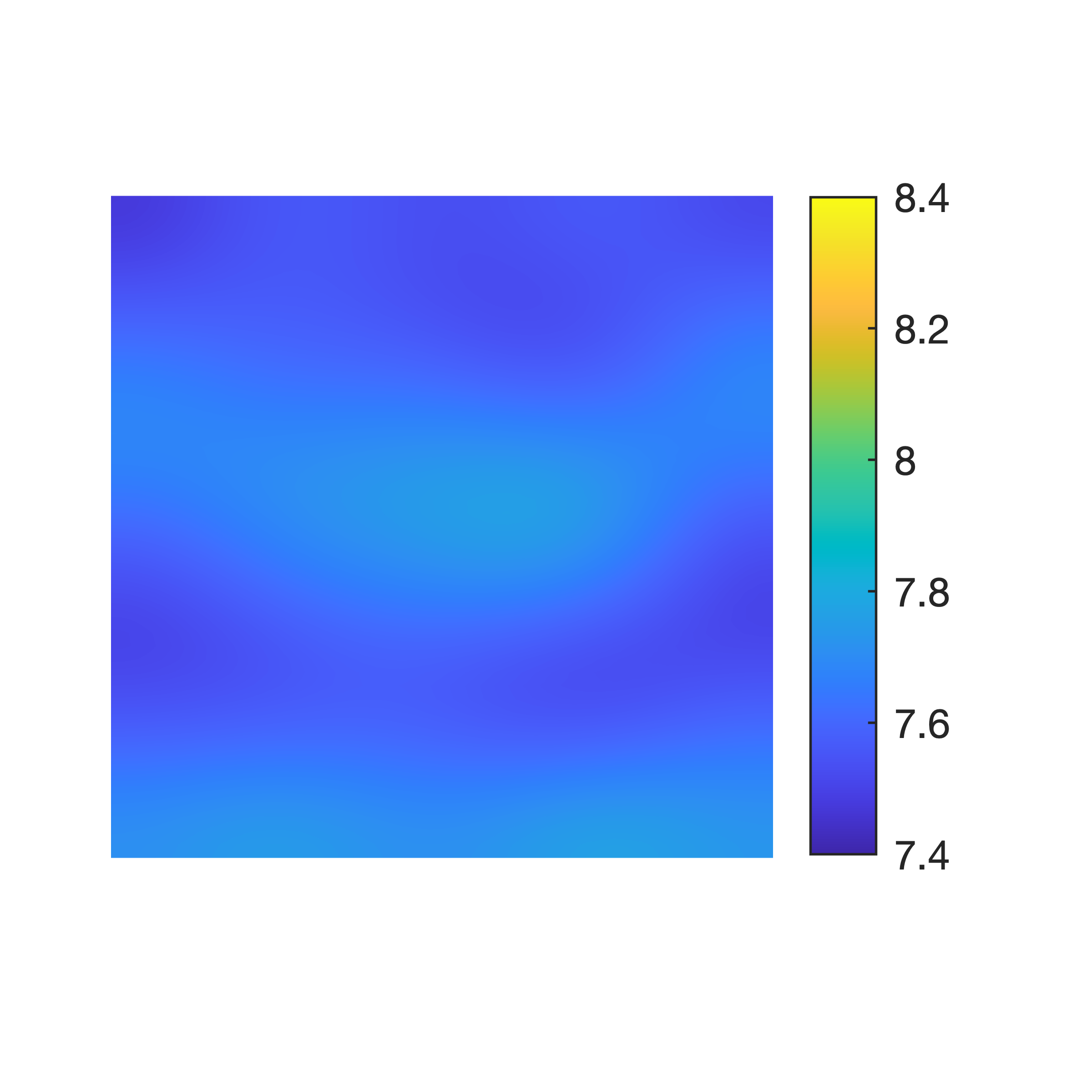} 
	\includegraphics[width=0.24\textwidth,trim=0cm 1.5cm 0cm 1cm,clip]{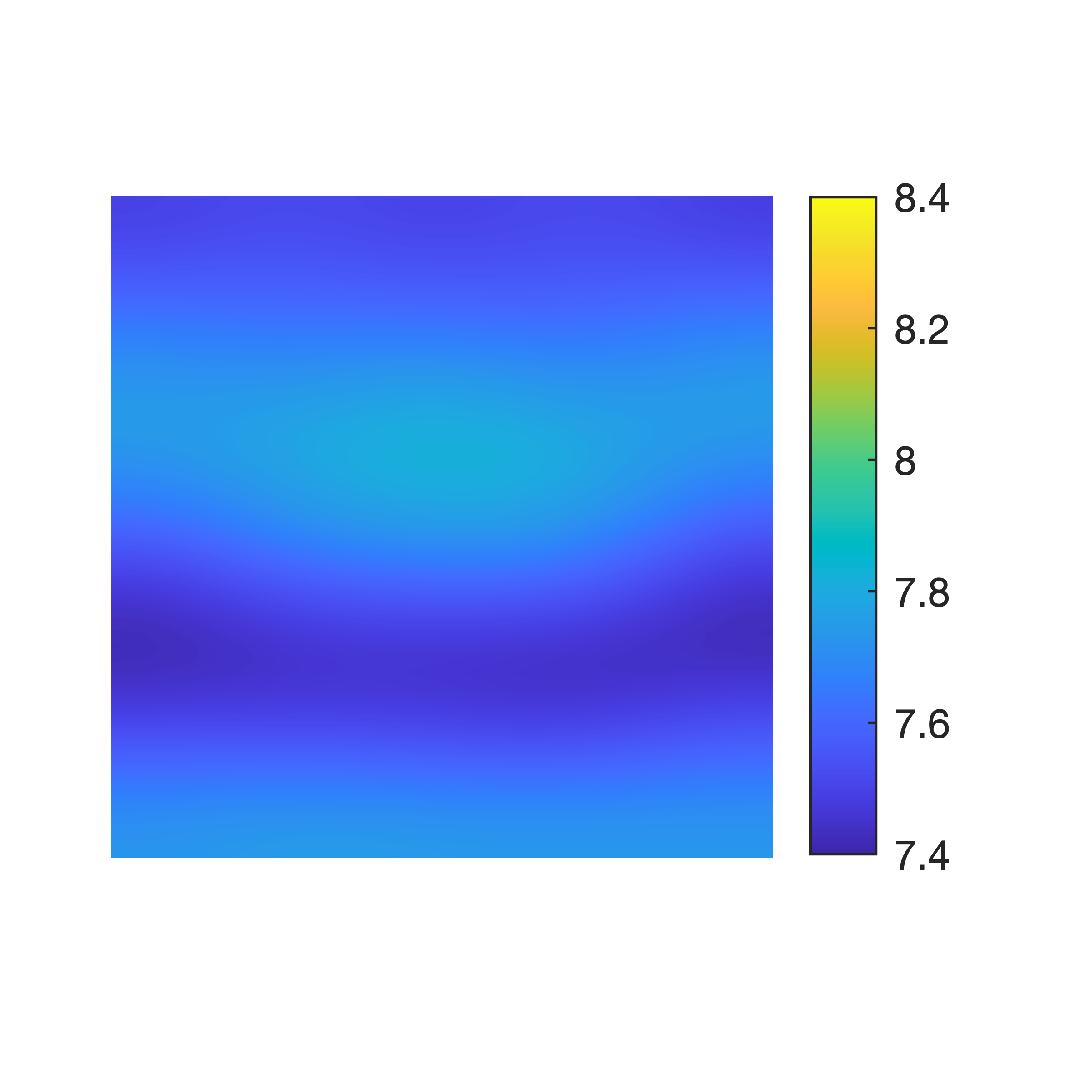}
	\includegraphics[width=0.24\textwidth,trim=0cm 1.5cm 0cm 1cm,clip]{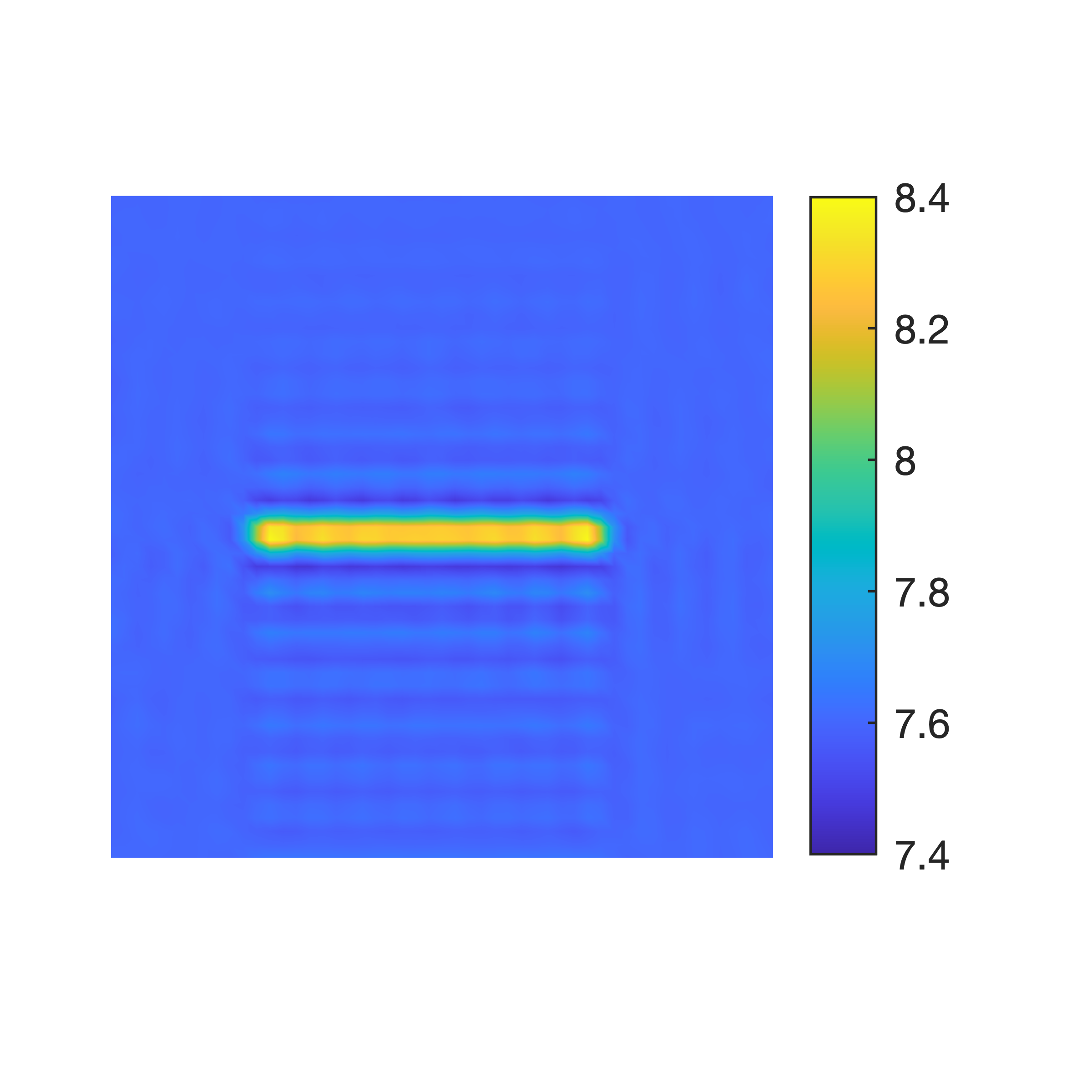} \\
	\includegraphics[width=0.24\textwidth,trim=0cm 1.5cm 0cm 1cm,clip]{example_5_true}
	\includegraphics[width=0.24\textwidth,trim=0cm 1.5cm 0cm 1cm,clip]{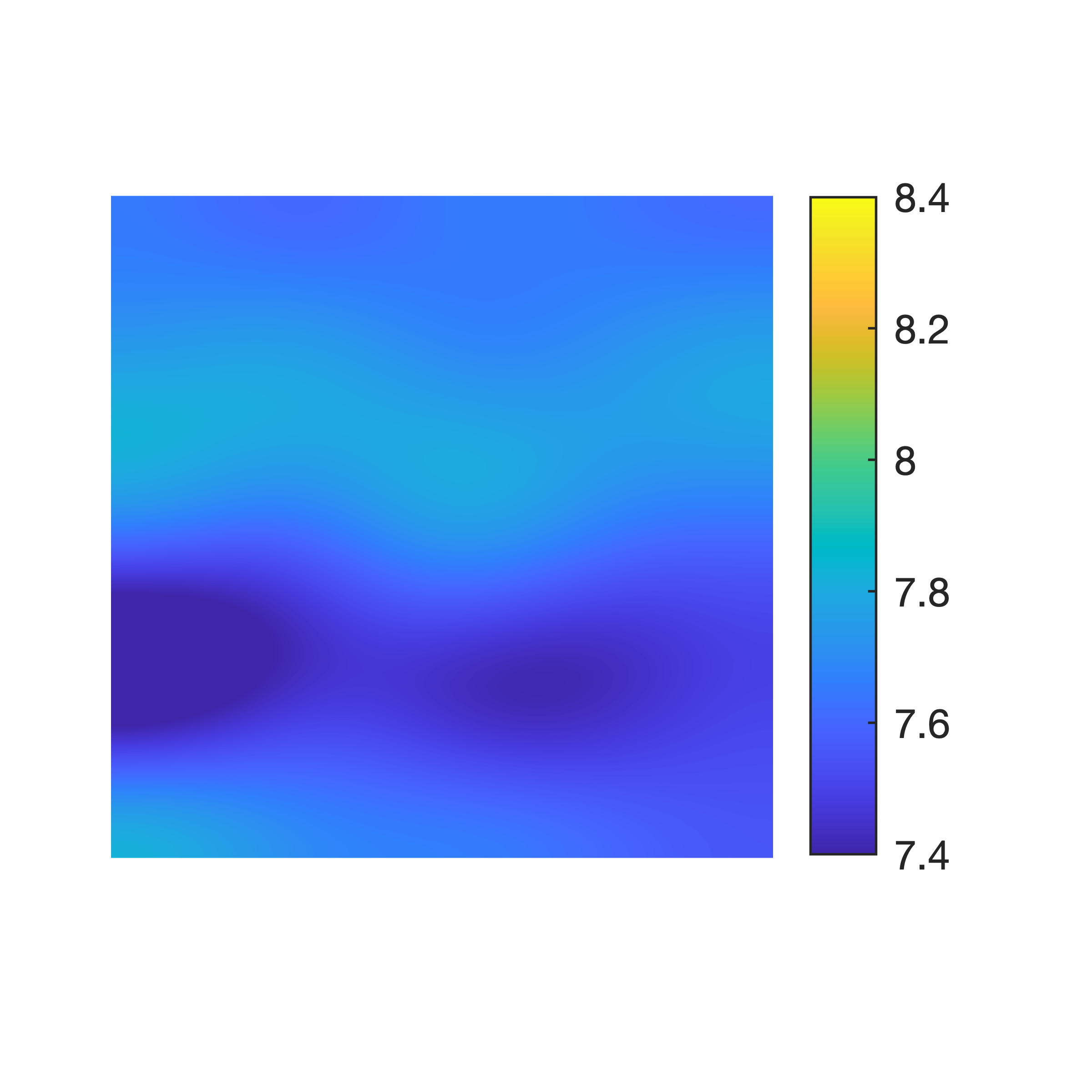} 
	\includegraphics[width=0.24\textwidth,trim=0cm 1.5cm 0cm 1cm,clip]{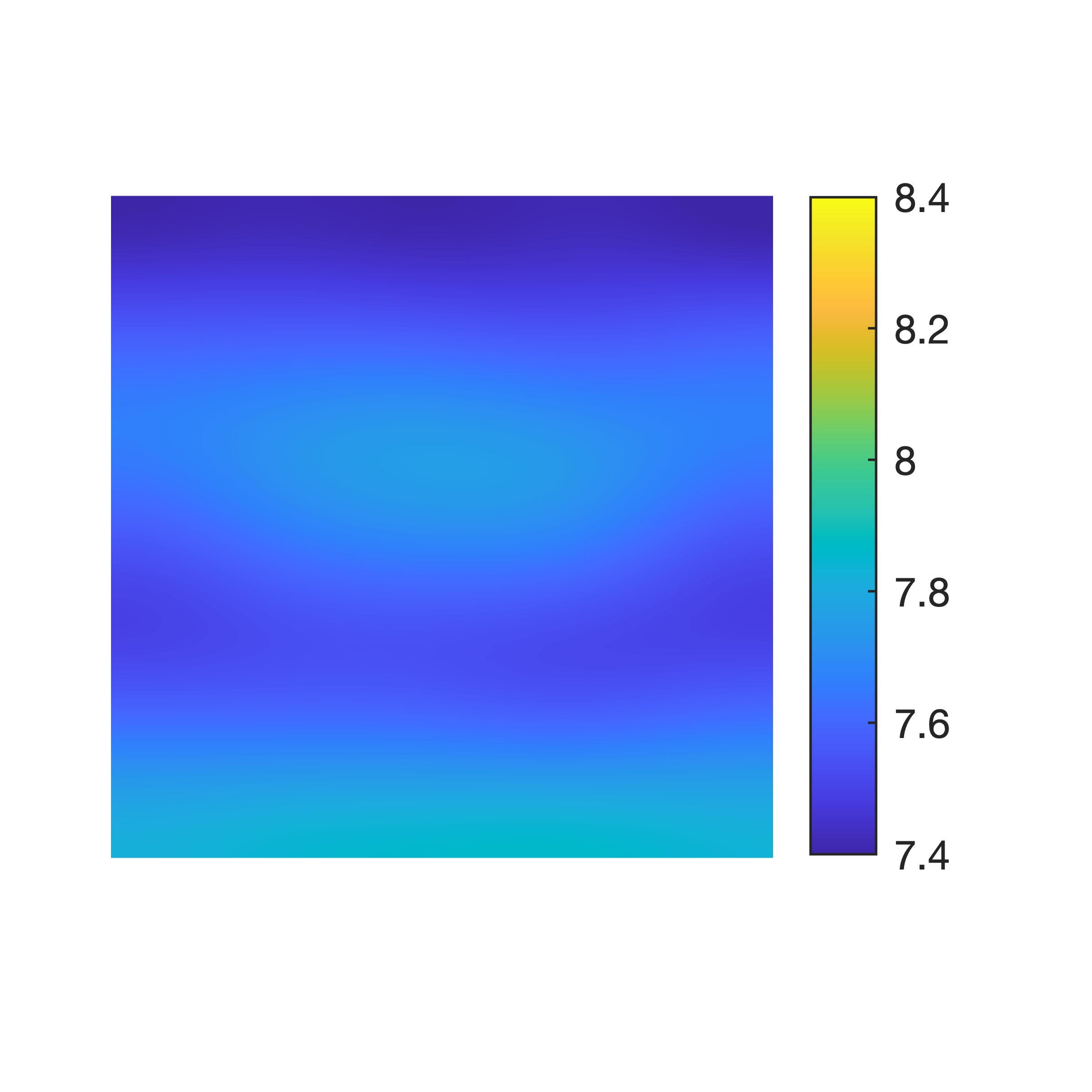}
	\includegraphics[width=0.24\textwidth,trim=0cm 1.5cm 0cm 1cm,clip]{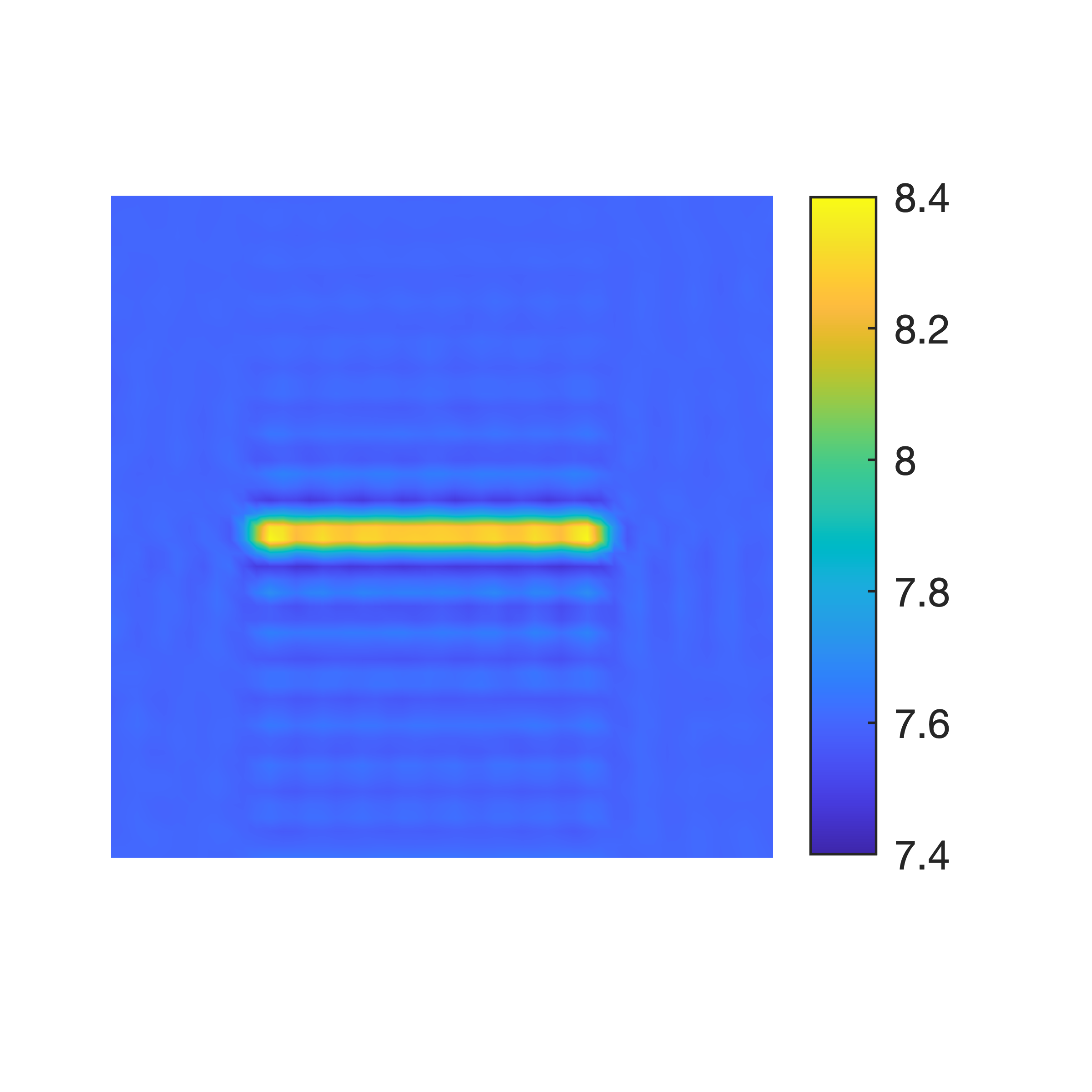} \\
	\includegraphics[width=0.24\textwidth,trim=0cm 1.5cm 0cm 1cm,clip]{example_5_true}
	\includegraphics[width=0.24\textwidth,trim=0cm 1.5cm 0cm 1cm,clip]{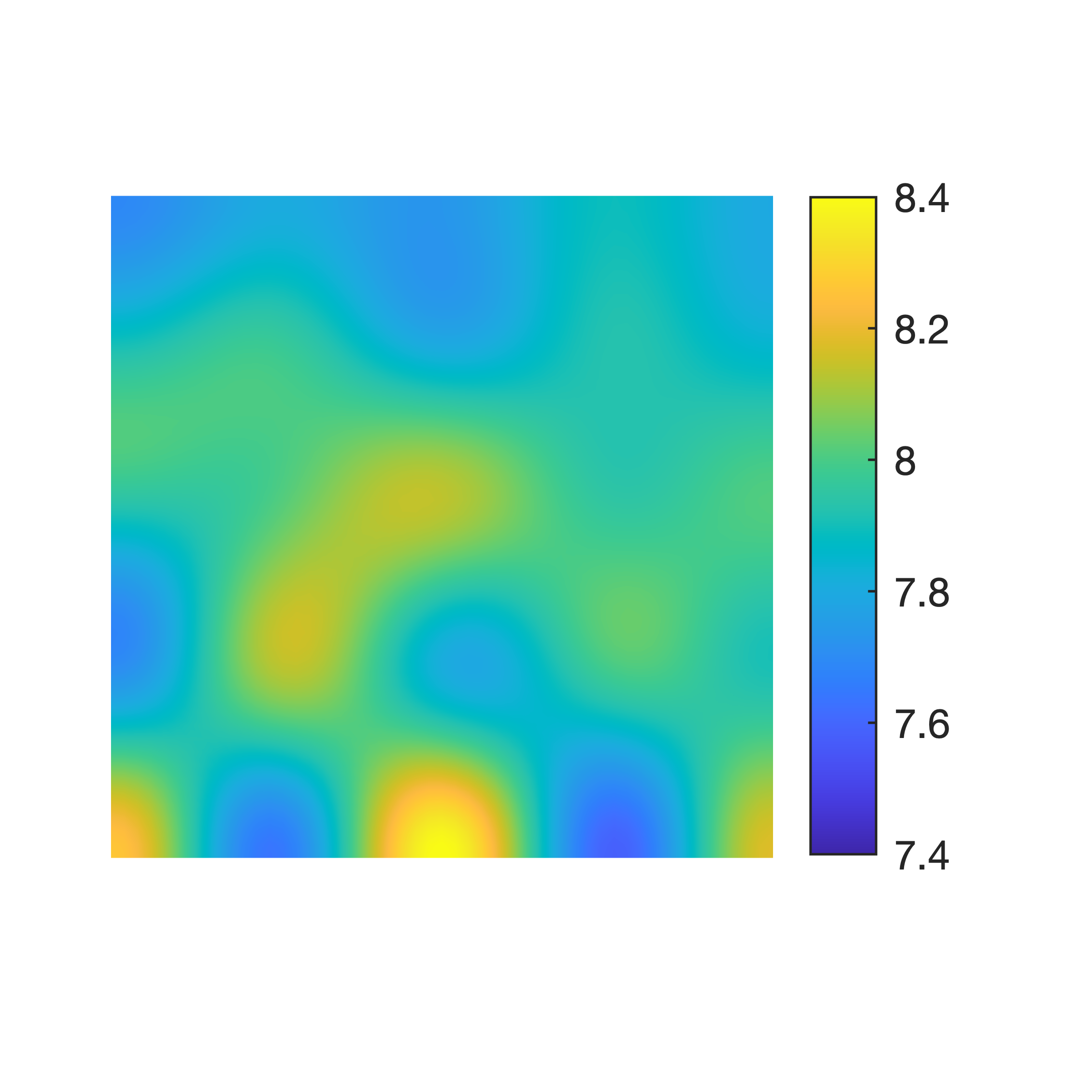} 
	\includegraphics[width=0.24\textwidth,trim=0cm 1.5cm 0cm 1cm,clip]{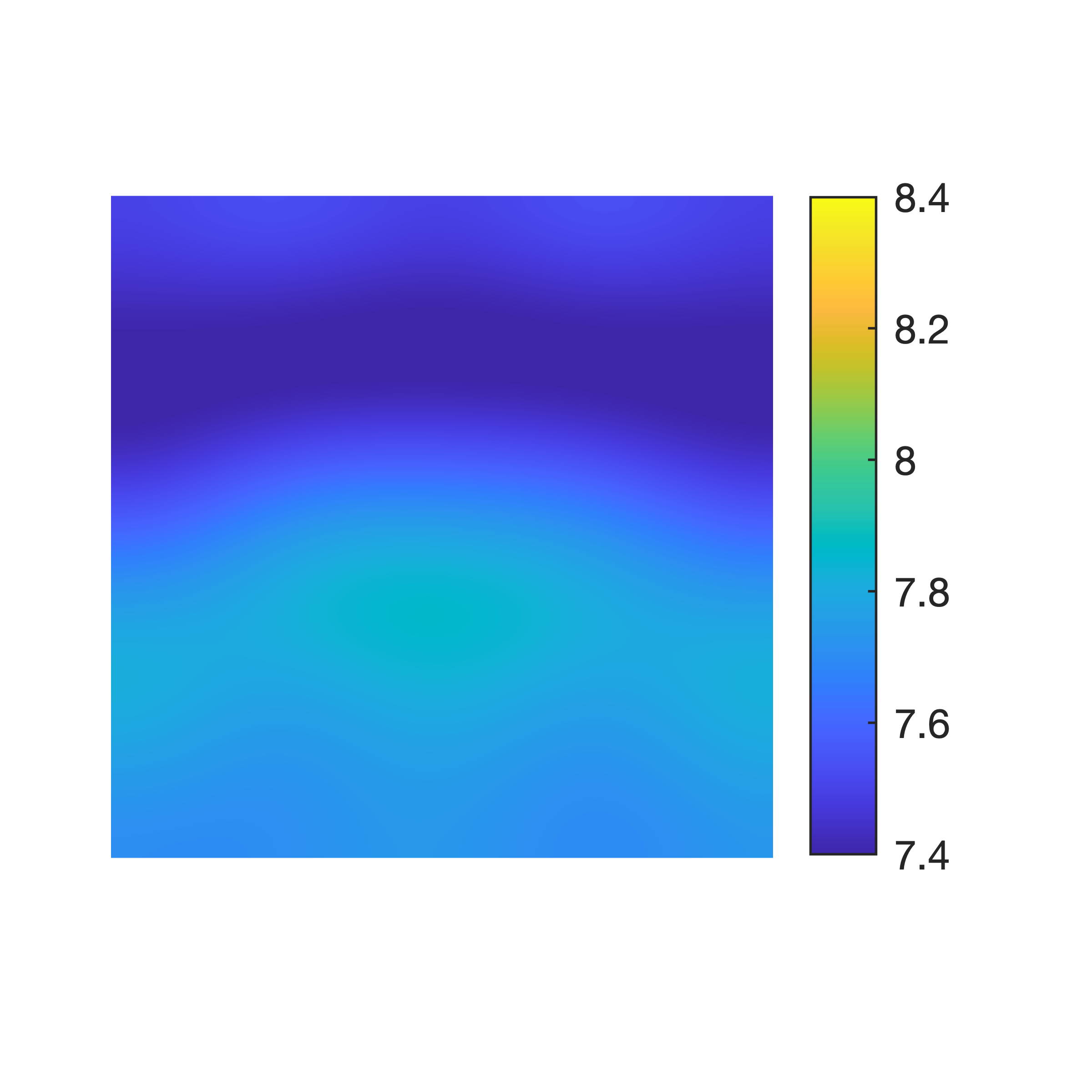}
	\includegraphics[width=0.24\textwidth,trim=0cm 1.5cm 0cm 1cm,clip]{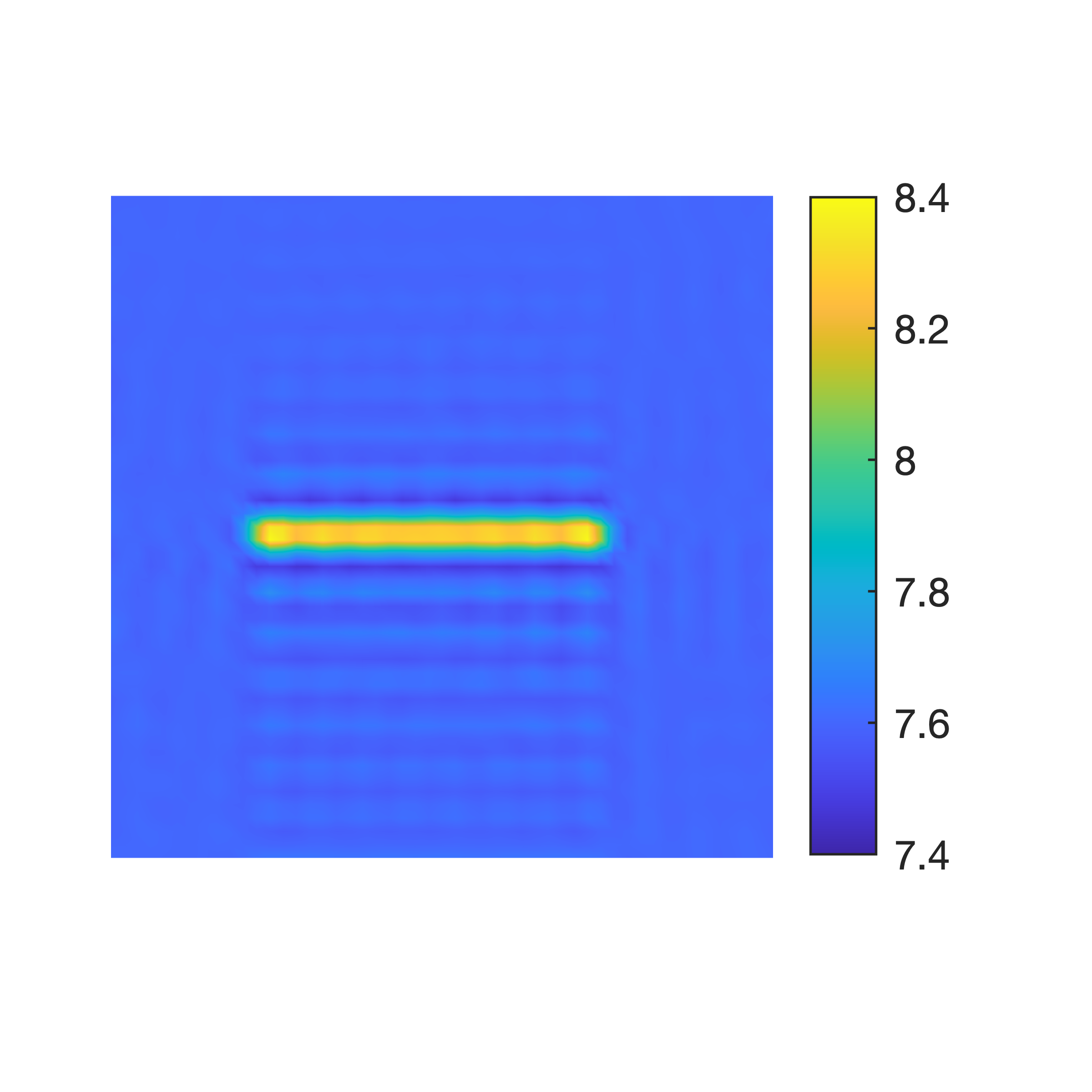} 
	\caption{\small{The reconstructed velocity images for the velocity model ~\eqref{velocity_model_outside}. From the top to the bottom are the reconstructions with noise-free data, data with $10\%$ multiplication Gaussian noise, and data with $10\%$ additive Gaussian noise, respectively. From the left to the right are the ground true velocity field, the reconstructed velocity field $\wh m_{(i)}$ with $J = 1$, the reconstructed velocity field $\wh m_{(i)}$ with $J = 20$, and the reconstructed velocity field $\wh m_{(ii)}$ by solving~\eqref{EQ:Online Stage II}, respectively.}}
	\label{outside_domian_eg}
\end{figure}

\section{Concluding remarks}
\label{SEC:Concl}

We presented in this work an offline-online computational strategy for coupling deep learning methods with classical model-based iterative reconstruction schemes for the FWI problem. The main advantage of the coupling lies in two aspects. First, the coupling requires much less rigorous training for the learning part than a purely learning-based approach. This makes the learning of the approximate inverse operator much more realistic with limited computational resources. Second, offline learning can still significantly reduce the online reconstruction with new datasets when used as a nonlinear preconditioner. The numerical simulations we performed demonstrated the feasibility of such a coupled approach.

There are many important issues in the current direction that need to be more rigorously investigated. One particular aspect is to develop a mathematical characterization of the training error in the learning process and study its impact on the reconstruction step. A second aspect is to improve the learning algorithm to learn more features in the inverse operator. As we reasoned in the paper, it is extremely challenging to learn all the details of the inverse operator. However, we believe that one could do much better than the numerical experiments in this work, where we pursue only a very small number of features in the learning process. Searching for better feature models for the velocity field as well as the time traces of the wavefield is also an important task with the potential to significantly improve the performance of the learning procedure. A third possible improvement is to find a better way to do the online reconstruction. In particular, if the training process is focused only on a limited number of Fourier modes, stage (i) of the reconstruction is not able to recover higher Fourier modes that are filtered by the trained preconditioner. It would be useful to find a better way to recover those Fourier modes than what we did in stage (ii). Adaptive training such as that in~\cite{HaReSo-arXiv25} could potentially offer some alternatives.

\section*{Acknowledgments}

We would like to thank the anonymous referees for their useful comments that helped us improve the quality of this work. This work is partially supported by the National Science Foundation through grants DMS-1937254 and DMS-2309802.

\appendix
\section{Adjoint state gradient calculation}
\label{SEC:Appendix A}

We summarize here the calculation of the Fr\'echet derivative of the objective function $\Phi(m)$ defined in~\eqref{EQ:Obj} with respect to the velocity field $m$.

Following Proposition~\ref{PROP:Frechet Diff}, the Fr\'echet differentiability of the map $\bff(m)$ with respect to $m$ is well-established under reasonable assumptions on the smoothness of the domain, the regularity of the incident wave source $h$ and the regularity of the velocity field $m$. With the assumption we have on the differentiability of the trained network $\wh\bff_{\wh\theta}^{-1}$, the Fr\'echet differentiability of $\Phi(m)$ in~\eqref{EQ:Obj} is ensured.

To simplify the notation, we denote by 
\begin{equation}\label{EQ:Residual}
	r(m):=\wh f_{\wh\theta}^{-1} (f(m))-\wh f_{\wh\theta}^{-1}(g^\delta)
\end{equation}
the data residual and $\Gamma\subset\Omega$ the subset of the boundary of the domain where acoustic data are measured. To be concrete, we take the regularization functional to be the $\cH^1$ semi-norm of the unknown $m$. We also assume that the velocity is known on the boundary of the domain so that the perturbation $\delta m|_{\partial\Omega}=0$. This assumption simplified the calculations below but is by no means essential.

Taking the derivative of $\Phi(m)$ with respect to $m$ in the direction $\delta m$, we have
\begin{equation}\label{EQ:Obj Grad1}
	\Phi'(m)[\delta m]=\int_{\Omega} r(m) \big(\wh f_{\wh\theta}^{-1}\big)'(f(m))\Big[f'(m)[\delta m]\Big] d\bx +\gamma\int_\Omega \nabla m \cdot \nabla \delta m\, d\bx\,.
\end{equation}
Let $\big(\wh f_{\wh\theta}^{-1}\big)'^{*}: L^2(\Omega)\mapsto L^2( (0, T]\times \Gamma)$ be the adjoint of the operator $\big(\wh f_{\wh\theta}^{-1}\big)'(f(m))$. Using the assumption that $\delta m|_{\partial\Omega}=0$, we can then write the above result as
\begin{equation}\label{EQ:Obj Grad2}
	\Phi'(m)[\delta m]=\int_0^T\int_{\Gamma}\big(\wh f_{\wh \theta}^{-1}\big)'^*[r(m)]  f'(m)[\delta m]\, dS(\bx) dt -\gamma\int_\Omega \Big(\Delta m\Big) \delta m\, d\bx\,.
\end{equation}
This can be written into the following form with the adjoint operator of $f'(m)$, $f'^*: L^2((0, T]\times \Gamma) \mapsto L^2(\Omega)$:
\begin{equation}\label{EQ:Obj Grad3}
	\Phi'(m)[\delta m]=\int_{\Omega} f'^*\Big[\big(\wh f_{\wh\theta}^{-1}\big)'^{*}[r(m)]\Big] \delta m\, d\bx -\gamma\int_\Omega \Big(\Delta m\Big) \delta m\, d\bx\,.
\end{equation}

The adjoint operator $f'^*$ can be found in the standard way. We document the calculation for the specific two-dimensional setup we have as follows. 

For the wave equation \eqref{EQ:Wave Equation}, we can formally differentiate $u$ with respect to $m$ to have that $u'$ solves
\begin{equation}
	\begin{array}{rcll}
	\dfrac{1}{m^2}\dfrac{\partial^2 u'}{\partial^2 t} -\Delta u' & = &  2\dfrac{\delta m}{m^3} \dfrac{\partial^2 u}{\partial^2 t}, & \mbox{in}\ \ (0, T]\times \Omega\\[1ex]
	u'(0, x, z)=\dfrac{\partial u'}{\partial t}(0, x, z)  &=& 0,& (x,z)\in(0, L)\times (-H, 0)\\
	u'(t, 0, z) & = & u'(t, L, z), &  (t, z)\in (0, T] \times (-H, 0)\\
	\dfrac{\partial u'}{\partial z}(t,x,-H) & = & 0, & (t, x)\in (0, T]\times (0, L)\\
	\dfrac{\partial u'}{\partial z}(t, x, 0) & = & 0, & (t, x)\in (0, T] \times (0, L)
	\end{array}
\end{equation}

Let us define the adjoint problem
\begin{equation}\label{EQ:Wave Adj}
	\begin{array}{rcll}
	\dfrac{1}{m^2}\dfrac{\partial^2 w}{\partial^2 t} -\Delta w & = &  0, & \mbox{in}\ \ (0, T]\times \Omega\\[1ex]
	w(T, x, z)=\dfrac{\partial w}{\partial t}(T, x, z)  &=& 0,& (x,z)\in(0, L)\times (-H, 0)\\
	w(t, 0, z)  =  w(t, L, z) & = & 0, &  (t, z)\in (0, T] \times (-H, 0)\\
	\dfrac{\partial w}{\partial x}(t, 0, z) + \dfrac{\partial w}{\partial x}(t, L, z) & = & 0, &  (t, z)\in (0, T] \times (-H, 0)\\
	\dfrac{\partial w}{\partial z}(t,x,-H) & = & 0, & (t, x)\in (0, T]\times (0, L)\\
	\dfrac{\partial w}{\partial z}(t, x, 0) & = & \big(\wh f_{\wh\theta}^{-1}\big)'^*[r(m)], & (t, x)\in (0, T] \times (0, L)
	\end{array}
\end{equation}

We can then multiply the equation for $u'$ by $w$ and the equation for $w$ by $u'$ and use integration by part to show that
\begin{equation}\label{EQ:Obj Grad4}
	\Phi'(m)[\delta m]=-\int_{\Omega} \dfrac{2}{m^3} \Big(\int_0^T \dfrac{\partial w}{\partial t} \dfrac{\partial u}{\partial t} d t\Big) \delta m\, d\bx-\gamma\int_\Omega \Big(\Delta m\Big) \delta m\, d\bx\,.
\end{equation}

When the data in the inversion are collected from $N_s$ different incoming sources $\{h_s\}_{s=1}^{N_s}$, the forward map $\bff(m)$ and the data ${\bf g}^\delta$ defined in (\ref{EQ:Nonl IP}). Let $u_s$ ($1\le s\le N_s$) be solution to~\eqref{EQ:Wave Equation} with source $h_s$, and $w_s$ be the solution to the adjoint equation~\eqref{EQ:Wave Adj} with the $s$-th component of $\big(\wh \bff_{\wh\theta}^{-1}\big)'^{*}[{\bf r}(m)]$, here
\begin{equation}\label{EQ: residual_full}
{\bf r}(m) = \wh \bff_{\wh \theta}^{-1} (\bff (m)) - \wh \bff_{\wh \theta}^{-1}({\bf g}^\delta),
\end{equation}
then derivative of $\Phi(m)$ can be computed as
\begin{equation}\label{EQ:Obj Grad5}
	\Phi'(m)[\delta m]=-\int_{\Omega} \dfrac{2}{m^3} \Big(\sum_{s=1}^{N_s}\int_0^T \dfrac{\partial w_s}{\partial t} \dfrac{\partial u_s}{\partial t} d t\Big) \delta m\, d\bx-\gamma\int_\Omega \Big(\Delta m\Big) \delta m\, d\bx\,.
\end{equation}

\begin{algorithm}
\caption{Gradient Calculation with Adjoint State}
\label{ALG:Adjoing-Gradient}
\begin{algorithmic}[1]
\For{$s=1$ to $N_s$}
	\State Solve~\eqref{EQ:Wave Equation} with $h_s$ for $u_s$
	\State Evaluate the $f(m; h_s)$ component of $\bff(m)$
\EndFor
\State Evaluate ${\bf r}(m)$ according to~\eqref{EQ: residual_full} with the network $\wh \bff_{\wh \theta}^{-1}$
\State Evaluate $\big(\wh \bff_{\wh \theta}^{-1}\big)'^{*}[{\bf r}(m)]$ with the neural network
\For{$s=1$ to $N_s$}
	\State Solve~\eqref{EQ:Wave Adj} with the $s$-th component of $\big(\wh \bff_{\wh \theta}^{-1}\big)'^{*}[{\bf r}(m)]$ as the source term for $w_s$
\EndFor
\State Evaluate $\Phi'(m)$ according to~\eqref{EQ:Obj Grad5}
\end{algorithmic}
\end{algorithm}
The computational procedure is summarized in Algorithm~\ref{ALG:Adjoing-Gradient}. The main difference between the calculation here and the adjoint calculation for a standard FWI gradient calculation is that we need to use the network $\wh \bff_{\wh \theta}^{-1}$  to backpropagate the data into the velocity field in Line $5$ of Algorithm~\ref{ALG:Adjoing-Gradient} to compute the residual, and then use the adjoint of the network operator (transpose of the gradient in the discrete case), $\big(\wh \bff_{\wh \theta}^{-1}\big)'^{*}$, to map the residual ${\bf r}(m)$ to the source of the adjoint wave equation in Line $6$.

\section{Inversion with truncated Neumann series}
\label{SEC:Appendix B}

The truncated Neumann series reconstruction~\eqref{EQ:Neumann} can be implemented with only the forward wave simulation and the learned neural network (without the need for the gradient of the learned operator). Let us define
\[
m_0:= \wh \bff_{\wh \theta}^{-1}({\bf g}^\delta), \quad R_J:= \sum\limits_{j=0}^{J-1} K^j\big(m_0)\,,
\]
with $K$ defined in~\eqref{EQ:Neumann}. We can then verify that
\begin{equation}\label{iteration formula}
    R_{J+1}(m_0) = (I + K\sum_{j=0}^{J-1} K^j)\big(m_0)=m_0 + KR_{J}(m_0) = m_0+ R_{J}(m_0) - \wh \bff_{\wh \theta}^{-1}(\bff( R_{J}(m_0)))\,.
\end{equation}
This leads to the computational procedure summarized in Algorithm~\ref{ALG:Neumann}. 
\begin{algorithm}[!htb]
\caption{Reconstruction with $J$-Term Truncated Neumann Series}
\label{ALG:Neumann}
\begin{algorithmic}[1]
\State Evaluate $m_0:=\wh \bff_{\wh \theta}^{-1}({\bf g}^\delta)$ with the learned neural network
\State Set $m\leftarrow m_0$;
\For{$j=1$ to $j=J-1$}
	
	\For{$s=1$ to $N_s$}
		\State Solve~\eqref{EQ:Wave Equation} with $(m, h_s)$ for $u_s$
		\State Evaluate the $f(m; h_s)$ component of $\bff(m)$
	\EndFor
	\State Update $m\leftarrow m_0 +m- \wh \bff_{\wh \theta}^{-1}(\bff(m))$

\EndFor
\end{algorithmic}
\end{algorithm}
The main difference between the calculation here and the adjoint calculation for a standard FWI gradient calculation is that we only need to evaluate the network $\wh \bff_{\wh \theta}^{-1}$  to project the data back into the velocity field to compute the residual $m_0 - \wh \bff_{\wh \theta}^{-1}(\bff(m))$, and update the current result $m$.

\section{Network structure and training}
\label{SEC:Appendix C}

For the sake of the reproducibility of our research, we provide here the structures of the encoder, decoder, and predictor networks we used in the encoder-decoder-predictor training framework described in Section~\ref{SUBSEC:Network}; see Figure~\ref{FIG:Encoder-Decoder-Predictor}. 
\begin{figure}[!htb]
	\centering
	\includegraphics[width=0.3\textwidth,trim=0cm 0.2cm 0cm 0cm,clip]{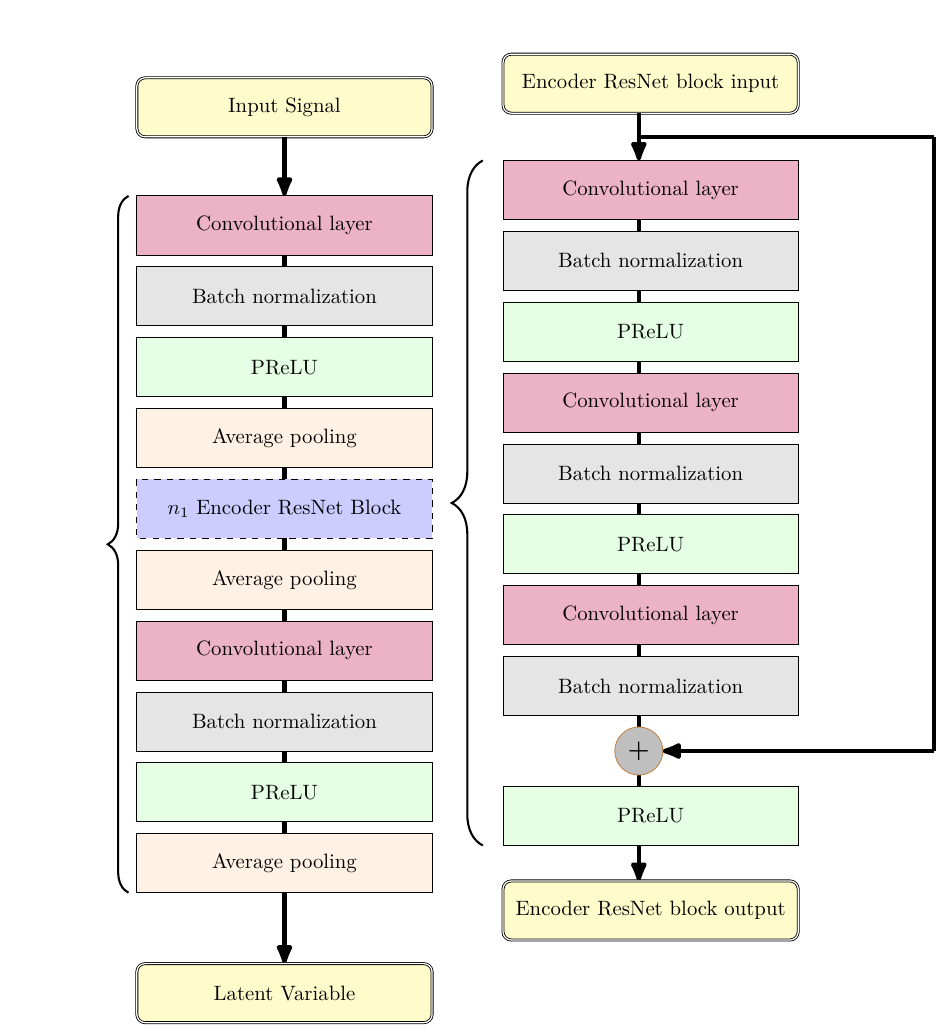}\hskip 1mm
	\includegraphics[width=0.3\textwidth,trim=0cm 0.2cm 0cm 0cm,clip]{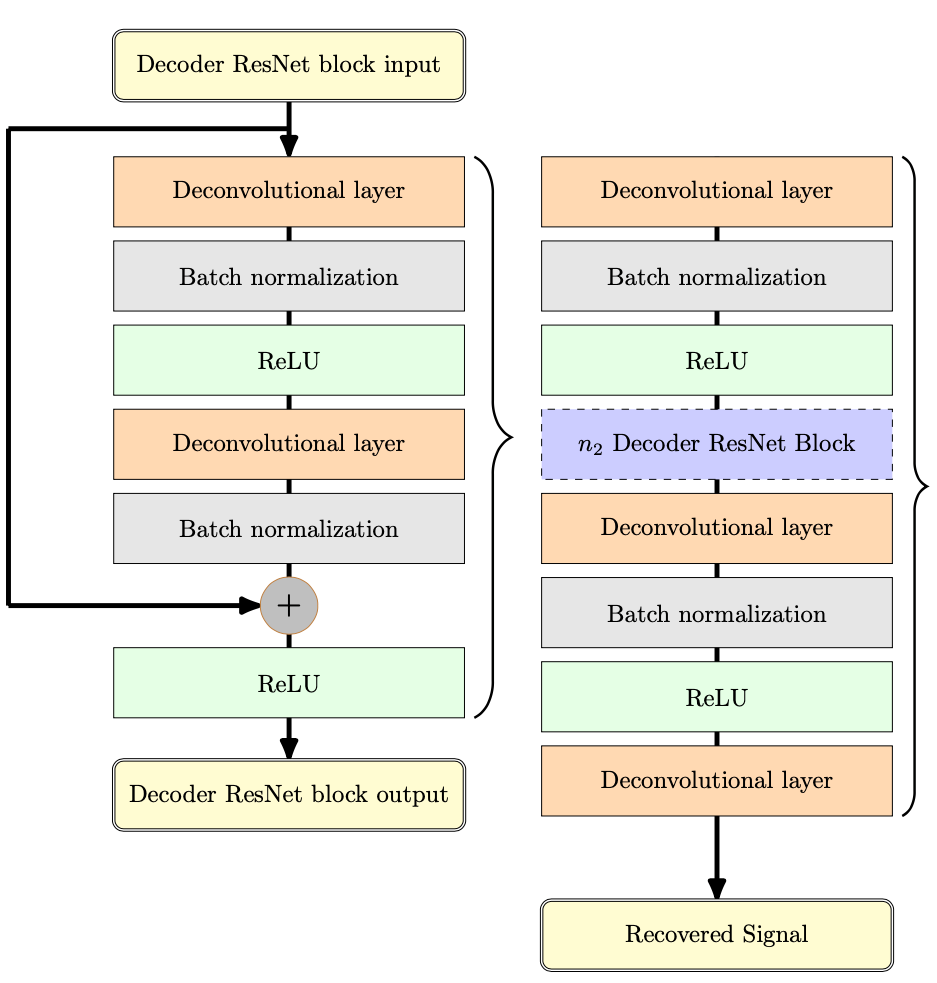}\hskip 9mm
	\includegraphics[width=0.3\textwidth,trim=0cm 0.2cm 0cm 0cm,clip]{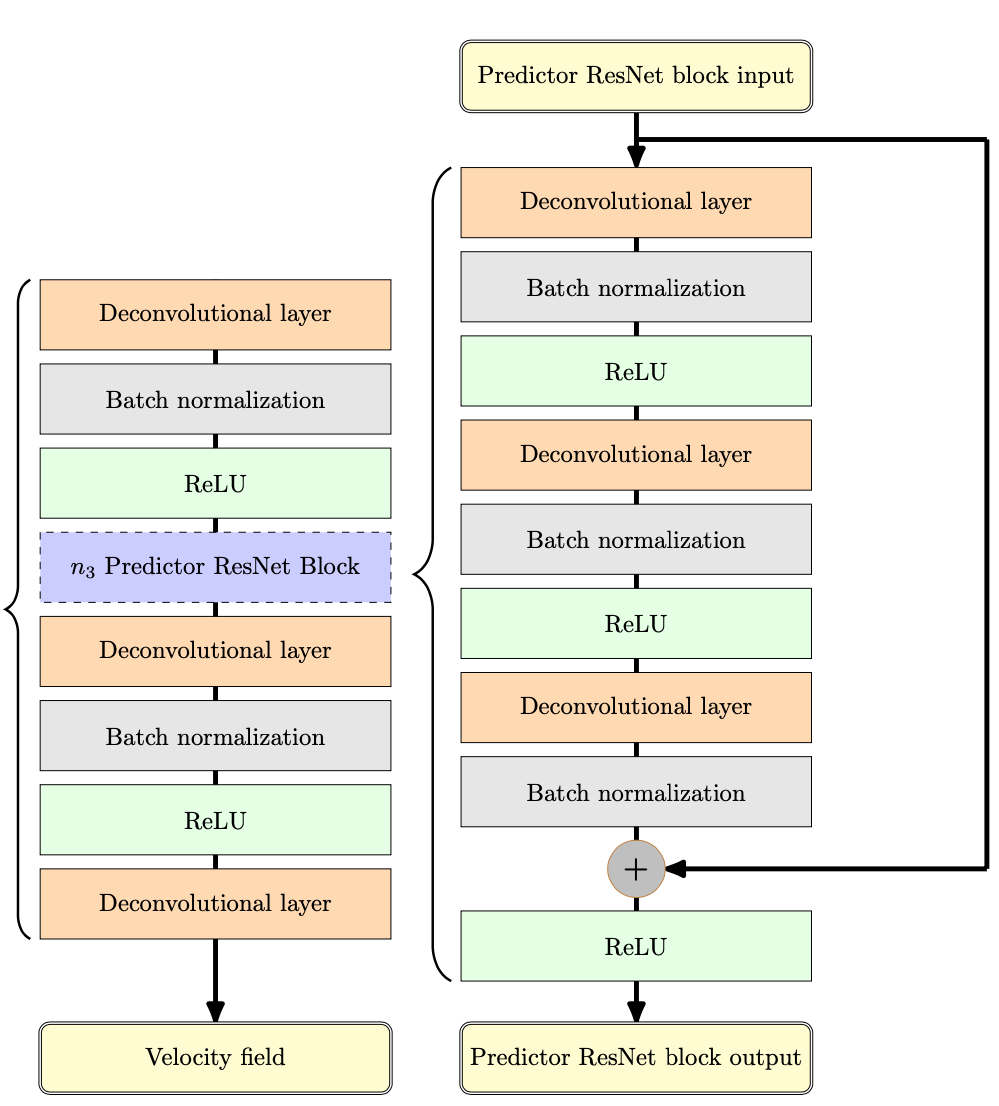}
	\put(-455,70){\rotatebox{90}{\tiny Encoder}}
	\put(-165,80){\rotatebox{270}{\tiny Decoder}}
	\put(-145,50){\rotatebox{90}{\tiny Predictor}}
	\caption{Network structures of the encoder, decoder and predictor networks.}
	\label{FIG:Encoder-Decoder-Predictor}
\end{figure}

The different layers of the networks are all standard, as indicated by their names. In our implementation, the input of the neural network is a $N_t\times N_d \times N_s$ tensor representing the solution of ~\eqref{EQ:Wave Equation} for $N_s$ sources, at $N_d$ detector points $\{\bx_d\}_{d=1}^{N_d}$, and on $N_t$ time instances $\{t_j\}_{j=1}^{N_t}$: $u(t_i, \bx_j; h_s)$, $i = 1,\cdots, N_t$, $j = 1,\cdots, N_d$, and $s=1, \cdots, N_s$. The network outputs the recovered input (from the decoder and the reconstructed velocity field from the predictor. When the output velocity field is represented with the Fourier basis, the output of the predictor is an $M\times M$ matrix whose $ij$-element is $\fm(k_i, k_j)$ ($0\le k_i, k_j\le M$). 

Besides the sizes of the network input (that is, the input of the encoder) and the network output (that is, the output of the predictor), the key parameters of the overall network are (i) the size of the latent variables, and (ii) the number of ResNet blocks in each of the sub-networks ($n_1$, $n_2$ and $n_3$). 
In our implementation, we tested the network structure with different numbers of ResNet blocks. The training results are not sensitive to the selection of such numbers (which controls the size of the overall network). In the numerical simulations we presented in the paper, we use $n_1=10$, $n_2=5$, and $n_3=10$. \RED{The computational code we used for the numerical simulations in this paper, implemented in PyTorch,} is deposited at
\href{https://github.com/wending1/FWI_Deep_Learning}{https://github.com/wending1/FWI\textunderscore Deep\textunderscore Learning}.

The network training is achieved with the Adam optimizer~\cite{KiBa-arXiv14}. The learning rate is initially set to be  $5\times 10^{-4}$, and decays by a factor of $1.2$ for every $5$ epoch. The batch size is chosen to be $128$. We stop the training after $50$ epochs.



\end{document}